\documentclass[oneside,english]{amsart}
\usepackage[T1]{fontenc}
\usepackage[latin9]{inputenc}
\usepackage{mathrsfs}
\usepackage{mathtools}
\usepackage{amsbsy}
\usepackage{amstext}
\usepackage{amsthm}
\usepackage{amssymb}
\usepackage{stmaryrd}
\PassOptionsToPackage{normalem}{ulem}
\usepackage{ulem}

\makeatletter
\numberwithin{equation}{section}
\numberwithin{figure}{section}
\theoremstyle{plain}
\newtheorem{thm}{\protect\theoremname}[section]
\theoremstyle{plain}
\newtheorem{lem}[thm]{\protect\lemmaname}
\theoremstyle{definition}
\newtheorem{defn}[thm]{\protect\definitionname}
\theoremstyle{plain}
\newtheorem{prop}[thm]{\protect\propositionname}
\theoremstyle{plain}
\newtheorem{cor}[thm]{\protect\corollaryname}
\theoremstyle{remark}
\newtheorem{rem}[thm]{\protect\remarkname}
\theoremstyle{plain}
\newtheorem{assumption}[thm]{\protect\assumptionname}

\makeatother

\usepackage{babel}
\providecommand{\assumptionname}{Assumption}
\providecommand{\corollaryname}{Corollary}
\providecommand{\definitionname}{Definition}
\providecommand{\lemmaname}{Lemma}
\providecommand{\propositionname}{Proposition}
\providecommand{\remarkname}{Remark}
\providecommand{\theoremname}{Theorem}

\begin{document}
\title{Stability and singular set of the two-convex level set flow}
\author{Siao-Hao Guo}
\begin{abstract}
\thanks{The research was partially supported by the grant 112-2115-M-002-016-MY3
of the National Science and Technology Council of Taiwan. \medskip{}
} 

The level set flow of a mean-convex closed hypersurface is stable
off singularities, in the sense that the level set flow of the perturbed
hypersurface would be close in the smooth topology to the original
flow wherever the latter is regular. To study the behavior near singularities,
we further assume that the initial hypersurface is two-convex and
that the flow has finitely many singular times. In this case, the
singular set of the flow would have finitely many connected components,
each of which is either a point or a compact $C^{1}$ embedded curve.
Then under additional conditions, we show that near each connected
component of the singular set of the original flow, the perturbed
flow would have ``the same type'' of singular set as that of the
singular component.
\end{abstract}

\maketitle
\tableofcontents{}

\section{\label{introduction}Introduction}

Let $\Sigma_{0}$ be a closed connected (smoothly embedded) hypersurface
in $\mathbb{R}^{n}$ with $n\geq3$. There exists a unique number
$T_{1}\in\left(0,\infty\right)$, called the first singular time,
so that the mean curvature flow (MCF) $\left\{ \Sigma_{t}\right\} _{t\in\left[0,T_{1}\right)}$
starting from $\Sigma_{0}$ at time $0$ cannot be smoothly extended
any further; in fact, the second fundamental form of $\Sigma_{t}$
blows up as $t\nearrow T_{1}$ (see \cite{H1}). The \textbf{level
set flow} (LSF) $\left\{ \Sigma_{t}\right\} _{t\in\left[0,\infty\right)}$
starting from $\Sigma_{0}$ at time $0$ is a continuation of the
MCF past the first singular time via the ``level set method'' (cf.
\cite{OS}, \cite{CGG}, \cite{ES}). Such a flow is uniquely determined
by $\Sigma_{0}$ and agrees with the MCF for $t\in\left[0,T_{1}\right)$.
The flow would vanish after a finite time; the unique number $T_{ext}\in\left(0,\infty\right)$
is called the ``time of extinction'' provided the flow does not
vanish until after time $T_{ext}$. 

In this paper we assume that $\Sigma_{0}$ is \textbf{mean-convex},
namely, its mean curvature is positive with respect to the inward
unit normal vector. In this case, the LSF $\left\{ \Sigma_{t}\right\} $
would stay in $\Omega_{0}$ for $t>0$, where $\Omega_{0}$ is the
open connected set bounded by $\Sigma_{0}$ (see the Jordan-Brouwer
separation theorem in \cite{GP}), and it would sweep out $\Omega_{0}$
monotonically. Specifically, for each point $p\in\Omega_{0}$, the
flow would pass through $p$ at some moment and then it would never
return. Let us denote by $u\left(p\right)$ the unique time when the
flow $\left\{ \Sigma_{t}\right\} $ passes through $p$.\footnote{When $p\in\partial\Omega_{0}=\Sigma_{0}$, set $u\left(p\right)=0$.}
As such, the LSF can be described as the level sets of the function
$u$ as follows:
\begin{equation}
\Sigma_{t}=\left\{ x\in\bar{\Omega}_{0}:u\left(x\right)=t\right\} ,\,\,\,t\geq0.\label{level sets and flow}
\end{equation}
Such a function $u$, called the \textbf{arrival time function} of
the flow, is actually Lipschitz continuous on $\bar{\Omega}_{0}$;
moreover, it is the unique viscosity solution\footnote{See \cite{ES} for the precise definition of viscosity solutions to
equation (\ref{level set flow}).} to the following Dirichlet problem (cf. \cite{ES}):
\begin{equation}
-\left(\mathrm{I}-\frac{\nabla u}{\left|\nabla u\right|}\varotimes\frac{\nabla u}{\left|\nabla u\right|}\right)\cdotp\nabla^{2}u=1\quad\textrm{in}\,\,\,\Omega_{0},\label{level set flow}
\end{equation}
\[
u=0\quad\textrm{on}\,\,\,\partial\Omega_{0}=\Sigma_{0}.
\]
In this case, the flow is called a \textbf{mean-convex LSF}. Note
that the superlevel set 
\begin{equation}
\Omega_{t}\coloneqq\left\{ x\in\Omega_{0}:u\left(x\right)>t\right\} ,\quad t>0\label{super level set}
\end{equation}
is contracting in the sense that 
\begin{equation}
t_{1}<t_{2}\,\Rightarrow\,\Omega_{t_{1}}\supset\bar{\Omega}_{t_{2}}.\label{shrinking}
\end{equation}
Note also that 
\[
T_{ext}=\max_{\Omega_{0}}\,u.
\]

A point $p\in\Omega_{0}$ is called a \textbf{regular point} of $\left\{ \Sigma_{t}\right\} $
provided the flow is locally smooth and monotonic\footnote{That is, $\partial_{t}x\cdot N>0$ for some choice of locally smooth
unit normal vector field $N$ of $\left\{ \Sigma_{t}\right\} $.} near the point $p$ for $t$ close to $u\left(p\right)$; by (\ref{level sets and flow})
and the implicit function theorem, this is equivalent to saying that
the arrival time function $u$ is locally smooth near $p$ with $\nabla u\left(p\right)\neq0$.
In that case, $u$ would satisfy (\ref{level set flow}) in the classical
sense near $p$ and $\left\{ \Sigma_{t}\right\} $ would be a MCF
near $p$ for $t$ close to $u\left(p\right)$, whose mean curvature
with respect to the unit normal $\frac{\nabla u}{\left|\nabla u\right|}$
would be
\begin{equation}
H=-\nabla\cdot\frac{\nabla u}{\left|\nabla u\right|}=\frac{-1}{\left|\nabla u\right|}\left(\mathrm{I}-\frac{\nabla u}{\left|\nabla u\right|}\varotimes\frac{\nabla u}{\left|\nabla u\right|}\right)\cdotp\nabla^{2}u=\frac{1}{\left|\nabla u\right|};\label{mean curvature}
\end{equation}
in particular, $\left\{ \Sigma_{t}\right\} $ is mean-convex with
respect to $N=\frac{\nabla u}{\left|\nabla u\right|}$.

On the other hand, if a point $p\in\Omega_{0}$ is not a regular point
of $\left\{ \Sigma_{t}\right\} $, it would be called a \textbf{singular
point} of the flow. The set of all singular points of $\left\{ \Sigma_{t}\right\} $,
denoted by $\mathcal{S}$, is of Hausdorff dimension at most $n-2$
(cf. \cite{W1}). As $\left\{ \mathcal{H}^{n-1}\lfloor\Sigma_{t}\right\} $
is a Brakke flow (cf. \cite{I1}, \cite{W1}), we may consider the
tangent flows at the singular points (cf. \cite{I2}). It turns out
that the tangent flows at each singular point is unique; it is either
a shrinking sphere or a shrinking $k$-cylinder\footnote{That is, up to a rotation, the self-shrinking of $S_{\sqrt{2\left(n-k-1\right)}}^{n-k-1}\times\mathbb{R}^{k}$,
where $k\in\left\{ 1,\cdots,n-2\right\} $.} (cf. \cite{W2}, \cite{W4}, \cite{CM2}). Following \cite{G2},
singular points are called \textbf{round points} and \textbf{$k$-cylindrical
points} if the tangent flows are shrinking spheres and shrinking $k$-cylinders,
respectively. 

By \cite{CM4}, at every singular point of $\left\{ \Sigma_{t}\right\} $
the arrival time function $u$ is twice differentiable and $\nabla u$
vanishes. As a corollary, singular points of $\left\{ \Sigma_{t}\right\} $
are precisely the critical points of $u$, namely,
\begin{equation}
\mathcal{S}=\left\{ x\in\Omega_{0}:\nabla u\left(x\right)=0\right\} .\label{critical points}
\end{equation}
By Lemma \ref{closedness of singular set}, the singular set $\mathcal{S}$
is a compact set. 

In this paper we would like to investigate the stability of mean-convex
LSF. Based on the uniqueness theorem of viscosity solutions to the
Dirichlet problem (\ref{level set flow}) in \cite{ES} and also the
smooth estimates for a mean-convex LSF in \cite{HK}, we have the
following result:
\begin{thm}
\label{stability theorem}Given a sequence of closed connected hypersurfaces
$\left\{ \Sigma_{0}^{k}\right\} _{k\in\mathbb{N}}$ tending in the
$C^{4}$ topology to a mean-convex closed connected hypersurface $\Sigma_{0}$,
i.e.,
\[
\Sigma_{0}^{k}\,\stackrel{C^{4}}{\rightarrow}\,\Sigma_{0}\quad\textrm{as}\,\,\,k\rightarrow\infty,
\]
their respective LSFs $\left\{ \Sigma_{t}^{k}\right\} $, $k\in\mathbb{N}$,
would converge locally smoothly to the mean-convex LSF $\left\{ \Sigma_{t}\right\} $
in $\left(\Omega_{0}\setminus\mathcal{S}\right)\times\left(0,\infty\right)$
as $k\rightarrow\infty$.
\end{thm}

Note that when $k$ is large, $\Sigma_{0}^{k}$ would also be mean-convex
and so $\left\{ \Sigma_{t}^{k}\right\} $ is a mean-convex LSF that
will stay in $\Omega_{0}^{k}$, where $\Omega_{0}^{k}$ is the region
bounded by $\Sigma_{0}^{k}$. Proof of Theorem \ref{stability theorem}
will be given in Section \ref{smooth convergence off singularities}.

As a consequence of Theorem \ref{stability theorem}, suppose that
$\left\{ \Sigma_{t}\right\} $ is regular for $t\in\left[a,b\right]$,
where $0<a<b<T_{ext}$; in other words, $\left\{ \Sigma_{t}\right\} _{t\in\left[a,b\right]}$
is a mean-convex MCF. Then $\left\{ \Sigma_{t}^{k}\right\} _{t\in\left[a,b\right]}$
would also be regular when $k$ is large and converge smoothly to
$\left\{ \Sigma_{t}\right\} _{t\in\left[a,b\right]}$ as $k\rightarrow\infty$
(see Corollary \ref{closedness during regular times}). Notice that
in the case where $b<T_{1}$, this is the well-known stability theorem
for MCF. So what's new is when $a>T_{1}$.

What's more, when $k$ is sufficiently large, Theorem \ref{stability theorem}
implies that the singular set $\mathcal{S}^{k}$ of the mean-convex
LSF $\left\{ \Sigma_{t}^{k}\right\} $ would be contained in an arbitrarily
small neighborhood of $\mathcal{S}$ (see Corollary \ref{regular space}). 

In order to compare the flows near $\mathcal{S}$, we need to work
out the structure of the singular set first. To this end, in what
follows we shall focus on the case where $\Sigma_{0}$ is  \textbf{two-convex},\footnote{$\kappa_{1}+\kappa_{2}>0$, where $\kappa_{1}\leq\kappa_{2}\leq\cdots\leq\kappa_{n-1}$
denote the principal curvatures of a closed hypersurface with respect
to the inward unit normal vector field.} which is a condition stronger than the mean-convexity\footnote{When $n=3,$ two-convex is just mean-convex.}
but weaker than the convexity. The motivation is that by \cite{CHN}
(see also \cite{HK}), the flow $\left\{ \Sigma_{t}\right\} $ would
be uniformly two-convex at every regular point with respect to the
unit normal $\frac{\nabla u}{\left|\nabla u\right|}$ (see (\ref{uniformly two-convex}));
consequently, every singular point of $\left\{ \Sigma_{t}\right\} $
must be either a round point or a $1$-cylindrical point, the latter
of which will henceforth be abbreviated to a \textbf{cylindrical point}.
Such a flow will be referred to as a \textbf{two-convex LSF}.

Note that $t\in\left(0,T_{ext}\right]$ is called a singular time
of the flow $\left\{ \Sigma_{t}\right\} $ provided 
\[
\mathcal{S}\cap\Sigma_{t}\neq\emptyset;
\]
by (\ref{critical points}), a singular time of $\left\{ \Sigma_{t}\right\} $
amounts to a critical value of $u$. Otherwise, $t$ would be called
a regular time of the flow, i.e., a regular value of $u$. It is conjectured
in \cite{CM4} that a mean-convex LSF has finitely many singular times
(see \cite{AAG} for the rotationally symmetric examples). Under the
hypothesis of finitely many singular times, we prove the following
theorem on the basis of \cite{CM2} and \cite{CM3}:
\begin{thm}
\label{structure of singular set}If the LSF $\left\{ \Sigma_{t}\right\} $
starting from a two-convex closed connected hypersurface $\Sigma_{0}$
has finitely many singular times 
\[
T_{1}\leq\cdots\leq T_{ext},
\]
then its singular set $\mathcal{S}$ is a finite disjoint union of
points and/or compact $C^{1}$ embedded curves.
\end{thm}

Note that by (\ref{critical points}) and the mean value theorem,
$u$ is constant on each connected component of $\mathcal{S}$, meaning
that singularities on each component of $\mathcal{S}$ occur at the
same time. Theorem \ref{structure of singular set} is related to
a speculation in \cite{CM3}, which conjectures that the spacetime
singular set has only finitely many connected components. The proof
of Theorem \ref{structure of singular set} will be give in Section
\ref{singular set at subsequent singular times}. 

Recall that singular points of $\left\{ \Sigma_{t}\right\} $ are
critical points of $u$. In general, critical points of a function
can be classified into local maximum points, local minimum points,
and saddle points. In the case of the arrival time function $u$,
there are no local minimum points (see (\ref{level set flow})). Moreover,
by Corollary \ref{isolated round point} and Corollary \ref{saddle criterion},
a point $p\in\Omega_{0}$ is a saddle point of $u$ if and only if
$p$ is not only a cylindrical point of $\left\{ \Sigma_{t}\right\} $
but also a boundary point of the superleverl set $\Omega_{u\left(p\right)}$
(see (\ref{super level set})). In that case, owing to the asymptotically
cylindrical behavior of $\left\{ \Sigma_{t}\right\} $ near $p$ (see
Section \ref{cylindrical point}), for any small $\phi>0$ there exists
$r=r\left(\phi\right)>0$ so that
\begin{equation}
\Omega_{u\left(p\right)}\cap B_{r}\left(p\right)\,\subset\,\mathscr{C}_{\phi}\left(p\right),\label{cusp singularity}
\end{equation}
where $\mathscr{C}_{\phi}\left(p\right)$ is a double cone with vertex
$p$, axis parallel to the tangent cylinder of $\left\{ \Sigma_{t}\right\} $
at $p$, and angle $\phi$ (see (\ref{cone})); hence, $\Omega_{u\left(p\right)}$
has a cusp singularity at $p$.

By contrast, if a point $p\in\Omega_{0}$ is a local maximum point
of $u$, then $\Omega_{u\left(p\right)}$ vanishes near $p$. According
to Remark \ref{one-sided local max}, when a connected component of
the singular set $\mathcal{S}$ (see Theorem \ref{structure of singular set})
is a curve, all its interior points must be local maximum points of
$u$. Thus, whether the flow has the same kind of behavior near these
curves depends crucially on the their endpoints. We then classify
the curves in Theorem \ref{structure of singular set} into the following
three types: 
\begin{enumerate}
\item \textbf{Vanishing type}: if either the curve has no endpoints (i.e.,
the curve is closed) or both of its endpoints are local maximum points
of $u$;
\item \textbf{Splitting type}: if both endpoints of the curve are saddle
points of $u$;
\item \textbf{Bumpy type}: if one endpoint of the curve is a local maximum
point of $u$ and the other is a saddle point of $u$.
\end{enumerate}
When a connected component of $\mathcal{S}$ is a point, it can be
regarded as a ``degenerate'' curve and so it can also be classified
into the above three categories (see the comment following Definition
\ref{singularity types}).

Now let us write 
\begin{equation}
\mathcal{S}\,=\,\bigsqcup_{j}\,\mathcal{S}_{j},\label{singular components}
\end{equation}
where each $\mathcal{S}_{j}$ is a connected component of $\mathcal{S}$
(called a \textbf{singular component} for short). Let $\hat{\delta}>0$
be a sufficiently small constant such that (\ref{delta 1}), (\ref{delta 2}),
and (\ref{delta 3}) hold. As mentioned earlier, when $k$ is large,
the singular set $\mathcal{S}^{k}$ of the flow $\left\{ \Sigma_{t}^{k}\right\} $
in Theorem \ref{stability theorem} would satisfy
\begin{equation}
\mathcal{S}^{k}\,\subset\,\bigsqcup_{j}\,\mathcal{S}_{j}^{\hat{\delta}},\label{singular set of perturbed flow}
\end{equation}
where $\mathcal{S}_{j}^{\hat{\delta}}$ is the $\hat{\delta}$-neighborhood
of $\mathcal{S}_{j}$. In addition, the singular times of $\left\{ \Sigma_{t}^{k}\right\} $
in $\mathcal{S}_{j}^{\hat{\delta}}$ would approximate to $u\left(\mathcal{S}_{j}\right)$
(see Proposition \ref{C0 compactness}). In the following theorem
we show that under certain conditions, the flow $\left\{ \Sigma_{t}^{k}\right\} $
in $\mathcal{S}_{j}^{\hat{\delta}}$ would have the same type of singular
set as $\mathcal{S}_{j}$. 
\begin{thm}
\label{stability of singular types}Let $\left\{ \Sigma_{t}\right\} $
be the two-convex LSF in Theorem \ref{structure of singular set}
with singular set (\ref{singular components}), and let $\left\{ \Sigma_{t}^{k}\right\} $,
$k\in\mathbb{N}$, be the LSFs in Theorem \ref{stability theorem}.
Given a small constant $\hat{\delta}>0$ fulfilling (\ref{delta 1}),
(\ref{delta 2}), and (\ref{delta 3}), suppose that for every sufficiently
large $k$ the following hold:
\begin{enumerate}
\item For each $j$, there is at most one singular time of $\left\{ \Sigma_{t}^{k}\right\} $
in $\mathcal{S}_{j}^{\hat{\delta}}$ (see Assumption \ref{local singular times hypothesis});
\item When the singular component $\mathcal{S}_{j}$ is of the splitting
type, the flow $\left\{ \Sigma_{t}\right\} $ ``does not get into
the singular set near the endpoints'' (see Assumption \ref{not getting into});
\item When the singular component $\mathcal{S}_{j}$ is of the bumpy type,
there are singularities of $\left\{ \Sigma_{t}^{k}\right\} $ in $\mathcal{S}_{j}^{\hat{\delta}}$
(see Assumption \ref{existence of singularities near bumpy}).
\end{enumerate}
Then there exists $k_{\hat{\delta}}\in\mathbb{N}$ so that for every
$k\geq k_{\hat{\delta}}$, $\left\{ \Sigma_{t}^{k}\right\} $ is a
two-convex LSF with singular set (\ref{singular set of perturbed flow});
moreover, in each $\mathcal{S}_{j}^{\hat{\delta}}$, $\left\{ \Sigma_{t}^{k}\right\} $
has precisely one singular component, which is of the same type as
$\mathcal{S}_{j}$.
\end{thm}

The proof of Theorem \ref{stability of singular types} will be divided
into three parts according to the type of $\mathcal{S}_{j}$: Proposition
\ref{stability of vanishing type} (in Section \ref{vanishing type})
is when $\mathcal{S}_{j}$ belongs to the vanishing type, Proposition
\ref{rule out bumpy type} (in Section \ref{splitting type}) is when
$\mathcal{S}_{j}$ belongs to the splitting type, and Proposition
\ref{uniqueness of bumpy} (in Section \ref{bumpy type}) is when
$\mathcal{S}_{j}$ belongs to the bumpy type. 

\section{\uuline{Stabilit\label{stability of mean-convex LSF off singulariites}y
of mean-convex LSF off singularities}}

The objective of this section is to establish Theorem \ref{stability theorem}.
To accomplish that, we manage to obtain uniform estimates for the
mean-convex LSFs away from singularities. This can be divided into
two parts: the estimates near the boundary (i.e., near the initial
time $0$) in Section \ref{classical theory of MCF} using the classical
theory of MCF, and the interior estimates at the regular points in
Section \ref{interior estimates for mean-convex LSF} based on the
noncollapsing property of mean-convex LSF (cf. \cite{HK}). What we
do in Section \ref{estimates for initial hypersurface} is to give
estimates of the noncollapsing constant and the entropy bound, which
are essential in Section \ref{interior estimates for mean-convex LSF}
and also Section \ref{cylindrical point} (for the cylindrical scales).
Lastly, in Section \ref{smooth convergence off singularities} we
employ these uniform estimates combined with the uniqueness theorem
of viscosity solutions to the Dirichlet problem (\ref{level set flow})
to prove Theorem \ref{stability theorem}.

\subsection{Estimates for initial hypersurface\label{estimates for initial hypersurface}}

We shall estimate the noncollapsing constant $\alpha$ in Section
\ref{Noncollapsing} and the entropy bound $\lambda$ in Section \ref{Entropy}
of the initial hypersurface $\Sigma_{0}$. Note that the constants
$\alpha$ and $\lambda$ will be preserved by the mean-convex LSF
due to the maximum principle and Huisken's monotonicity formula, respectively
(see the comments following Proposition \ref{noncollapsing} and Proposition
\ref{entropy}).

\subsubsection{Noncollapsing\label{Noncollapsing}}

Recall that $\Sigma_{0}$ is called \textbf{$\alpha$-noncollapsing}
for some $\alpha>0$ provided for every $p\in\Sigma_{0}$, there exist
a pair of open balls of radius $\alpha/H\left(p\right)$ on each side\footnote{One of the balls is on the inside, i.e., $\Omega_{0}$, and the other
is on the outside, i.e., $\mathbb{R}^{n}\setminus\bar{\Omega}_{0}$,
where $\Omega_{0}$ is the region bounded by $\Sigma_{0}$.} of $\Sigma_{0}$ that kiss at $p$ (cf. \cite{An}, \cite{ALM},
\cite{HK}). Such a constant $\alpha$ always exists and can be estimated
as follows. 
\begin{lem}
\label{rolling ball condition}There exists a constant $\vartheta\left(n\right)\in\left(0,\frac{1}{2}\right]$
with the following property: If $\Sigma$ is a hypersurface in $\mathbb{R}^{n}$
such that $0\in\Sigma$, $T_{0}\Sigma=\mathbb{R}^{n-1}\times\left\{ 0\right\} $
and that for some $r>0$, $\Sigma\cap B_{r}$ is a graph 
\[
x^{n}=f\left(x^{1},\cdots,x^{n-1}\right)
\]
over (some part of) $T_{0}\Sigma$ with 
\[
\max_{\Sigma\cap B_{r}}\,r\left|A\right|\,\leq\,1,
\]
where $A$ is the second fundamental form of $\Sigma$. Then 
\[
B_{+}\coloneqq B_{\vartheta\left(n\right)r}\left(0,\vartheta\left(n\right)r\right),\quad B_{-}\coloneqq B_{\vartheta\left(n\right)r}\left(0,-\vartheta\left(n\right)r\right)
\]
are contained in $B_{r}$ and located on the upside and downside of
$\Sigma\cap B_{r}$, respectively.
\end{lem}

\begin{proof}
Note that $f\left(0\right)=0$ and $\nabla f\left(0\right)=0$. Let
$\rho\in\left(0,\frac{r}{2}\right]$ be a radius so that the local
graph function $f$ is defined on $B_{\rho}^{n-1}$ with $\left|\nabla f\right|\leq\varepsilon$,
where $\varepsilon\in\left(0,1\right)$ is a small constant to be
determined. With the local graph parametrization, the metric and the
second fundamental form of $\Sigma$ can be written as
\[
g_{ij}=\delta_{ij}+\partial_{i}f\,\partial_{j}f,\quad A_{ij}=\frac{\partial_{ij}f}{\sqrt{1+\left|\nabla f\right|^{2}}}.
\]
Let 
\[
K=\max_{\Sigma\cap B_{r}}\left|A\right|.
\]
Then it follows from 
\[
\left|A_{ij}v^{i}v^{j}\right|\,\leq\,Kg_{ij}v^{i}v^{j}\quad\forall\,v\in\mathbb{R}^{n-1}
\]
that on $B_{\rho}^{n-1}$ we have for every $v\in\mathbb{R}^{n-1},$
\begin{equation}
\left|\partial_{ij}f\,v^{i}v^{j}\right|\,\leq\,\sqrt{1+\left|\nabla f\right|^{2}}\,K\left(\delta_{ij}+\partial_{i}f\,\partial_{j}f\right)v^{i}v^{j}\,\leq\,2K\left|v\right|^{2}\label{rolling ball condition: Hessian}
\end{equation}
provided $\varepsilon=\varepsilon\left(n\right)\in\left(0,1\right)$
is sufficiently small. In particular, we obtain
\[
\left|\nabla f\right|\,\leq\,\hat{C}\left(n\right)K\rho\quad\textrm{on}\,\,B_{\rho}^{n-1},
\]
where $\hat{C}\left(n\right)\geq1$ is a constant. As a result, we
deduce that the ``optimal'' radius $\rho$ in the above must be
bounded below by
\begin{equation}
\hat{\rho}\coloneqq\frac{\varepsilon\left(n\right)}{2\hat{C}\left(n\right)K}\,\geq\,\frac{\varepsilon\left(n\right)}{2\hat{C}\left(n\right)}r.\label{rolling ball condition: radius}
\end{equation}
In addition, by Taylor's theorem we have 
\[
f\left(x'\right)=\int_{0}^{1}\partial_{ij}f\left(sx'\right)\,x^{i}x^{j}\,\left(1-s\right)ds
\]
for $x'=\left(x^{1},\cdots,x^{n-1}\right)\in B_{\hat{\rho}}^{n-1}$.
It follows from (\ref{rolling ball condition: Hessian}) that 
\begin{equation}
\left|f\left(x'\right)\right|\,\leq\,K\left|x'\right|^{2}\,\leq\,\frac{1}{r}\left|x'\right|^{2}\,\leq\,\frac{\hat{C}\left(n\right)}{\varepsilon\left(n\right)r}\left|x'\right|^{2}\quad\forall\,x'\in B_{\hat{\rho}}^{n-1}.\label{rolling ball coniditon: parabola}
\end{equation}
Set 
\[
\vartheta\left(n\right)=\frac{\varepsilon\left(n\right)}{2\hat{C}\left(n\right)}.
\]
By (\ref{rolling ball condition: radius}) and (\ref{rolling ball coniditon: parabola}),
the two open balls $B_{+}$ and $B_{-}$ are contained in $B_{r}\cap\left\{ \left|x'\right|<\vartheta\left(n\right)r\right\} $
and located on each side of the graph $x^{n}=f\left(x'\right)$. 
\end{proof}
In order to apply Lemma \ref{rolling ball condition} to every point
of $\Sigma_{0}$, we introduce the following definition:
\begin{defn}
\label{injectivity radius}For a closed hypersurface $\Sigma$, let
$\mathfrak{R}$ be the set of every $r\in\left(0,K^{-1}\right]$,
where 
\[
K=\max_{\Sigma}\left|A\right|,
\]
such that that for every $p\in\Sigma$, $\Sigma\cap B_{r}\left(p\right)$
is a graph over (some part of) $T_{p}\Sigma$. Note that $\mathfrak{R}\neq\emptyset$
due to the embeddedness and compactness of $\Sigma$. Then define
\[
\textrm{rad}\,\,\Sigma=\sup\,\mathfrak{R}.
\]
\end{defn}

The following proposition follows from applying Lemma \ref{rolling ball condition}
to every point of $\Sigma_{0}$.
\begin{prop}
\label{noncollapsing}Let $\imath,\kappa$ be positive constants such
that 
\[
\textrm{rad}\,\,\Sigma_{0}\geq\imath,\quad\min_{\Sigma_{0}}H\geq\kappa.
\]
Then there is a constant $\alpha=\alpha\left(n,\imath,\kappa\right)>0$
\footnote{The constant $\alpha$ can be chosen to be any number in $\left(0,\vartheta\left(n\right)\iota\kappa\right)$.}such
that $\Sigma_{0}$ is $\alpha$-noncollaspsing. 
\end{prop}

According to \cite{HK}, the $\alpha$-noncollaspsing property of
$\Sigma_{0}$ can be preserved by the mean-convex LSF in the sense
that near every regular point $p\in\Omega_{0}$, there exist two open
balls of radius 
\[
\frac{\alpha}{H\left(p\right)}=\alpha\left|\nabla u\left(p\right)\right|
\]
kissing at $p$, one of which is inside $\Omega_{u\left(p\right)}$
and the other is outside $\bar{\Omega}_{u\left(p\right)}$.

\subsubsection{Entropy\label{Entropy}}

The area ratios of $\Sigma_{0}$ can be estimated as follows.
\begin{lem}
\label{area ratio}Let $K,\mathfrak{A},\mathfrak{D}$ be positive
constants such that
\[
\max_{\Sigma_{0}}H\leq K,\quad\mathcal{H}^{n-1}\left(\Sigma_{0}\right)\leq\mathfrak{A},\quad\textrm{diam}\,\Sigma_{0}\leq\mathfrak{D}.
\]
Then for every $p\in\mathbb{R}^{n}$ and $r>0$ we have
\[
\frac{\mathcal{H}^{n-1}\left(\Sigma_{0}\cap B_{r}\left(p\right)\right)}{r^{n-1}}\,\leq\,2^{n-1}e^{K\left(\mathfrak{D}+1\right)}\mathfrak{A}.
\]
\end{lem}

\begin{proof}
Set $R=\mathfrak{D}+1$. 

\noindent\uline{Case 1}: For $p\in\Sigma_{0}$ and $r\in\left(0,R\right]$,
by the monotonicity formula (cf. \cite{Al}) we have 
\begin{equation}
\frac{\mathcal{H}^{n-1}\left(\Sigma_{0}\cap B_{r}\left(p\right)\right)}{r^{n-1}}\,\leq\,e^{K\left(R-r\right)}\frac{\mathcal{H}^{n-1}\left(\Sigma_{0}\cap B_{R}\left(p\right)\right)}{R^{n-1}}\,\leq\,e^{KR}\,\frac{\mathcal{H}^{n-1}\left(\Sigma_{0}\right)}{R^{n-1}}.\label{area ratio: monotonicity formula}
\end{equation}
\uline{Case 2}: For $p\in\Sigma_{0}$ and $r>R$, we have 
\[
\frac{\mathcal{H}^{n-1}\left(\Sigma_{0}\cap B_{r}\left(p\right)\right)}{r^{n-1}}\,\leq\,\frac{\mathcal{H}^{n-1}\left(\Sigma_{0}\right)}{R^{n-1}}.
\]
\uline{Case 3}: For $p\notin\Sigma_{0}$ and $r\in\left(0,\frac{1}{2}R\right]$,
either $\Sigma_{0}\cap B_{r}\left(p\right)=\emptyset$, in which case
we have $\mathcal{H}^{n-1}\left(\Sigma_{0}\cap B_{r}\left(p\right)\right)=0$,
or we can find $q\in\Sigma_{0}\cap B_{r}\left(p\right)$ so that (\ref{area ratio: monotonicity formula})
yields
\[
\frac{\mathcal{H}^{n-1}\left(\Sigma_{0}\cap B_{r}\left(p\right)\right)}{r^{n-1}}\,\leq\,\frac{\mathcal{H}^{n-1}\left(\Sigma_{0}\cap B_{2r}\left(q\right)\right)}{r^{n-1}}\,\leq\,2^{n-1}e^{KR}\,\frac{\mathcal{H}^{n-1}\left(\Sigma_{0}\right)}{R^{n-1}}.
\]
\uline{Case 4}: For $p\notin\Sigma_{0}$ and $r>\frac{1}{2}R$,
we have 
\[
\frac{\mathcal{H}^{n-1}\left(\Sigma_{0}\cap B_{r}\left(p\right)\right)}{r^{n-1}}\leq\frac{\mathcal{H}^{n-1}\left(\Sigma_{0}\right)}{\left(\frac{1}{2}R\right)^{n-1}}=2^{n-1}\frac{\mathcal{H}^{n-1}\left(\Sigma_{0}\right)}{R^{n-1}}.
\]

Therefore, for every $p\in\mathbb{R}^{n}$ and $r>0$ we have 
\[
\frac{\mathcal{H}^{n-1}\left(\Sigma_{0}\cap B_{r}\left(p\right)\right)}{r^{n-1}}\leq2^{n-1}e^{KR}\frac{\mathcal{H}^{n-1}\left(\Sigma_{0}\right)}{R^{n-1}}\leq2^{n-1}e^{K\left(\mathfrak{D}+1\right)}\mathfrak{A}.
\]
\end{proof}
Recall that the entropy of $\Sigma_{0}$ is defined as (cf. \cite{CM1})
\[
E\left[\Sigma_{0}\right]=\sup_{p\in\mathbb{R}^{n},\,r>0}F\left(\frac{1}{r}\left(\Sigma_{0}-p\right)\right),
\]
where $F$ is the Gaussian area, namely,
\[
F\left(\Sigma\right)\coloneqq\int_{\Sigma}\frac{1}{\left(4\pi\right)^{\frac{n-1}{2}}}\,e^{-\frac{\left|x\right|^{2}}{4}}\,d\mathcal{H}^{n-1}\left(x\right).
\]
As the area ratios of $\Sigma_{0}$ are uniformly bounded, $\Sigma_{0}$
has finite entropy (cf. \cite{E}); moreover, by a simple calculation
we have
\begin{equation}
E\left[\Sigma_{0}\right]\,\leq\,C\left(n\right)\sup_{p\in\mathbb{R}^{n},\,r>0}\frac{\mathcal{H}^{n-1}\left(\Sigma_{0}\cap B_{r}\left(p\right)\right)}{r^{n-1}}.\label{entropy estimate}
\end{equation}
Combining (\ref{entropy estimate}) with Lemma \ref{area ratio} give
the following proposition.
\begin{prop}
\label{entropy}Let $K,\mathfrak{A},\mathfrak{D}$ be positive constants
such that 
\[
\sup_{\Sigma_{0}}H\leq K,\quad\mathcal{H}^{n-1}\left(\Sigma_{0}\right)\leq\mathfrak{A},\quad\textrm{diam}\,\Sigma_{0}\leq\mathfrak{D}.
\]
Then there exists a constant $\lambda=\lambda\left(n,K,\mathfrak{A},\mathfrak{D}\right)>0$
so that $E\left[\Sigma_{0}\right]\leq\lambda$. 
\end{prop}

Since $\left\{ \mathcal{H}^{n-1}\lfloor\Sigma_{t}\right\} _{t\geq0}$
is a Brakke flow (cf. \cite{I1}, \cite{W1}), it follows from Huisken's
monotonicity formula (cf. \cite{H2}, \cite{I2}) that the entropy
is nonincreasing along $\left\{ \Sigma_{t}\right\} $; in particular,
\[
E\left[\Sigma_{t}\right]\leq\lambda\quad\forall\,t\geq0.
\]

\subsection{Classical theory of MCF\label{classical theory of MCF} }

In this subsection we shall get smooth estimates for $\left\{ \Sigma_{t}\right\} $
when $t$ is close to $0$, which corresponds to the estimates at
points near $\partial\Omega_{0}=\Sigma_{0}$. As the LSF $\left\{ \Sigma_{t}\right\} $
is a MCF for $t\in\left[0,T_{1}\right)$, where $T_{1}$ is the first
singular time, such estimates are primarily based on the classical
theory of MCF. The main results are Proposition \ref{time of smooth existence},
Proposition \ref{moving distance}, (\ref{short time curvature estimate}),
and (\ref{short time covariant derivatives of curvature estimate}).

Recall that one of the ways to estimate the MCF near the initial time
is to describe the flow as a evolving graph over the initial hypersurface
and then study the graph function via the PDE it satisfies. Specifically,
when $t$ is sufficiently close to $0$, $\Sigma_{t}$ can parametrized
as a normal graph of $v_{t}=v\left(\cdot,t\right)$ over $\Sigma_{0}$,
namely, 
\begin{equation}
x_{t}=x\left(\cdot,t\right)=x_{0}+v_{t}\,N_{0},\label{normal graph parametrization}
\end{equation}
where $x_{0}$ and $N_{0}$ are the position vector and the inward
unit normal vector of $\Sigma_{0}$, respectively. Thus, $\left\{ \Sigma_{t}\right\} $
being a MCF starting from $\Sigma_{0}$ at time $0$ corresponds to
that the function $v$ satisfying a certain quasilinear parabolic
PDE with $v\left(\cdot,0\right)=0$. More precisely, under the condition
that 
\begin{equation}
\left|\nabla v\right|\,+\,\frac{\left|v\right|}{\textrm{rad}\,\,\Sigma_{0}}\,\leq\,\varepsilon\left(n\right),\label{small C1 norm}
\end{equation}
where $\nabla$ is the Levi-Civita connection on $\Sigma_{0}$, the
norm $\left|\cdot\right|$ is defined by the induced metric $g_{0}$
on $\Sigma_{0}$, and $\textrm{rad}\,\,\Sigma_{0}$ is defined in
Definition \ref{injectivity radius}, the PDE is given by
\begin{equation}
\partial_{t}v\,=\,H_{0}\,+\,a^{ij}\nabla_{ij}v\,+\,\left|A_{0}\right|^{2}v\,+\,b,\label{quasilinear parabolic equation}
\end{equation}
with 
\[
a^{ij}\,=\,g_{0}^{ij}\,+\,2A_{0}^{ij}v\,+\,Q^{ij}\left(\nabla v,A_{0}v\right),
\]
\[
b\,=\,v\,\nabla A_{0}*\nabla v\,+\,A*Q\left(\nabla v,A_{0}v\right)\,+\,v\,\nabla A_{0}*Q\left(\nabla v,A_{0}v\right)
\]
(see Lemma 3.14 in \cite{G1}), where
\begin{enumerate}
\item $H_{0}$ is the mean curvature of $\Sigma_{0}$;
\item $\nabla_{ij}v$ is the Hessian of $v$ on $\Sigma_{0}$;
\item the notation $Q$ means some kind of analytic function/tensor that
is at least ``quadratic'' (in the form of contraction via the metric
$g_{0}$) in its arguments;
\item the notation $\ast$ means some kind of contraction of tensors. 
\end{enumerate}
The short time existence of MCF, the uniqueness of MCF, the stability
of MCF, etc., all follow from studying (\ref{quasilinear parabolic equation})
using the classical theory of parabolic PDE (cf. \cite{HP}). The
following proposition is one of the direct results. 
\begin{prop}
\label{time of smooth existence}Let $n,\iota,K,K',K''$ be positive
constants such that 
\begin{equation}
\textrm{rad}\,\,\Sigma_{0}\geq\iota,\quad\max_{\Sigma_{0}}\left|A\right|\leq K,\quad\max_{\Sigma_{0}}\left|\nabla A\right|\leq K',\quad\max_{\Sigma_{0}}\left|\nabla^{2}A\right|\leq K''.\label{time of smooth existence: constatns}
\end{equation}
Then there exists $\dot{T}=\dot{T}\left(n,\iota,K,K',K''\right)>0$
so that $\left\{ \Sigma_{t}\right\} _{0\leq t\leq2\dot{T}}$ is a
MCF;\footnote{In particular, $T_{1}>2\dot{T}$.} moreover, for every
$t\in\left[0,2\dot{T}\right]$, the hypersurface $\Sigma_{t}$ stays
in a tubular neighborhood of $\Sigma_{0}$ and is a normal graph of
$v\left(\cdot,t\right)$ over $\Sigma_{0}$ with $v\in C^{\infty}\left(\Sigma_{0}\times\left[0,2\dot{T}\right]\right)$
satisfying (\ref{small C1 norm}), (\ref{quasilinear parabolic equation}),
and 
\begin{equation}
\left|\nabla^{2}v\right|\,\leq\,C\left(n,\iota,K,K',K''\right).\label{time of smooth existence: Hessian}
\end{equation}

Furthermore, if $\left\{ \Sigma_{0}^{k}\right\} _{k\in\mathbb{N}}$
is a sequence of closed connected hypersurfaces tending in the $C^{4}$
topology to $\Sigma_{0}$ as $k\rightarrow\infty$, then for every
sufficiently large $k$ the following hold:
\begin{enumerate}
\item The LSF $\left\{ \Sigma_{t}^{k}\right\} $ starting from $\Sigma_{0}^{k}$
at time $0$ is a MCF during $t\in\left[0,2\dot{T}\right]$; in addition,
the flow $\left\{ \Sigma_{t}^{k}\right\} _{0\leq t\leq2\dot{T}}$
converges in the $C^{2}$ topology to $\left\{ \Sigma_{t}\right\} _{0\leq t\leq2\dot{T}}$
as $k\rightarrow\infty$. 
\item The initial hypersurface $\Sigma_{0}^{k}$ is mean-convex;\footnote{When $\Sigma_{0}$ is two-convex, $\Sigma_{0}^{k}$ would be two-convex
when $k$ is large (see Proposition \ref{two-convexity of approximations}).} hence the associated arrival time function $u^{k}:\Omega_{0}^{k}\rightarrow\left[0,\infty\right)$
of $\left\{ \Sigma_{t}^{k}\right\} $ is defined, where $\Omega_{0}^{k}$
is the open connected region bounded by $\Sigma_{0}^{k}$.
\end{enumerate}
\end{prop}

By virtue of the smooth compactness of MCF, the sense of convergence
in Proposition \ref{time of smooth existence} can be improved away
from the initial time as follows.
\begin{cor}
\label{smooth convergence near initial time}In Proposition \ref{time of smooth existence},
for any given $\tau\in\left(0,\dot{T}\right)$, the MCF $\left\{ \Sigma_{t}^{k}\right\} _{t\in\left[\tau,2\dot{T}\right]}$
would converge smoothly to $\left\{ \Sigma_{t}\right\} _{t\in\left[\tau,2\dot{T}\right]}$
as $k\rightarrow\infty$. 
\end{cor}

\begin{proof}
Choose a large open ball $\mathcal{B}$ in $\mathbb{R}^{n}$ so that
\[
\Sigma_{t}\subset\mathcal{B}\quad\forall\,t\in\left[0,2\dot{T}\right].
\]
By Proposition \ref{time of smooth existence}, there exists $k_{0}\in\mathbb{N}$
so that for every $k\geq k_{0}$, the MCF $\left\{ \Sigma_{t}^{k}\right\} _{t\in\left[0,2\dot{T}\right]}$
satisfies 
\[
\Sigma_{t}^{k}\subset\mathcal{B}\quad\forall\,\,0\leq t\leq2\dot{T},
\]
\[
\max_{0\leq t\leq2\dot{T}}\max_{\Sigma_{t}^{k}}\left|A\right|\,\leq\,2\max_{0\leq t\leq2\dot{T}}\max_{\Sigma_{t}}\left|A\right|,
\]
\[
\mathcal{H}^{n}\left(\Sigma_{0}^{k}\right)\leq2\,\mathcal{H}^{n}\left(\Sigma_{0}\right).
\]
Fix $\tau\in\left(0,\dot{T}\right)$. It follows from the compactness
theorem for MCF (cf. \cite{H2}) that any subsequence of 
\[
\left\{ \left\{ \Sigma_{t}^{k}\right\} _{t\in\left[\tau,2\dot{T}\right]}\right\} _{k\geq k_{0}}
\]
has a smoothly convergent subsequence, whose limit by Proposition
\ref{time of smooth existence} is indeed $\left\{ \Sigma_{t}\right\} _{t\in\left[\tau,2\dot{T}\right]}$.
Thus, the whole sequence $\left\{ \left\{ \Sigma_{t}^{k}\right\} _{t\in\left[\tau,2\dot{T}\right]}\right\} _{k\geq k_{0}}$
must converge smoothly to $\left\{ \Sigma_{t}\right\} _{t\in\left[\tau,2\dot{T}\right]}$
as $k\rightarrow\infty$, completing the proof.
\end{proof}
Due to the maximum principle, the mean-convexity is preserved by the
MCF (cf. \cite{M}) and hence $\left\{ \Sigma_{t}\right\} _{0\leq t\leq2\dot{T}}$
moves monotonically toward the inside. Below we estimate how far the
flow can be away from $\partial\Omega_{0}=\Sigma_{0}$.
\begin{prop}
\label{moving distance}Let $\kappa$ be a positive constant such
that 
\begin{equation}
\min_{\Sigma_{0}}H\geq\kappa.\label{moving distance: mean curvature}
\end{equation}
Then the normal graph of $v_{t}=v\left(\cdot,t\right)$ in Proposition
\ref{time of smooth existence} moves monotonically toward the inside
with
\[
\partial_{t}v\geq\kappa
\]
on $\Sigma_{0}\times\left[0,2\dot{T}\right]$. In particular, 
\[
\textrm{dist}\left(\Omega_{\dot{T}},\Sigma_{0}\right)\,\geq\,\kappa\dot{T}.
\]
\end{prop}

\begin{proof}
Along a MCF we have $\partial_{t}x\cdot N=H$ (cf. \cite{M}). In
the normal graph parametrization (\ref{normal graph parametrization}),
this means that 
\[
\partial_{t}v\,N_{0}\cdot N=H.
\]
Then it follows from the maximum principle for the mean curvature
along MCF that
\begin{equation}
\left|\partial_{t}v\right|\,\geq\,\left|\partial_{t}v\,N_{0}\cdot N\right|\,=\,H\,\geq\,\kappa\label{moving distance: mean-convexity}
\end{equation}
on $\Sigma_{0}\times\left[0,2\dot{T}\right]$ (cf. \cite{M}). In
particular, $\partial_{t}v\neq0$ everywhere on on $\Sigma_{0}\times\left[0,2\dot{T}\right]$.
On the other hand, since 
\[
\partial_{t}v\left(\cdot,0\right)=H_{0}\geq\kappa,
\]
by the intermediate value theorem we deduce that $\partial_{t}v\left(\cdot,t\right)>0$
for every $t\in\left[0,2\dot{T}\right]$; hence by (\ref{moving distance: mean-convexity})
we obtain $\partial_{t}v\geq\kappa$ on $\Sigma_{0}\times\left[0,2\dot{T}\right]$.
\end{proof}
Note that Proposition \ref{time of smooth existence} yields the estimates
of the second fundamental form of $\left\{ \Sigma_{t}\right\} _{0\leq t\leq2\dot{T}}$
and its covariant derivatives. More specifically, by Lemma 3.14 in
\cite{G1}, under the condition (\ref{small C1 norm}), for every
$t\in\left[0,2\dot{T}\right]$ we have
\begin{equation}
\max_{\Sigma_{t}}\left|A\right|\,\leq\,C\left(n\right)\left(\max_{\Sigma_{0}}\left|A\right|\,+\,\max_{\Sigma_{0}}\left|\nabla^{2}v\left(\cdotp,t\right)\right|\,+\,\max_{\Sigma_{0}}\left|v\left(\cdotp,t\right)\right|\,\max_{\Sigma_{0}}\left|\nabla A\right|\right)\label{short time curvature}
\end{equation}
Let $\kappa$ be the constant in (\ref{moving distance: mean curvature}).
Note that (\ref{small C1 norm}) implies

\begin{equation}
\frac{\kappa}{\sqrt{n-1}}\left|v\right|\,\leq\,\frac{1}{\sqrt{n-1}}H_{0}\left|v\right|\,\leq\,\left|A_{0}v\right|\,\leq\,\frac{\left|v\right|}{\textrm{rad}\,\,\Sigma_{0}}\,\leq\,\varepsilon\left(n\right).\label{short time deviation}
\end{equation}
Thus, (\ref{time of smooth existence: Hessian}), (\ref{short time curvature}),
and (\ref{short time deviation}) yield
\begin{equation}
\max_{0\leq t\leq2\dot{T}}\max_{\Sigma_{t}}\left|A\right|\,\leq\,C\left(n,\iota,\kappa,K,K',K''\right),\label{short time curvature estimate}
\end{equation}
where the notations are as defined in Proposition \ref{time of smooth existence}
and Proposition \ref{moving distance}. Then it follows the maximum
principle (for instance, see Proposition 6.2 in \cite{G1}) that 
\[
\max_{\Sigma_{t}}\left|\nabla A\right|\,+\,\max_{\Sigma_{t}}\left|\nabla^{2}A\right|\,\leq\,C\left(n,\,\max_{0\leq\tau\leq2\dot{T}}\max_{\Sigma_{\tau}}\left|A\right|,\,\,\max_{\Sigma_{0}}\left|\nabla A\right|,\,\,\max_{\Sigma_{0}}\left|\nabla^{2}A\right|\right)
\]
for $0\leq t\leq2\dot{T}$. So by (\ref{short time curvature estimate})
we obtain 
\begin{equation}
\max_{0\leq t\leq2\dot{T}}\max_{\Sigma_{t}}\left|\nabla A\right|\,+\,\max_{0\leq t\leq2\dot{T}}\max_{\Sigma_{t}}\left|\nabla^{2}A\right|\,\leq\,C\left(n,\iota,\kappa,K,K',K''\right).\label{short time covariant derivatives of curvature estimate}
\end{equation}

\subsection{Interior estimates for mean-convex LSF\label{interior estimates for mean-convex LSF}}

In this subsection we shall obtain smooth estimates for $\left\{ \Sigma_{t}\right\} $
at regular points that are away from $\partial\Omega_{0}=\Sigma_{0}$.
Such estimates are based on the noncollapsing property of mean-convex
LSF (cf. \cite{HK}), saying that the smooth scale (see Definition
\ref{smooth scale}) at a regular point $p$ is bounded below by,
up to a multiplicative constant, 
\[
H\left(p\right)^{-1}\,=\,\left|\nabla u\left(p\right)\right|.
\]
(see Proposition \ref{interior estimate}, Proposition \ref{smooth estimate},
and Corollary \ref{local graph estimate}). In the end of the subsection
there is an estimate of the Hessian $\nabla^{2}u$ (see Proposition
\ref{Hessian estimate}), which is pivotal in getting the convergence
of $\nabla u^{k}$'s\footnote{$u^{k}$ is the arrival time function of $\left\{ \Sigma_{t}^{k}\right\} $,
see Proposition \ref{time of smooth existence}.} in Section \ref{smooth convergence off singularities} (see Proposition
\ref{C1 compactness}).

To begin with, we give an estimate of $\left|\nabla u\right|$ from
above in Proposition \ref{gradient estimate}, which will be used
in Section \ref{smooth convergence off singularities} to prove the
convergence of $u^{k}$'s (see Proposition \ref{C0 compactness}).
In \cite{ES} there is already a gradient estimate for $u$, which
depends implicitly on $\Sigma_{0}$. Here we improve the estimate
such that it depends more explicitly on $\Sigma_{0}$. Proposition
\ref{gradient estimate} can also be regarded as a generalization
of the maximum principle for the mean curvature along a mean-convex
LSF. 
\begin{prop}
\label{gradient estimate}Let $\kappa$ be a positive constant such
that 
\[
\min_{\Sigma_{0}}H\,\geq\,\kappa.
\]
Then the gradient of the arrival time function $u$ satisfies 
\[
\sup_{\Omega_{0}}\left|\nabla u\right|\,\leq\,\kappa^{-1}.
\]
As a consequence, at every regular point $p\in\Omega_{0}$ we have
\[
H\left(p\right)=\left|\nabla u\left(p\right)\right|^{-1}\,\geq\,\kappa.
\]
\end{prop}

\begin{proof}
By \cite{ES}, for every $\epsilon\in\left(0,1\right]$ there is $u^{\epsilon}\in C^{\infty}\left(\bar{\Omega}_{0}\right)$
satisfying the $\epsilon$-elliptic regularization of (\ref{level set flow}),
that is,
\begin{equation}
-\left(\mathrm{I}-\frac{\nabla u^{\epsilon}\varotimes\nabla u^{\epsilon}}{\left|\nabla u^{\epsilon}\right|^{2}+\epsilon^{2}}\right)\cdotp\nabla^{2}u^{\epsilon}=1\quad\textrm{in}\,\,\Omega_{0},\label{gradient estimate: elliptic regularization}
\end{equation}
\[
u^{\epsilon}=0\quad\textrm{on}\,\,\Sigma_{0};
\]
and as $\epsilon\searrow0$ 
\begin{equation}
u^{\epsilon}\left(x\right)\,\stackrel{C^{0}}{\rightarrow}\,u\left(x\right)\quad\textrm{on}\,\,\bar{\Omega}_{0}.\label{gradient estimate: uniform convergence}
\end{equation}
Moreover, there exists a constant $C\geq1$ such that for every $x\in\bar{\Omega}_{0}$,
\[
\frac{1}{C}\,\textrm{dist}\left(x,\Sigma_{0}\right)\,\leq\,u^{\epsilon}\left(x\right)\,\leq\,C\,\textrm{dist}\left(x,\Sigma_{0}\right),
\]
\begin{equation}
\left|\nabla u^{\epsilon}\left(x\right)\right|\leq C.\label{gradient estimate: gradient bound}
\end{equation}
The geometric interpretation of the $\epsilon$-elliptic regularization
is that 
\begin{equation}
\Gamma_{t}^{\epsilon}\coloneqq\left\{ \left(x,x^{n+1}\right)\in\bar{\Omega}_{0}\times\mathbb{R}\,:\,\,x^{n+1}=\epsilon^{-1}u^{\epsilon}\left(x\right)-\epsilon^{-1}t\right\} ,\quad t\geq0\label{gradient estimate: translating soliton}
\end{equation}
is a translating MCF in higher dimensional space, which moves downward
(along the negative $x^{n+1}$-direction) with speed $\epsilon^{-1}$;
in other words, $\Gamma_{0}^{\epsilon}=\left\{ \,x^{n+1}=\epsilon^{-1}u^{\epsilon}\left(x\right)\right\} $
is a translating soliton. Note that its mean curvature with respect
to the downward unit normal 
\begin{equation}
\frac{\left(\nabla\epsilon^{-1}u^{\epsilon},-1\right)}{\sqrt{\left|\nabla\epsilon^{-1}u^{\epsilon}\right|^{2}+1}}=\frac{\left(\nabla u^{\epsilon},-\epsilon\right)}{\sqrt{\left|\nabla u^{\epsilon}\right|^{2}+\epsilon^{2}}}\label{gradient estimate: downward normal}
\end{equation}
is 
\[
-\nabla\cdot\frac{\nabla\epsilon^{-1}u^{\epsilon}}{\sqrt{\left|\nabla\epsilon^{-1}u^{\epsilon}\right|^{2}+1}}=\frac{-1}{\sqrt{\left|\nabla\epsilon^{-1}u^{\epsilon}\right|^{2}+1}}\left(\mathrm{I}-\frac{\nabla\epsilon^{-1}u^{\epsilon}\varotimes\nabla\epsilon^{-1}u^{\epsilon}}{\left|\nabla\epsilon^{-1}u^{\epsilon}\right|^{2}+1}\right)\cdotp\nabla^{2}\epsilon^{-1}u^{\epsilon}
\]
\begin{equation}
=\frac{-1}{\sqrt{\left|\nabla u^{\epsilon}\right|^{2}+\epsilon^{2}}}\left(\mathrm{I}-\frac{\nabla u^{\epsilon}\varotimes\nabla u^{\epsilon}}{\left|\nabla u^{\epsilon}\right|^{2}+\epsilon^{2}}\right)\cdotp\nabla^{2}u^{\epsilon}=\frac{1}{\sqrt{\left|\nabla u^{\epsilon}\right|^{2}+\epsilon^{2}}}.\label{gradient estimate: approximate mean curvature}
\end{equation}
By (\ref{gradient estimate: uniform convergence}), as $\epsilon\searrow0$
\begin{equation}
u^{\epsilon}\left(x\right)-\epsilon\,x^{n+1}\,\stackrel{C_{loc}^{0}}{\longrightarrow}\,u\left(x\right)\quad\textrm{on}\,\,\bar{\Omega}_{0}\times\mathbb{R}_{+},\label{gradient estimate: arrival time}
\end{equation}
meaning that the arrival time function of $\left\{ \Gamma_{t}^{\epsilon}\right\} $
(which can be obtained by solving the equation in (\ref{gradient estimate: translating soliton})
for $t$) converges locally uniformly to the arrival time function
of the ``cylindrical'' mean-convex LSF

\begin{equation}
\Gamma_{t}\coloneqq\Sigma_{t}\times\mathbb{R}_{+}=\left\{ \left(x,x^{n+1}\right)\in\bar{\Omega}_{0}\times\mathbb{R}_{+}:\,u\left(x\right)=t\right\} ,\quad t\geq0.\label{gradient estimate: cylindrical LSF}
\end{equation}
In fact, the convergence is locally smooth for small $t>0$ and away
from $x^{n+1}=0$ (as was mentioned in section 4.1 in \cite{HK}).
To see that, firstly let us get a uniform bound for the local area
as follows. By (\ref{gradient estimate: gradient bound}), Proposition
\ref{moving distance}, coarea formula, and the fact that the total
area is non-increasing along MCF $\left\{ \Sigma_{t}\right\} _{t\in\left[0,\dot{T}\right]}$,\footnote{See Proposition \ref{time of smooth existence} for the definition
of $\dot{T}$.} for each $0<\epsilon\leq\frac{\kappa}{2C}\dot{T}$ we have
\[
\mathcal{H}^{n}\left(\Gamma_{0}^{\epsilon}\cap\left\{ 0<x^{n+1}<2\right\} \right)
\]
\[
=\int_{0<u^{\epsilon}<2\epsilon}\sqrt{\left|\nabla\epsilon^{-1}u^{\epsilon}\right|^{2}+1}\,dx\,=\,\frac{1}{\epsilon}\int_{0<u^{\epsilon}<2\epsilon}\sqrt{\left|\nabla u^{\epsilon}\right|^{2}+\epsilon^{2}}\,dx
\]
\[
\leq\frac{1}{\epsilon}\int_{0<u^{\epsilon}<2\epsilon}\sqrt{C^{2}+\epsilon^{2}}\,dx\,\leq\,\frac{1}{\epsilon}\sqrt{C^{2}+\epsilon^{2}}\,\,\mathcal{H}^{n}\left\{ x\in\Omega_{0}\,:\,0<\textrm{dist}\left(x,\Sigma_{0}\right)<2C\epsilon\right\} 
\]
\[
\leq\frac{1}{\epsilon}\sqrt{C^{2}+\epsilon^{2}}\,\,\mathcal{H}^{n}\left\{ x\in\Omega_{0}\,:\,0<u\left(x\right)<\frac{2C}{\kappa}\epsilon\right\} 
\]
\[
=\frac{1}{\epsilon}\sqrt{C^{2}+\epsilon^{2}}\,\int_{0}^{\frac{2C}{\kappa}\epsilon}\int_{\Sigma_{t}}\frac{1}{\left|\nabla u\right|}\,d\mathcal{H}^{n-1}dt
\]
\[
\leq\frac{2C}{\kappa}\sqrt{C^{2}+1}\,K\,\mathcal{H}^{n-1}\left(\Sigma_{0}\right),
\]
where 
\[
K=\sup_{0<u<\dot{T}}\frac{1}{\left|\nabla u\right|}=\max_{0\leq u\leq\dot{T}}\max_{\Sigma_{t}}H.
\]
Let $\delta\in\left(0,\frac{1}{2}\dot{T}\right)$ be a sufficiently
small constant (depending on $n$, $\dot{T}$). It follows from the
local area bound (see Proposition 4.9 in \cite{E}) and Brakke's compactness
theorem (cf. \cite{I1}) that for some sequence $\epsilon_{i}\searrow0$,
the MCF $\left\{ \Gamma_{t}^{\epsilon_{i}}\right\} _{t\in\left(\frac{1}{16}\delta,\frac{31}{16}\delta\right)}$
converge in the weak topology to an integral Brakke flow $\left\{ \mu_{t}\right\} _{t\in\left(\frac{1}{16}\delta,\frac{31}{16}\delta\right)}$
in $\Omega_{0}\times\left(\frac{1}{16},\frac{31}{16}\right)$. In
view of (\ref{gradient estimate: arrival time}) we have $\textrm{spt}\,\mu_{t}\subset\Gamma_{t}$;
it then follows from the one-sided minimization property of mean-convex
MCF (see 3.9 in \cite{W1} and Remark 2.5 in \cite{HK}) that $\mu_{t}$
is of unit density. Thus, we have $\mu_{t}\leq\mathcal{H}^{n}\lfloor\Gamma_{t}$.
Basing on the curvature estimate for MCF (see Theorem 2.1 in \cite{G1}
for the modification of \cite{W3} and Theorem 5.6 in \cite{E}),
we would get local smooth estimates for $\left\{ \Gamma_{t}^{\epsilon_{i}}\right\} _{t\in\left(\frac{1}{4}\delta,\frac{7}{4}\delta\right)}$
in $\Omega_{0}\times\left(\frac{1}{4},\frac{7}{4}\right)$ (see Proposition
3.22 in \cite{E}) for $i$ sufficiently large. Upon passing to a
subsequence, the smooth compactness theorem gives that the MCF $\left\{ \Gamma_{t}^{\epsilon_{i}}\right\} _{t\in\left(\frac{1}{2}\delta,\frac{3}{2}\delta\right)}$
converge locally smoothly to a MCF $\left\{ \tilde{\Gamma}_{t}\right\} _{t\in\left(\frac{1}{2}\delta,\frac{3}{2}\delta\right)}$
in $\Omega_{0}\times\left(\frac{1}{2},\frac{3}{2}\right)$. Since
for each $t\in\left(\frac{1}{2}\delta,\frac{3}{2}\delta\right)$,
$\tilde{\Gamma}_{t}$ is an embedded hypersurface contained in the
embedded hypersurface $\Gamma_{t}$, we deduce that $\tilde{\Gamma}_{t}=\Gamma_{t}$
in $\Omega_{0}\times\left(\frac{1}{2},\frac{3}{2}\right)$. Therefore,
$\left\{ \Gamma_{t}^{\epsilon_{i}}\right\} _{t\in\left(\frac{1}{2}\delta,\frac{3}{2}\delta\right)}$
converge locally smoothly to $\left\{ \Gamma_{t}\right\} _{t\in\left(\frac{1}{2}\delta,\frac{3}{2}\delta\right)}$
in $\Omega_{0}\times\left(\frac{1}{2},\frac{3}{2}\right)$. 

By (\ref{gradient estimate: cylindrical LSF}), the infimum mean curvature
of $\left\{ \Gamma_{t}\right\} _{t\in\left(\frac{1}{2}\delta,\frac{3}{2}\delta\right)}$
in $\Omega_{0}\times\left(\frac{1}{2},\frac{3}{2}\right)$ is 
\[
\inf_{\frac{1}{2}\delta<t<\frac{3}{2}\delta}\min_{\Sigma_{t}}H\geq\kappa
\]
by the maximum principle for mean curvature (cf. \cite{M}). By (\ref{gradient estimate: translating soliton})
and (\ref{gradient estimate: approximate mean curvature}), the infimum
mean curvature of $\left\{ \Gamma_{t}^{\epsilon_{i}}\right\} _{t\in\left(\frac{1}{2}\delta,\frac{3}{2}\delta\right)}$
in $\Omega_{0}\times\left(\frac{1}{2},\frac{3}{2}\right)$ is 
\[
\inf_{\frac{1}{2}\delta+\frac{1}{2}\epsilon_{i}<u^{\epsilon_{i}}<\frac{3}{2}\delta+\frac{3}{2}\epsilon_{i}}\frac{1}{\sqrt{\left|\nabla u^{\epsilon_{i}}\right|^{2}+\epsilon_{i}^{2}}}.
\]
By the smooth convergence in the last paragraph, given $\sigma>0$
there exists $i_{\sigma}\in\mathbb{N}$ so that 
\[
\inf_{\frac{1}{2}\delta+\frac{1}{2}\epsilon_{i}<u^{\epsilon_{i}}<\frac{3}{2}\delta+\frac{3}{2}\epsilon_{i}}\frac{1}{\sqrt{\left|\nabla u^{\epsilon_{i}}\right|^{2}+\epsilon^{2}}}\,\geq\,\frac{\kappa}{1+\sigma}\quad\forall\,i\geq i_{\sigma}
\]
In addition, due to (\ref{gradient estimate: uniform convergence})
we may assume that 
\[
\left\{ \frac{1}{2}\delta+\frac{1}{2}\epsilon_{i}<u^{\epsilon_{i}}<\frac{3}{2}\delta+\frac{3}{2}\epsilon_{i}\right\} \,\supset\,\left\{ u=\delta\right\} \,=\,\Sigma_{\delta}\quad\forall\,i\geq i_{\sigma}
\]
and hence 
\[
\min_{\Sigma_{\delta}}\frac{1}{\sqrt{\left|\nabla u^{\epsilon_{i}}\right|^{2}+\epsilon^{2}}}\,\geq\,\frac{\kappa}{1+\sigma}\quad\forall\,i\geq i_{\sigma},
\]
which gives that

\begin{equation}
\max_{\Sigma_{\delta}}\left|\nabla u^{\epsilon_{i}}\right|\,\leq\,\max_{\Sigma_{\delta}}\sqrt{\left|\nabla u^{\epsilon_{i}}\right|^{2}+\epsilon^{2}}\,\leq\,\frac{1+\sigma}{\kappa}\quad\forall\,i\geq i_{\sigma}.\label{gradient estimate: approximate gradient}
\end{equation}

On the other hand, differentiating (\ref{gradient estimate: elliptic regularization})
with respect to $x^{k}$ following by multiplying $\partial_{k}u^{\epsilon}$
yields
\[
-a_{\epsilon}^{ij}\left(\nabla u^{\epsilon}\right)\,\partial_{ij}\left(\left|\nabla u^{\epsilon}\right|^{2}\right)\,-\,\frac{\partial a_{\epsilon}^{ij}}{\partial\xi^{l}}\left(\nabla u^{\epsilon}\right)\,\partial_{ij}u^{\epsilon}\,\partial_{l}\left(\left|\nabla u^{\epsilon}\right|^{2}\right)\,=\,-2a_{\epsilon}^{ij}\left(\nabla u^{\epsilon}\right)\,\partial_{ik}u^{\epsilon}\,\partial_{jk}u^{\epsilon}\,\leq\,0,
\]
where $a_{\epsilon}^{ij}\left(\xi\right)=\mathrm{\delta}^{ij}-\frac{\xi{}^{i}\varotimes\xi^{j}}{\left|\xi\right|^{2}+\epsilon^{2}}$.
It follows from the maximum principle that 
\begin{equation}
\max_{\bar{\Omega}_{\delta}}\left|\nabla u^{\epsilon}\right|=\max_{\partial\Omega_{\delta}=\Sigma_{\delta}}\left|\nabla u^{\epsilon}\right|\label{gradient estimate: maximum principle}
\end{equation}
Thus, combining (\ref{gradient estimate: approximate gradient}) with
(\ref{gradient estimate: maximum principle}) gives 
\begin{equation}
\max_{\bar{\Omega}_{\delta}}\left|\nabla u^{\epsilon_{i}}\right|\leq\frac{1+\sigma}{\kappa}\quad\forall\,i\geq i_{\sigma}.\label{gradient estimate: gradient}
\end{equation}
Finally, given a point $p\in\Omega_{0}$, we have the following three
cases to consider:

\uline{Case 1}: $p\in\Omega_{0}\setminus\Omega_{\delta}$. Then
$p$ must be a regular point of the flow since $u\left(p\right)\in\left(0,T_{1}\right)$
is a regular time. Moreover, since $\left\{ \Sigma_{t}\right\} _{t\in\left[0,T_{1}\right)}$
is a MCF, the maximum principle for the mean curvature gives 
\[
H\left(p\right)=\frac{1}{\left|\nabla u\left(p\right)\right|}\geq\kappa
\]
(cf. \cite{M}). So $\left|\nabla u\left(p\right)\right|\leq\kappa^{-1}$.

\uline{Case 2}: $p\in\Omega_{\delta}$ is a singular point of the
flow. In this case we have $\nabla u\left(p\right)=0$. 

\uline{Case 3}: $p\in\Omega_{\delta}$ is a regular point of the
flow. Let $v=\frac{\nabla u\left(p\right)}{\left|\nabla u\left(p\right)\right|}$.
Then by the mean value theorem, (\ref{gradient estimate: uniform convergence}),
and (\ref{gradient estimate: gradient}) we have 
\[
\frac{1}{s}\left|u\left(p+sv\right)-u\left(p\right)\right|\,=\,\lim_{i\rightarrow\infty}\frac{1}{s}\left|u^{\epsilon_{i}}\left(p+sv\right)-u^{\epsilon_{i}}\left(p\right)\right|\,\leq\,\frac{1+\sigma}{\kappa}
\]
for every $0<s<\textrm{dist}\,\left(p,\Sigma_{\delta}\right)$. Letting
$s\searrow0$ gives 
\[
\left|\nabla u\left(p\right)\right|=\left|\nabla u\left(p\right)\cdot v\right|\leq\frac{1+\sigma}{\kappa}
\]
Since $\sigma>0$ is arbitrary, we have $\left|\nabla u\left(p\right)\right|\leq\kappa^{-1}$.
\end{proof}
Next, we will be devoted to the smooth estimates for $\left\{ \Sigma_{t}\right\} _{t>\dot{T}}$
on its regular part. Before that, we have the following simple observation.
\begin{lem}
\label{closedness of singular set}The singular set $\mathcal{S}$
of the flow $\left\{ \Sigma_{t}\right\} $ is a closed set (and hence
is compact).
\end{lem}

\begin{proof}
Given a sequence $\left\{ p_{i}\right\} _{i\in\mathbb{N}}\subset\mathcal{S}$
tending to a point $p\in\mathbb{R}^{n}$, note that by Proposition
\ref{time of smooth existence}, we have $p_{i}\in\Omega_{\dot{T}}$
(i.e., $u\left(p_{i}\right)>\dot{T}$) for every $i\in\mathbb{N}$.
Thus, $p\in\bar{\Omega}_{\dot{T}}$. 

On the other hand, recall that $\mathcal{S}$ is the set of critical
points of $u$. So $\nabla u\left(p_{i}\right)=0$ for every $i$.
By the continuity of $\nabla u$ on $\bar{\Omega}_{\dot{T}}$ (cf.
\cite{CM4}), we obtain 
\[
\nabla u\left(p\right)=\lim_{i\rightarrow\infty}\nabla u\left(p_{i}\right)=0.
\]
Therefore $p\in\mathcal{S}$.
\end{proof}
The set of all regular points in $\Omega_{0}$ is obviously an open
set. In fact, let $p\in\Omega_{0}$ be a regular point of $\left\{ \Sigma_{t}\right\} $,
since $u$ is smooth near $p$ with $\nabla u\left(p\right)\neq0$,
the implicit function theorem yields that there is a neighborhood
of $p$ where the level sets of $u$ are all regular and they jointly
constitute a MCF by (\ref{level set flow}). This motivates the following
definition.
\begin{defn}
\label{smooth scale}Let $p\in\Omega_{0}$ be a regular point of $\left\{ \Sigma_{t}\right\} $.
When a radius $r>0$ is sufficiently small, we have $B_{r}\left(p\right)\subset\Omega_{0}\setminus\mathcal{S}$
\footnote{Note that $u$ is smooth in $B_{r}\left(p\right)$ with $\nabla u\neq0$;
it follows that (\ref{level set flow}) would be satisfied in the
classical sense on $B_{r}\left(p\right)$ and hence the level sets
of $u$ in $B_{r}\left(p\right)$ form a MCF.}and that every level set $\Sigma_{t}=\left\{ u=t\right\} $ in $B_{r}\left(p\right)$
is a $\varepsilon\left(n\right)$-Lipschitz graph over $T_{p}\Sigma_{u\left(p\right)}$,
where $\varepsilon\left(n\right)\in\left(0,1\right)$ is a small constant
as required by Lemma 3.10 in \cite{G1} and Lemma \ref{rolling ball condition}
(see (\ref{rolling ball condition: Hessian})). The supremum of all
such radii is called the \textbf{smooth scale} of $\left\{ \Sigma_{t}\right\} $
at $p$.
\end{defn}

In the following proposition we estimates the smooth scales from below.
\begin{prop}
\label{interior estimate}Let $\iota,\kappa,K,K',K'',\mathfrak{A},\mathfrak{D}$
be positive constants such that 
\[
\textrm{rad}\,\,\Sigma_{0}\geq\iota,\quad\min_{\Sigma_{0}}H\geq\kappa,
\]
\[
\max_{\Sigma_{0}}\left|A\right|\leq K,\quad\max_{\Sigma_{0}}\left|\nabla A\right|\leq K',\quad\max_{\Sigma_{0}}\left|\nabla^{2}A\right|\leq K'',
\]
\[
\mathcal{H}^{n-1}\left(\Sigma_{0}\right)\leq\mathfrak{A},\quad\textrm{diam}\,\Sigma_{0}\leq\mathfrak{D}.
\]
Then there exists $\gamma=\gamma\left(n,\iota,\kappa,K,K',K'',\mathfrak{A},\mathfrak{D}\right)>0$
so that the smooth scale of $\left\{ \Sigma_{t}\right\} $ at a regular
point $p\in\Omega_{\dot{T}}$ is greater than $\gamma\left|\nabla u\left(p\right)\right|$,
where $\dot{T}>0$ is the constant in Proposition \ref{time of smooth existence}.
\end{prop}

\begin{proof}
By Proposition \ref{noncollapsing} and \cite{HK}, $\left\{ \Sigma_{t}\right\} $
is $\alpha$-noncollaspsing, where $\alpha=\alpha\left(n,\imath,\kappa\right)>0$;
by Proposition \ref{entropy} and \cite{H2}, the entropy of $\left\{ \Sigma_{t}\right\} $
is bounded above by $\lambda=\lambda\left(n,K,\mathfrak{A},\mathfrak{D}\right)>0$.

Let $p\in\Omega_{\dot{T}}$ be a regular point of the flow. By Proposition
\ref{moving distance} we have 
\[
B_{\kappa\dot{T}}\left(p\right)\subset\Omega_{0};
\]
additionally, Proposition \ref{gradient estimate} implies 
\[
\kappa\dot{T}\,=\,\kappa\dot{T}\,H\left(p\right)\,\left|\nabla u\left(p\right)\right|\,\geq\,\kappa^{2}\dot{T}\left|\nabla u\left(p\right)\right|,
\]
\[
\sqrt{\dot{T}}\,=\,\sqrt{\dot{T}}\,H\left(p\right)\,\left|\nabla u\left(p\right)\right|\,\geq\,\kappa\sqrt{\dot{T}}\left|\nabla u\left(p\right)\right|,
\]
which gives that 
\[
\min\left\{ \left|\nabla u\left(p\right)\right|,\,\kappa\dot{T},\,\sqrt{\dot{T}}\right\} \,\geq\,\min\left\{ 1,\,\kappa^{2}\dot{T},\,\kappa\sqrt{\dot{T}}\right\} \left|\nabla u\left(p\right)\right|.
\]
It follows from Haslhofer-Kleiner regularity theorem (see Theorem
1.8, Remark 2.7, and Corollary 3.3 in \cite{HK}) that $\left\{ \Sigma_{t}\right\} _{t\in\left(u\left(p\right)-\rho^{2},u\left(p\right)+\rho^{2}\right)}$
is regular in $B_{\rho}\left(p\right)$, where 
\[
\rho=\eta^{2}\left|\nabla u\left(p\right)\right|
\]
with
\[
\eta=\frac{\min\left\{ 1,\kappa^{2}\dot{T},\kappa\sqrt{\dot{T}}\right\} }{C\left(n,\alpha,\lambda\right)}\in\left(0,1\right),
\]
and that 
\begin{equation}
\sup_{t\in\left(u\left(p\right)-\rho^{2},u\left(p\right)+\rho^{2}\right)}\,\,\sup_{\Sigma_{t}\cap B_{\rho}\left(p\right)}\rho\left|A\right|\,\leq\,1,\label{interior estimate: curvature}
\end{equation}
\begin{equation}
\sup_{t\in\left(u\left(p\right)-\rho^{2},u\left(p\right)+\rho^{2}\right)}\,\,\sup_{\Sigma_{t}\cap B_{\rho}\left(p\right)}\rho^{2}\left|\nabla A\right|\,\leq\,1,\label{interior estimate: covariant derivative}
\end{equation}
\begin{equation}
\inf_{t\in\left(u\left(p\right)-\rho^{2},u\left(p\right)+\rho^{2}\right)}\,\,\inf_{\Sigma_{t}\cap B_{\rho}\left(p\right)}H\left|\nabla u\left(p\right)\right|\,\geq\,\frac{1}{C\left(n\right)}.\label{interior estimate: Harnack}
\end{equation}
Let 
\begin{equation}
x\left(t\right):\left(u\left(p\right)-\frac{\rho^{2}}{2\sqrt{n-1}},\,u\left(p\right)+\frac{\rho^{2}}{2\sqrt{n-1}}\right)\rightarrow B_{\frac{\rho}{2}}\left(p\right)\label{interior estimate: curve}
\end{equation}
be the ``trajectory'' of $p$ along the MCF $\left\{ \Sigma_{t}\right\} $,
namely, 
\[
x\left(u\left(p\right)\right)=p,\quad x\left(t\right)\in\Sigma_{t},\quad\frac{d}{dt}x\left(t\right)=\vec{H}\left(x\left(t\right)\right).
\]
Notice that in (\ref{interior estimate: curve}) we have used the
fact that 
\begin{equation}
\left|\frac{d}{dt}x\left(t\right)\right|\,=\,H\left(x\left(t\right)\right)\,\leq\,\sqrt{n-1}\left|A\left(x\left(t\right)\right)\right|\,\leq\,\frac{\sqrt{n-1}}{\rho}.\label{interior estimate: position}
\end{equation}
Let $\varepsilon=\varepsilon\left(n\right)>0$ be the small constant
as specified in Definition \ref{smooth scale}. By (\ref{interior estimate: curvature}),
and the $\alpha$-noncollapsing property of the flow, there exists
$\sigma=\sigma\left(n,\alpha,\varepsilon\right)\in\left(0,\frac{1}{2}\right)$
so that $\Sigma_{t}\cap B_{\sigma\rho}\left(x\left(t\right)\right)$
is a $\frac{\varepsilon}{3}$-Lipschtiz graph over $T_{x\left(t\right)}\Sigma_{t}$.
Using the evolution of the unit normal vector (cf. \cite{M}) and
(\ref{interior estimate: covariant derivative}) we get
\begin{equation}
\left|\frac{d}{dt}N\left(x\left(t\right)\right)\right|\,=\,\left|-\nabla H\left(x\left(t\right)\right)\right|\,\leq\,C\left(n\right)\left|\nabla A\left(x\left(t\right)\right)\right|\,\leq\,\frac{C\left(n\right)}{\rho^{2}}.\label{interior estimate: direction}
\end{equation}
By (\ref{interior estimate: position}) and (\ref{interior estimate: direction}),
there is $\theta=\theta\left(n,\sigma,\varepsilon\right)\in\left(0,1\right)$
so that 
\[
\Sigma_{t}\cap B_{\frac{1}{2}\sigma\rho}\left(p\right)\,\subset\,\Sigma_{t}\cap B_{\sigma\rho}\left(x\left(t\right)\right)
\]
is a $\frac{\varepsilon}{2}$-Lipschtiz graph over  $T_{p}\Sigma_{u\left(p\right)}$
for every $t\in\left(u\left(p\right)-\theta\rho^{2},\,u\left(p\right)+\theta\rho^{2}\right)$
and that 
\begin{equation}
\sqrt{n-1}\,\theta\,\leq\,\frac{\varepsilon\sigma}{4\hat{C}},\label{interior estimate: drift of center}
\end{equation}
where $\hat{C}=\hat{C}\left(n\right)$ is the constant in the proof
of Lemma \ref{rolling ball condition}. 

Moreover, by (\ref{interior estimate: curvature}), (\ref{interior estimate: drift of center}),
and the argument in the proof of Lemma \ref{rolling ball condition}
(with slight modifications), we infer that the domains of the aforementioned
local graphs over $T_{p}\Sigma_{u\left(p\right)}$ all contain a $(n-1)$-dimensional
ball of radius $\frac{\varepsilon\sigma}{4\hat{C}}\rho$ centered
at $p$. For ease of notations, let us assume that $p=0$, $T_{p}\Sigma_{u\left(p\right)}=\mathbb{R}^{n-1}\times\left\{ 0\right\} $,
and $\frac{\nabla u\left(p\right)}{\left|\nabla u\left(p\right)\right|}=\left(0,1\right)$.
Then 
\[
\mathcal{D}\coloneqq\bigsqcup_{t\in\left(u\left(p\right)-\theta\rho^{2},u\left(p\right)+\theta\rho^{2}\right)}\Sigma_{t}\cap B_{\frac{1}{2}\sigma\rho}\left(p\right)
\]
contains the graph $x^{n+1}=f\left(x',t\right)$ satisfying 
\begin{equation}
f\left(0,u\left(p\right)\right)=0,\quad\nabla f\left(0,u\left(p\right)\right)=0,\quad\left|\nabla f\left(x',t\right)\right|\leq\frac{\varepsilon}{2},\label{interior estimate: spacial derivative}
\end{equation}
where 
\begin{equation}
x'=\left(x^{1},\cdots,x^{n-1}\right)\in B_{\frac{\varepsilon\sigma}{4\hat{C}}\rho}^{n-1},\quad t\in\left(u\left(p\right)-\theta\rho^{2},\,u\left(p\right)+\theta\rho^{2}\right).\label{interior estimate: domain}
\end{equation}
It follows from the regularity of the flow and the equation of MCF
(cf. \cite{M}) that 
\begin{equation}
\frac{\partial_{t}f}{\sqrt{1+\left|\nabla f\right|^{2}}}=\nabla\cdot\frac{\nabla f}{\sqrt{1+\left|\nabla f\right|^{2}}}=H.\label{interior estimate: MCF equation}
\end{equation}
Note by (\ref{interior estimate: Harnack}) we have $H\geq\frac{\eta^{2}}{C\left(n\right)}\rho^{-1}$,
which gives that 
\begin{equation}
\partial_{t}f=\sqrt{1+\left|\nabla f\right|^{2}}\,H\,\geq\,H\,\geq\,\frac{\eta^{2}}{C\left(n\right)}\rho^{-1}.\label{interior estimate: temporal derivaitve}
\end{equation}
 By (\ref{interior estimate: spacial derivative}), (\ref{interior estimate: domain}),
and  (\ref{interior estimate: temporal derivaitve}), we conclude
that 
\[
\mathcal{D}\,\supset\,B_{\min\left\{ \frac{\theta\eta^{2}}{2\varepsilon C\left(n\right)},\frac{\varepsilon\sigma}{4\hat{C}}\right\} \rho}^{n-1}\times B_{\frac{\theta\eta^{2}}{2C\left(n\right)}\rho}^{1}\,\supset\,B_{\gamma\left|\nabla u\left(p\right)\right|}^{n},
\]
where $\gamma=\gamma\left(n,\iota,\kappa,K,K',K'',\mathfrak{A},\mathfrak{D}\right)$. 
\end{proof}
All the higher order covariant derivatives of the second fundamental
form of $\Sigma_{t}$ follow immediately by \cite{EH}. 
\begin{prop}
\label{smooth estimate}In Proposition \ref{interior estimate}, by
taking a smaller constant if necessary, we may assume that 
\begin{equation}
\gamma=\gamma\left(n,\iota,\kappa,K,K',K'',\mathfrak{A},\mathfrak{D}\right)\,\leq\,\frac{1}{2}\left(\sqrt{1+4\kappa^{2}\dot{T}}-1\right),\label{smooth estimate: refined constant}
\end{equation}
where $\dot{T}=\dot{T}\left(n,\iota,K,K',K''\right)>0$ is the constant
in Proposition \ref{time of smooth existence}. Then for every regular
point $p\in\Omega_{\dot{T}}$, the flow $\left\{ \Sigma_{t}\right\} $
satisfies 
\[
\sup_{\Sigma_{t}\cap B_{r}\left(p\right)}r^{m+1}\left|\nabla^{m}A\right|\,\leq\,C\left(n,m\right)\quad\forall\,m\geq0,
\]
where $r=\frac{1}{2}\gamma\left|\nabla u\left(p\right)\right|$. 
\end{prop}

\begin{proof}
Fix a regular point $p\in\Omega_{\dot{T}}$ and let $r=\frac{1}{2}\gamma\left|\nabla u\left(p\right)\right|$.
By Proposition \ref{gradient estimate}, for every $x\in B_{2r}\left(p\right)\subset\Omega_{0}$
(see Proposition \ref{interior estimate}) we have 
\[
\left|u\left(x\right)-u\left(p\right)\right|\,\leq\,\kappa^{-1}\left|x-p\right|\,<\,2\kappa^{-1}r\,=\,\kappa^{-1}\gamma\left|\nabla u\left(p\right)\right|,
\]
which yields that 
\begin{equation}
u\left(x\right)\,>\,u\left(p\right)-\kappa^{-1}\gamma\left|\nabla u\left(p\right)\right|\,>\,\dot{T}-\kappa^{-1}\gamma\left|\nabla u\left(p\right)\right|\label{smooth estimate: time}
\end{equation}
\[
=\left(\dot{T}H\left(p\right)-\kappa^{-1}\gamma\right)\left|\nabla u\left(p\right)\right|\,\geq\,\left(\kappa\dot{T}-\kappa^{-1}\gamma\right)\left|\nabla u\left(p\right)\right|\,\geq\,4r^{2}.
\]
Note that the last inequality comes from Proposition \ref{gradient estimate}
and (\ref{smooth estimate: refined constant}). 

Since $\left\{ \Sigma_{t}\right\} _{t\geq0}$ is a MCF in $B_{2r}\left(p\right)$
with 
\[
0<\left(N\cdot N\left(p\right)\right)^{-1}\leq\sqrt{1+\varepsilon^{2}\left(n\right)}
\]
(see Definition \ref{smooth scale} and Proposition \ref{interior estimate}),
the Ecker-Huisken smooth estimates for MCF (cf. \cite{EH}) gives
that 
\[
\sup_{\Sigma_{t}\cap B_{r}\left(p\right)}r^{m+1}\left|\nabla^{m}A\right|\,\leq\,C\left(n,m\right)\quad\forall\,m\in\mathbb{N}\cup\left\{ 0\right\} 
\]
for every $t>4r^{2}$. The conclusion follows immediately by noting
that (\ref{smooth estimate: time}) gives that $\Sigma_{t}\cap B_{2r}\left(p\right)=\emptyset$
for $t\in\left[0,4r^{2}\right]$.
\end{proof}
We are in a position to give an estimate of the Hessian $\nabla^{2}u$,
which is essential to Proposition \ref{C1 compactness} in Section
\ref{smooth convergence off singularities}.
\begin{prop}
\label{Hessian estimate}Let $\iota,\kappa,K,K',K'',\mathfrak{A},\mathfrak{D}$
be positive constants such that 
\[
\textrm{rad}\,\,\Sigma_{0}\geq\iota,\quad\min_{\Sigma_{0}}H\geq\kappa,
\]
\[
\max_{\Sigma_{0}}\left|A\right|\leq K,\quad\max_{\Sigma_{0}}\left|\nabla A\right|\leq K',\quad\max_{\Sigma_{0}}\left|\nabla^{2}A\right|\leq K'',
\]
\[
\mathcal{H}^{n-1}\left(\Sigma_{0}\right)\leq\mathfrak{A},\quad\textrm{diam}\,\Sigma_{0}\leq\mathfrak{D}.
\]
Then the Hessian of the arrival time function $u$ satisfies 
\[
\sup_{\Omega_{0}}\left|\nabla^{2}u\right|\,\leq\,C\left(n,\iota,\kappa,K,K',K'',\mathfrak{A},\mathfrak{D}\right).
\]
\end{prop}

\begin{proof}
If $x\in\Omega_{0}$ is regular point of the flow, say $x\in\Sigma_{t}$
with $t=u\left(x\right)$. Then let $N\left(x\right)=\frac{\nabla u\left(x\right)}{\left|\nabla u\left(x\right)\right|}$
be the unit normal vector and $\left\{ e_{1},\cdots,e_{n-1}\right\} $
be an orthonormal basis of $T_{x}$$\Sigma_{t}$, by the calculations
in \cite{CM4} we have 
\begin{equation}
-\nabla^{2}u\cdot\left(e_{i}\otimes e_{j}\right)\,=\,\frac{A\left(e_{i},e_{j}\right)}{H},\label{Hessian estimate: calculation}
\end{equation}
\[
-\nabla^{2}u\cdot\left(N\otimes e_{i}\right)\,=\,\frac{\nabla H\cdot e_{i}}{H^{2}},
\]
\[
-\nabla^{2}u\cdot\left(N\otimes N\right)\,=\,\frac{\triangle H}{H^{3}}\,+\,\frac{\left|A\right|^{2}}{H^{2}}.
\]
where $A$ denotes the second fundamental form of $\Sigma_{t}$ and
$\triangle$ is the Laplace-Beltrami operator on $\Sigma_{t}$. Let
$\dot{T}$ be as given in Proposition \ref{time of smooth existence}.
Given a point $p\in\Omega_{0}$, below we divide into three cases
to estimate $\left|\nabla^{2}u\left(p\right)\right|$.

\uline{Case 1}: $p\in\Omega_{0}\setminus\Omega_{\dot{T}}$, i.e.,
$u\left(p\right)\in\left(0,\dot{T}\right]$. By (\ref{short time curvature estimate}),
(\ref{short time covariant derivatives of curvature estimate}), and
Proposition \ref{gradient estimate}, we have
\[
\frac{\left|A\left(p\right)\right|}{H\left(p\right)}\,\leq\,\frac{C\left(n,\iota,\kappa,K,K',K''\right)}{\kappa},
\]
\[
\frac{\left|\nabla H\left(p\right)\right|}{H^{2}\left(p\right)}\,\leq\,C\left(n\right)\frac{\left|\nabla A\left(p\right)\right|}{H^{2}\left(p\right)}\,\leq\,\frac{C\left(n,\iota,\kappa,K,K',K''\right)}{\kappa^{2}},
\]
\[
\frac{\left|\triangle H\left(p\right)\right|}{H^{3}\left(p\right)}\,\leq\,C\left(n\right)\frac{\left|\nabla^{2}A\left(p\right)\right|}{H^{3}\left(p\right)}\,\leq\,\frac{C\left(n,\iota,\kappa,K,K',K''\right)}{\kappa^{3}}.
\]
It follows by (\ref{Hessian estimate: calculation}) that $\left|\nabla^{2}u\left(p\right)\right|\leq C\left(n,\iota,\kappa,K,K',K''\right)$.

\uline{Case 2}: $p\in\Omega_{\dot{T}}$ is a regular point. By
Definition \ref{smooth scale}, Proposition \ref{interior estimate},
and Proposition \ref{smooth estimate}, we have 
\[
\frac{\left|A\left(p\right)\right|}{H\left(p\right)}=\left|A\left(p\right)\right|\left|\nabla u\left(p\right)\right|\leq\frac{C\left(n\right)}{\gamma\left(n,\iota,\kappa,K,K',K'',\mathfrak{A},\mathfrak{D}\right)},
\]
\[
\frac{\left|\nabla H\left(p\right)\right|}{H^{2}\left(p\right)}\leq C\left(n\right)\left|\nabla A\left(p\right)\right|\left|\nabla u\left(p\right)\right|^{2}\leq\frac{C\left(n\right)}{\gamma^{2}\left(n,\iota,\kappa,K,K',K'',\mathfrak{A},\mathfrak{D}\right)},
\]
\[
\frac{\left|\triangle H\left(p\right)\right|}{H^{3}\left(p\right)}\leq C\left(n\right)\left|\nabla^{2}A\left(p\right)\right|\left|\nabla u\left(p\right)\right|^{3}\leq\frac{C\left(n\right)}{\gamma^{3}\left(n,\iota,\kappa,K,K',K'',\mathfrak{A},\mathfrak{D}\right)}.
\]
It follows by (\ref{Hessian estimate: calculation}) that $\left|\nabla^{2}u\left(p\right)\right|\leq C\left(n,\iota,\kappa,K,K',K'',\mathfrak{A},\mathfrak{D}\right)$.

\uline{Case 3}: $p\in\Omega_{\dot{T}}$ is a singular point. By
\cite{CM4}, 
\[
-\nabla^{2}u\left(p\right)=\frac{1}{n-1}I_{n}
\]
if $p$ is a round point and 
\[
-\nabla^{2}u\left(p\right)=\frac{1}{n-2}\mathcal{O}\left(\begin{array}{cc}
I_{n-1} & 0\\
0 & 0
\end{array}\right)\mathcal{O}^{-1},
\]
if $p$ is a cylindrical point, where 
\[
\mathcal{O}=\left[N\left(p\right),e_{1},\cdots,e_{n-1}\right]
\]
is the orthogonal matrix such that $e_{n}$ is an unit vector along
the axial direction of the tangent cylinder at $p$ and $\left\{ e_{1},\cdots,e_{n-1}\right\} $
is an orthonormal basis of the orthogonal complement of $\textrm{span}\left\{ N\left(p\right),e_{n-1}.\right\} $.
In either case, $\left|\nabla^{2}u\left(p\right)\right|\leq C\left(n\right).$
\end{proof}
The following corollary is based on Lemma 3.10 in \cite{G1}, Proposition
\ref{smooth estimate}, and Proposition \ref{Hessian estimate}. It
will be used in Lemma \ref{local convergence} in Section \ref{smooth convergence off singularities}. 
\begin{cor}
\label{local graph estimate}Let $p\in\Omega_{\dot{T}}$ be a regular
point of the flow $\left\{ \Sigma_{t}\right\} $ and choose 
\begin{equation}
r\in\left[\frac{1}{4}\gamma\left|\nabla u\left(p\right)\right|,\,\frac{1}{2}\gamma\left|\nabla u\left(p\right)\right|\right],\label{local graph estimate: scale}
\end{equation}
where $\gamma>0$ is the constant in Proposition \ref{interior estimate}
and Proposition \ref{smooth estimate}. 

If $\Sigma_{t}\cap B_{2r}\left(p\right)$ is a $2\varepsilon\left(n\right)$-Lipschitz
graph of $f\left(\cdot,t\right)$ over some hyperplane $\Pi$ (which
can be chosen as, but not restricted to, $T_{p}\Sigma_{u\left(p\right)}$
by Proposition \ref{interior estimate} and (\ref{local graph estimate: scale})).
Then the graph of $f\left(\cdot,t\right)$ restricted in $B_{r}\left(p\right)$
would satisfy
\begin{equation}
r^{m+2l-1}\left|\partial_{t}^{l}\nabla^{m}f\right|\,\leq\,C\left(n,m,l\right)\quad\textrm{whenever}\,\,m+2l\geq2\label{local graph estimate: derivatives}
\end{equation}
and
\[
r\,\partial_{t}f\,\geq\,\frac{\gamma}{4\Lambda},
\]
where $\Lambda=\Lambda\left(n,\iota,\kappa,K,K',K'',\mathfrak{A},\mathfrak{D}\right)>0$
is a constant.
\end{cor}

\begin{proof}
By Lemma 3.10 in \cite{G1} and Proposition \ref{smooth estimate},
we have 
\begin{equation}
r^{m-1}\left|\nabla^{m}f\right|\leq C\left(n,m\right)\quad\forall\,m\geq2.\label{local graph estimate: spacial derivatives}
\end{equation}
Since $\left\{ \Sigma_{t}\right\} $ is regular and graphical in $B_{2r}\left(p\right)$
by Definition \ref{smooth scale}, Proposition \ref{interior estimate},
and (\ref{local graph estimate: scale}), the function $f$ would
satisfy 
\begin{equation}
\partial_{t}f=\sqrt{1+\left|\nabla f\right|^{2}}\,\nabla\cdot\frac{\nabla f}{\sqrt{1+\left|\nabla f\right|^{2}}}=\left(I-\frac{\nabla f\otimes\nabla f}{1+\left|\nabla f\right|^{2}}\right)\cdot\nabla^{2}f\label{local graph estimate: MCF equation}
\end{equation}
(cf. \cite{M}). Using (\ref{local graph estimate: spacial derivatives})
and keeping on differentiating (\ref{local graph estimate: MCF equation})
with respect to $t$, we then obtain (\ref{local graph estimate: derivatives}).

Additionally, by Definition \ref{smooth scale} and Proposition \ref{interior estimate},
we have $B_{2r}\left(p\right)\subset\Omega_{0}\setminus\mathcal{S}$.
It follows from Proposition \ref{Hessian estimate}, and (\ref{local graph estimate: scale})
that
\[
\left|\nabla u\left(x\right)\right|\,\leq\,\left|\nabla u\left(p\right)\right|+C\left(n\right)\left\Vert \nabla^{2}u\right\Vert _{C\left(\bar{B}_{r}\left(p\right)\right)}r\,\leq\,\Lambda\left|\nabla u\left(p\right)\right|\quad\forall\,x\in B_{r}\left(p\right),
\]
where $\Lambda=\Lambda\left(n,\iota,\kappa,K,K',K'',\mathfrak{A},\mathfrak{D}\right)>0$
is a constant. Thus, by (\ref{local graph estimate: scale}) and (\ref{local graph estimate: MCF equation})
we obtain
\[
r\,\partial_{t}f=r\sqrt{1+\left|\nabla f\right|^{2}}\,H\,\geq\,rH\,\geq\,\frac{\gamma}{4}\,\frac{\left|\nabla u\left(p\right)\right|}{\left|\nabla u\right|}\,\geq\,\frac{\gamma}{4\Lambda}.
\]
\end{proof}

\subsection{Smooth convergence off singularities\label{smooth convergence off singularities}}

Let $\iota,\kappa,K,K',K'',\mathfrak{A},\mathfrak{D}$ be positive
constants such that 
\begin{equation}
\textrm{rad}\,\,\Sigma_{0}>\iota,\quad\min_{\Sigma_{0}}H>\kappa,\label{uniform parameters}
\end{equation}
\[
\max_{\Sigma_{0}}\left|A\right|<K,\quad\max_{\Sigma_{0}}\left|\nabla A\right|<K',\quad\max_{\Sigma_{0}}\left|\nabla^{2}A\right|<K'',
\]
\[
\mathcal{H}^{n-1}\left(\Sigma_{0}\right)<\mathfrak{A},\quad\textrm{diam}\,\Sigma_{0}<\mathfrak{D}.
\]
Let $\left\{ \Sigma_{0}^{k}\right\} _{k\in\mathbb{N}}$ be a sequence
of closed connected hypersurfaces tending to $\Sigma_{0}$ in the
$C^{4}$ topology (see Proposition \ref{time of smooth existence}).
\uline{Due to the convergence, we may assume for simplicity that
(\mbox{\ref{uniform parameters}}) hold with \mbox{$\Sigma_{0}^{k}$}
in place of \mbox{$\Sigma_{0}$} for every \mbox{$k\in\mathbb{N}$}
and that Proposition \mbox{\ref{time of smooth existence}} holds
for every \mbox{$k\in\mathbb{N}$}.}

With the uniform estimates from Section \ref{classical theory of MCF}
and Section \ref{interior estimates for mean-convex LSF}, in this
subsection we shall be equipped to prove Theorem \ref{stability theorem}.
The critical part is to show that $u^{k}$ (i.e., the arrival time
function of $\left\{ \Sigma_{t}^{k}\right\} $, see Proposition \ref{time of smooth existence})
converges locally in the $C^{1}$ topology to $u$ (see Proposition
\ref{C0 compactness} and Proposition \ref{C1 compactness}). The
convergence signifies that if $p$ is a regular point of $\left\{ \Sigma_{t}\right\} $,
because of
\[
\lim_{k\rightarrow\infty}\left|\nabla u^{k}\left(p\right)\right|\,=\,\left|\nabla u\left(p\right)\right|
\]
and that the smooth scale of $\left\{ \Sigma_{t}^{k}\right\} $ at
$p$ depends on $\left|\nabla u^{k}\left(p\right)\right|$ (see Proposition
\ref{interior estimate}), it is not surprising that $\left\{ \Sigma_{t}^{k}\right\} $
would converge locally smoothly to $\left\{ \Sigma_{t}\right\} $
near $p$ as $k\rightarrow\infty$. This is the core of Corollary
\ref{local convergence}, from which Theorem \ref{stability theorem}
follows.

Let us start with the uniform convergence of the arrival time functions.
Note that $u:\bar{\Omega}_{0}\rightarrow\left[0,\infty\right)$ can
be extended continuously to $\mathbb{R}^{n}$ simply by interpreting
it to be zero on $\mathbb{R}^{n}\setminus\bar{\Omega}_{0}$. The same
is true for every $u^{k}:\bar{\Omega}_{0}^{k}\rightarrow\left[0,\infty\right)$.
The following proposition is based on Proposition \ref{gradient estimate}
and the uniqueness of viscosity solutions to the Dirichlet problem
(\ref{level set flow}) (cf. \cite{ES}).
\begin{prop}
\label{C0 compactness}Upon identifying the arrival time function
$u^{k}$ with its natural continuous extension (by interpreting it
to be zero outside $\bar{\Omega}_{0}^{k}$), $u^{k}$ converges uniformly
to $u$ in $\mathbb{R}^{n}$ as $k\rightarrow\infty$.
\end{prop}

\begin{proof}
Choose $R>0$ large so that $\bar{\Omega}_{0}$ is contained in the
open ball $B_{R}$. For convenience, let us assume that for every
$k\in\mathbb{N}$, $\Sigma_{0}^{k}$ is also contained in $B_{R}$.
Note that by \cite{ES} we have
\begin{equation}
\Sigma_{t}^{k}\,\subset\,B_{\sqrt{R^{2}-2\left(n-1\right)t}}\quad\forall\,t\geq0,\label{C0 compactness: support}
\end{equation}
which yields that
\begin{equation}
\max_{\mathbb{R}^{n}}u^{k}=\max_{\bar{\Omega}_{0}^{k}}u^{k}\leq\frac{R^{2}}{2\left(n-1\right)}.\label{C0 compactness: bound}
\end{equation}
Given $\epsilon\in\left(0,\dot{T}\right)$, where $\dot{T}>0$ is
as given in Proposition \ref{time of smooth existence}, choose $\delta>0$
so small that $\Omega_{\epsilon/2}\subset\Omega_{0}$ is outside the
$4\delta$-tubular-neighborhood of $\Sigma_{0}$. By Proposition \ref{time of smooth existence},
there exists $k_{\epsilon}\in\mathbb{N}$ so that for every $k\geq k_{\epsilon}$,
$\Sigma_{0}^{k}$ is contained in the $\delta$-tubular-neighborhood
of $\Sigma_{0}$ and that $\Omega_{\epsilon/2}^{k}\subset\Omega_{0}$
\footnote{$\Omega_{\epsilon/2}^{k}=\left\{ u^{k}>\frac{\epsilon}{2}\right\} $
is the region bounded by $\Sigma_{\epsilon/2}^{k}$.}is outside the $3\delta$-neighborhood of $\Sigma_{0}$. Then for
each $x_{0}\in\mathbb{R}^{n}$, we claim that 
\begin{equation}
\left|u^{k}\left(x\right)-u^{k}\left(x_{0}\right)\right|\,\leq\,\epsilon\quad\forall\,x\in B_{\min\left\{ \delta,\,\kappa\epsilon\right\} }\left(x_{0}\right),\,\,k\geq k_{\epsilon}.\label{C0 compactness: equicontinuity}
\end{equation}
Let $k\geq k_{\epsilon}$. Below we divide into three cases to verify: 

\uline{Case 1}: $x_{0}$ belongs to the $2\delta$-tubular-neighborhood
of $\Sigma_{0}$ (so $x_{0}\notin\Omega_{\epsilon/2}^{k}$). Then
for every $x\in B_{\delta}\left(x_{0}\right)$, it is within the $3\delta$-tubular-neighborhood
of $\Sigma_{0}$ and hence falls outside $\Omega_{\epsilon/2}^{k}$;
thus, 
\[
\left|u^{k}\left(x\right)-u^{k}\left(x_{0}\right)\right|\,\leq\,u^{k}\left(x\right)+u^{k}\left(x_{0}\right)\,\leq\,\frac{\epsilon}{2}+\frac{\epsilon}{2}=\epsilon.
\]

\uline{Case 2}: $x_{0}\in\Omega_{0}$ is outside the $2\delta$-tubular-neighborhood
of $\Sigma_{0}$ (so $x_{0}\in\Omega_{0}^{k}$). Then for every $x\in B_{\min\left\{ \delta,\kappa\epsilon\right\} }\left(x_{0}\right)$,
it is in $\Omega_{0}$ and outside the $\delta$-tubular-neighborhood
of $\Sigma_{0}$ and so falls within $\Omega_{0}^{k}$. Thus, $B_{\min\left\{ \delta,\kappa\epsilon\right\} }\left(x_{0}\right)\subset\Omega_{0}^{k}$
and for every $x\in B_{\min\left\{ \delta,\kappa\epsilon\right\} }\left(x_{0}\right)$,
by (\ref{uniform parameters}) and Proposition \ref{gradient estimate}
we have 
\[
\left|u^{k}\left(x\right)-u^{k}\left(x_{0}\right)\right|\,\leq\,\kappa^{-1}\left|x-x_{0}\right|\,\leq\,\epsilon.
\]

\uline{Case 3}: $x_{0}\notin\Omega_{0}$ is outside the $2\delta$-tubular-neighborhood
of $\Sigma_{0}$ (so $x_{0}\notin\Omega_{0}^{k}$). For every $x\in B_{\delta}\left(x_{0}\right)$,
it is outside $\Omega_{0}$ and the $\delta$-tubular-neighborhood
of $\Sigma_{0}$ and hence falls outside $\Omega_{0}^{k}$; thus,
\[
\left|u^{k}\left(x\right)-u^{k}\left(x_{0}\right)\right|=0.
\]

By (\ref{C0 compactness: support}), (\ref{C0 compactness: bound}),
and (\ref{C0 compactness: equicontinuity}), Arzelà-Ascoli compactness
theorem implies that for any subsequence of $\left\{ u^{k}\right\} _{k\in\mathbb{N}}$
there would exist a further subsequence that converges uniformly to
some continuous function $\hat{u}:\mathbb{R}^{n}\rightarrow\left[0,\infty\right)$.
Notice that $\hat{u}$ vanishes on $\mathbb{R}^{n}\setminus\bar{\Omega}_{0}$,
which (by the continuity of $\hat{u}$) implies that $\left.\hat{u}\right|_{\Sigma_{0}}=0$.
Moreover, by the argument in the proof of Theorem 2.7 in \cite{ES},
$\hat{u}$ would satisfy the equation (\ref{level set flow}) in the
viscosity sense on $\Omega_{0}$. Thus, by the uniqueness of viscosity
solutions to the Dirichlet problem of (\ref{level set flow}) (see
Theorem 7.5 in \cite{ES}), we deduce that $\hat{u}=u$. 

By the last paragraph, any subsequence of $\left\{ u^{k}\right\} _{k\in\mathbb{N}}$
has a further subsequence converging uniformly to $u$. Therefore,
the whole sequence $\left\{ u^{k}\right\} _{k\in\mathbb{N}}$ must
converge uniformly to $u$.
\end{proof}
We immediately have the following corollary. 
\begin{cor}
\label{extinction time}Let $T_{ext}^{k}$ be the time of extinction
of $\left\{ \Sigma_{t}^{k}\right\} $. Then we have 
\[
\lim_{k\rightarrow\infty}T_{ext}^{k}=\lim_{k\rightarrow\infty}\max_{\Omega_{0}^{k}}u^{k}=\max_{\Omega_{0}}u=T_{ext}.
\]
\end{cor}

Given the convergence of the arrival time functions in Proposition
\ref{C0 compactness}, we can proceed to prove the convergence of
gradients with the help of Proposition \ref{gradient estimate} and
Proposition \ref{Hessian estimate}.
\begin{prop}
\label{C1 compactness}$\nabla u^{k}$ converges locally uniformly
to $\nabla u$ in $\Omega_{0}$ as $k\rightarrow\infty$. 
\end{prop}

\begin{proof}
Let $\mathcal{B}$ be any open ball in $\mathbb{R}^{n}$ such that
$\bar{\mathcal{B}}\subset\Omega_{0}$. Choose $k_{\mathcal{B}}\in\mathbb{N}$
sufficiently large so that $\bar{\mathcal{B}}\subset\Omega_{0}^{k}$
for every $k\geq k_{\mathcal{B}}$. By Proposition \ref{gradient estimate}
and Proposition \ref{Hessian estimate}, any subsequence of $\left\{ \left.\nabla u^{k}\right|_{\bar{\mathcal{B}}}\right\} _{k\geq k_{\mathcal{B}}}$
has a further subsequence, say $\left\{ \left.\nabla u^{k_{j}}\right|_{\bar{\mathcal{B}}}\right\} _{j\in\mathbb{N}},$
that converges uniformly to some continuous vector-valued function
$F=\left(F^{1},\cdots,F^{n}\right)$ on $\bar{\mathcal{B}}$ by Arzelà-Ascoli
compactness theorem. Note that for every function $\zeta\in C_{c}^{1}\left(\mathcal{B}\right)$,
by Proposition \ref{C0 compactness} there holds
\[
\int_{\mathcal{B}}F^{i}\,\zeta\,dx=\lim_{j\rightarrow\infty}\int_{\mathcal{B}}\partial_{i}u^{k_{j}}\,\zeta\,dx=-\lim_{j\rightarrow\infty}\int_{\mathcal{B}}u^{k_{j}}\,\partial_{i}\zeta\,dx
\]
\[
=-\int_{\mathcal{B}}u\,\partial_{i}\zeta\,dx=\int_{\mathcal{B}}\partial_{i}u\,\zeta\,dx\quad\forall\,i\in\left\{ 1,\cdots,n\right\} ,
\]
we deduce that $F=\left.\nabla u\right|_{\mathcal{\bar{B}}}$. Therefore,
the whole sequence $\left\{ \left.\nabla u^{k}\right|_{\bar{\mathcal{B}}}\right\} _{k\geq k_{\mathcal{B}}}$
must converges uniformly to $\left.\nabla u\right|_{\bar{\mathcal{B}}}$.
\end{proof}
What follow are two corollaries concerning the singular sets and singular
times of $\left\{ \Sigma_{t}^{k}\right\} $, respectively. Note that
by Proposition \ref{time of smooth existence} and Proposition \ref{moving distance},
we may assume that the singular set of the mean-convex LSF $\left\{ \Sigma_{t}^{k}\right\} $
is contained in $\Omega_{\dot{T}}$ when $k$ is large.
\begin{cor}
\label{regular space}Given an open set $U$ strictly contained in
$\Omega_{0}\setminus\mathcal{S}$, there is $k_{U}\in\mathbb{N}$
so that for every $k\geq k_{U}$, the flow $\left\{ \Sigma_{t}^{k}\right\} $
is regular in $U$.

As an illustration, if $U=\Omega_{\dot{T}}\setminus\overline{\mathcal{S}^{\delta}}$
with $\delta>0$, where 
\[
\mathcal{S}^{\delta}\coloneqq\left\{ x:\textrm{dist}\left(x,\mathcal{S}\right)<\delta\right\} ,
\]
then for every $k\geq k_{U}$, the singular set of $\left\{ \Sigma_{t}^{k}\right\} $
would be contained in $\overline{\mathcal{S}^{\delta}}$.
\end{cor}

\begin{proof}
Let $U$ be an open set such that $\bar{U}\subset\Omega_{0}\setminus\mathcal{S}$.
Firstly, choose $k_{0}\in\mathbb{N}$ sufficiently large so that $\bar{U}\subset\Omega_{0}^{k}$
for $k\geq k_{0}$. It follows from Proposition \ref{C1 compactness}
that there is $k_{U}\geq k_{0}$ so that 
\[
\min_{\bar{U}}\left|\nabla u^{k}\right|\,\geq\,\frac{1}{2}\min_{\bar{U}}\left|\nabla u\right|\,>\,0,
\]
which yields that the flow $\left\{ \Sigma_{t}^{k}\right\} $ is regular
in $U$.
\end{proof}
\begin{cor}
\label{regular time}Suppose that the flow $\left\{ \Sigma_{t}\right\} $
is regular during $t\in\left[a,b\right]$, where $0<a<b<T_{ext}$,
then there exists $k_{\left[a,b\right]}\in\mathbb{N}$ so that for
every $k\geq k_{\left[a,b\right]}$, the flow $\left\{ \Sigma_{t}^{k}\right\} $
is regular during $t\in\left[a,b\right]$.
\end{cor}

\begin{proof}
By Lemma \ref{closedness of singular set} and the continuity of $u$,
the set of singular times is compact (and hence closed). So there
exists 
\[
0<a'<a<b<b'<T_{ext}
\]
such that 
\[
\left\{ a'\leq u\leq b'\right\} \,\subset\,\Omega_{0}\setminus\mathcal{S}.
\]
Then $U\coloneqq\left\{ a'<u<b'\right\} $ is an open set strictly
contained in $\Omega_{0}\setminus\mathcal{S}$. By Proposition \ref{C0 compactness},
there is $k_{0}\in\mathbb{N}$ so that for $k\geq k_{0}$, 
\[
\left\{ a\leq u^{k}\leq b\right\} \,\subset\,\left\{ a'<u<b'\right\} \,=\,U.
\]
It follows from Corollary \ref{regular space} that there exists $k_{U}\geq k_{0}$
so that for every $k\geq k_{U}$, the flow $\left\{ \Sigma_{t}^{k}\right\} $
is regular in $U$ and so $\left\{ \Sigma_{t}^{k}\right\} $ is regular
during $t\in\left[a,b\right]$.
\end{proof}
On account of the uniform estimates for the flows (see Proposition
\ref{interior estimate}, Proposition \ref{smooth estimate}, and
Corollary \ref{local graph estimate}) and the local $C^{1}$ convergence
of the arrival time functions (see Proposition \ref{C0 compactness}
and Proposition \ref{C1 compactness}), we are now able to prove the
following lemma, which is the very essence of Theorem \ref{stability theorem}.
\begin{lem}
\label{local convergence}The flow $\left\{ \Sigma_{t}^{k}\right\} $
converges locally smoothly to $\left\{ \Sigma_{t}\right\} $ in $\left(\Omega_{\dot{T}}\setminus\mathcal{S}\right)\times\left(0,\infty\right)$
as $k\rightarrow\infty$.
\end{lem}

\begin{proof}
Fix $p\in\Omega_{\dot{T}}\setminus\mathcal{S}$. By Proposition \ref{C0 compactness}
and Proposition \ref{C1 compactness} we have 
\begin{equation}
\lim_{k\rightarrow\infty}u^{k}\left(p\right)=u\left(p\right)>\dot{T}\label{local convergence: time}
\end{equation}
\begin{equation}
\lim_{k\rightarrow\infty}\nabla u^{k}\left(p\right)=\nabla u\left(p\right)\neq0.\label{local convergence: gradient}
\end{equation}
So when $k$ is large, $p\in\Omega_{\dot{T}}^{k}$ is a regular point
of $\left\{ \Sigma_{t}^{k}\right\} $. Let $r=\frac{1}{3}\gamma\left|\nabla u\left(p\right)\right|$.
By Definition \ref{smooth scale}, Proposition \ref{interior estimate},
and (\ref{local convergence: gradient}), there exists $k_{0}\in\mathbb{N}$
so that when $k\geq k_{0}$, 
\[
\frac{2}{3}\left|\nabla u\left(p\right)\right|\,\leq\,\left|\nabla u^{k}\left(p\right)\right|\,\leq\,\frac{4}{3}\left|\nabla u\left(p\right)\right|
\]
and every level hypersurface $\Sigma_{t}^{k}=\left\{ u^{k}=t\right\} $
in $B_{2r}\left(p\right)$ is a $2\varepsilon\left(n\right)$-Lipschitz
graph of $f^{k}\left(\cdot,t\right)$ over $T_{p}\Sigma_{u\left(p\right)}$.
Also, $\Sigma_{t}=\left\{ u=t\right\} $ in $B_{2r}\left(p\right)$
is certainly a $\varepsilon\left(n\right)$-Lipschitz graph of $f\left(\cdot,t\right)$
over $T_{p}\Sigma_{u\left(p\right)}$. Then by Corollary \ref{local graph estimate}
the graph of $f^{k}\left(\cdot,t\right)$ restricted in $B_{r}\left(p\right)$
would satisfy 
\begin{equation}
r^{m+2l-1}\left|\partial_{t}^{l}\nabla^{m}f^{k}\right|\,\leq\,C\left(n,m,l\right)\quad\textrm{whenever}\,\,m+2l\geq2,\label{local convergence: derivatives}
\end{equation}
\begin{equation}
\left|\nabla f^{k}\right|\,\leq\,2\varepsilon\left(n\right)\,\leq\,2,\label{local convergence: small Lipschitz}
\end{equation}
\begin{equation}
a\,\leq\,r\,\partial_{t}f^{k}\,\leq\,b,\label{local convergence: increasing}
\end{equation}
where $a=\frac{\gamma}{4\Lambda}$ and $b=b\left(n\right)$ are positive
constants. 

For ease of notations, let us assume that $p=0$, $T_{p}\Sigma_{u\left(p\right)}=\mathbb{R}^{n-1}\times\left\{ 0\right\} $,
and $\frac{\nabla u\left(p\right)}{\left|\nabla u\left(p\right)\right|}=\left(0,1\right)$.
So for $k\geq k_{0}$, $\Sigma_{t}^{k}$ in $B_{2r}\left(p\right)$
can be described as $x^{n}=f^{k}\left(x',t\right)$, where $x'=\left(x^{1},\cdots,x^{n-1}\right)$.
For instance, since $0\in\Sigma_{u^{k}\left(0\right)}^{k}$, we have
\begin{equation}
f^{k}\left(0,u^{k}\left(0\right)\right)=0.\label{local convergence: initial point}
\end{equation}
From (\ref{local convergence: small Lipschitz}), (\ref{local convergence: increasing}),
and (\ref{local convergence: initial point}), we deduce that the
domain of $f^{k}\left(\cdot,t\right)$ contains $\bar{B}_{\frac{a}{16b}r}^{n-1}$
for every $t\in\left[u^{k}\left(0\right)-\frac{r^{2}}{4b},\,u^{k}\left(0\right)+\frac{r^{2}}{4b}\right]$;
moreover, 
\begin{equation}
\frac{a}{8b}r\,\leq\,f^{k}\left(x',\,u^{k}\left(0\right)+\frac{r^{2}}{4b}\right)\,\leq\,\frac{r}{2},\label{local convergence: upper and lower bound}
\end{equation}
\[
-\frac{r}{2}\,\leq\,f^{k}\left(x',\,u^{k}\left(0\right)-\frac{r^{2}}{4b}\right)\,\leq\,-\frac{a}{8b}r,
\]
for every $x'\in\bar{B}_{\frac{a}{16b}r}^{n-1}$ . Notice that by
(\ref{local convergence: increasing}) we have
\begin{equation}
f^{k}\left(x',\,u^{k}\left(0\right)-\frac{r^{2}}{4b}\right)\,\leq\,f^{k}\left(x',t\right)\,\leq\,f^{k}\left(x',\,u^{k}\left(0\right)+\frac{r^{2}}{4b}\right).\label{local convergence: bounds}
\end{equation}
for every $x'\in\bar{B}_{\frac{a}{16b}r}^{n-1}$ and $t\in\left[u^{k}\left(0\right)-\frac{r^{2}}{4b},\,u^{k}\left(0\right)+\frac{r^{2}}{4b}\right]$. 

For each $\delta\in\left(0,\frac{1}{4b}\right)$, let
\[
\mathcal{D}_{\delta}\,=\,\bar{B}_{\frac{a}{16b}r}^{n-1}\times\left[u\left(0\right)+\delta r^{2}-\frac{r^{2}}{4b},\,u\left(0\right)-\delta r^{2}+\frac{r^{2}}{4b}\right].
\]
It follows from (\ref{local convergence: time}), (\ref{local convergence: derivatives}),
(\ref{local convergence: small Lipschitz}), (\ref{local convergence: upper and lower bound}),
(\ref{local convergence: bounds}) that any subsequence of $\left\{ \left.f^{k}\right|_{\mathcal{D}_{\delta}}\right\} $
has a further subsequence that converges smoothly to some function
$\tilde{f}$ on $\mathcal{D}_{\delta}$. We claim that $\tilde{f}=\left.f\right|_{\mathcal{D}_{\delta}}$.
To prove that, fix
\[
t_{0}\in\left[u\left(0\right)+\delta r^{2}-\frac{r^{2}}{4b},\,u\left(0\right)-\delta r^{2}+\frac{r^{2}}{4b}\right]
\]
and let $\epsilon>0$ be any sufficiently small number. Note that
(\ref{local convergence: increasing}) holds with $f$ in place of
$f^{k}$, so we have 
\[
\left\{ x^{n}>f\left(x',t_{0}+\epsilon\right)\,:\,x'\in\bar{B}_{\frac{a}{16b}r}^{n-1}\right\} \cap B_{r}\,\subset\,\left\{ u>t_{0}+\epsilon\right\} ,
\]
\[
\left\{ x^{n}<f\left(x',t_{0}-\epsilon\right)\,:\,x'\in\bar{B}_{\frac{a}{16b}r}^{n-1}\right\} \cap B_{r}\,\subset\,\left\{ u<t_{0}-\epsilon\right\} .
\]
By Proposition \ref{C0 compactness}, there is $k_{\epsilon}\in\mathbb{N}$
so that for $k\geq k_{\epsilon}$ we have
\[
\left\{ u>t_{0}+\epsilon\right\} \,\subset\,\left\{ u^{k}>t_{0}\right\} ,\quad\left\{ u<t_{0}-\epsilon\right\} \,\subset\,\left\{ u^{k}<t_{0}\right\} ,
\]
which yields that
\[
f\left(x',t_{0}-\epsilon\right)\,\leq\,f^{k}\left(x',t_{0}\right)\,\leq\,f\left(x',t_{0}+\epsilon\right)\quad\forall\,x'\in\bar{B}_{\frac{a}{16b}r}^{n-1},\,\,k\geq k_{\epsilon}.
\]
By the squeeze theorem, $f^{k}\left(x',t_{0}\right)\rightarrow f\left(x',t_{0}\right)$
for every $\,x'\in\bar{B}_{\frac{a}{16b}r}^{n-1}$, proving the claim.

Therefore, we obtain 
\[
f^{k}\,\stackrel{C^{\infty}}{\rightarrow}\,f\quad\textrm{on}\,\,\mathcal{D}_{\delta}.
\]
Note that (\ref{local convergence: derivatives}), (\ref{local convergence: small Lipschitz}),
(\ref{local convergence: increasing}), (\ref{local convergence: initial point}),
(\ref{local convergence: upper and lower bound}), and (\ref{local convergence: bounds})
hold with $f$ and $u$ in place of $f^{k}$ and $u^{k}$, respectively.
Now choose $\delta>0$ sufficiently small so that 
\[
\frac{a}{12b}r\,\leq\,f\left(x',\,u\left(0\right)-\delta r^{2}+\frac{r^{2}}{4b}\right)\,\leq\,\frac{r}{2},
\]
\[
-\frac{r}{2}\,\leq\,f\left(x',\,u\left(0\right)+\delta r^{2}-\frac{r^{2}}{4b}\right)\,\leq\,-\frac{a}{12b}r
\]
for every $x'\in\bar{B}_{\frac{a}{16b}r}^{n-1}$. Then when $k$ is
large we have 
\[
\frac{a}{16b}r\,\leq\,f^{k}\left(x',\,u\left(0\right)-\delta r^{2}+\frac{r^{2}}{4b}\right)\,\leq\,\frac{r}{2},
\]
\[
-\frac{r}{2}\,\leq\,f^{k}\left(x',\,u\left(0\right)+\delta r^{2}-\frac{r^{2}}{4b}\right)\,\leq\,-\frac{a}{16b}r
\]
for every $x'\in\bar{B}_{\frac{a}{16b}r}^{n-1}$. 

Finally, we conclude that the flow $\left\{ \Sigma_{t}^{k}\right\} $
converges smoothly to the flow $\left\{ \Sigma_{t}\right\} $ in $\bar{B}_{\frac{a}{16b}r}^{n-1}\times\left[-\frac{r}{2},\frac{r}{2}\right]$
for $t\in\left[u\left(0\right)+\delta r^{2}-\frac{r^{2}}{4b},\,u\left(0\right)-\delta r^{2}+\frac{r^{2}}{4b}\right]$
as $k\rightarrow\infty$; when $k$ is large, the flows $\left\{ \Sigma_{t}^{k}\right\} $
and $\left\{ \Sigma_{t}\right\} $ are empty in $\bar{B}_{\frac{a}{16b}r}^{n-1}\times\left[-\frac{a}{16b}r,\frac{a}{16b}r\right]$
for $t\notin\left[u\left(0\right)+\delta r^{2}-\frac{r^{2}}{4b},\,u\left(0\right)-\delta r^{2}+\frac{r^{2}}{4b}\right]$. 
\end{proof}
Now we are in a position to prove Theorem \ref{stability theorem}. 
\begin{proof}
(\textit{of Theorem \ref{stability theorem}}) Note that $\mathcal{S}\subset\Omega_{\dot{T}}$
by Proposition \ref{time of smooth existence}. In view of Lemma \ref{local convergence},
it suffices to show that $\left\{ \Sigma_{t}^{k}\right\} $ converge
locally smoothly to $\left\{ \Sigma_{t}\right\} $ in $\Omega_{0}\setminus\Omega_{\dot{T}}$
as $k\rightarrow\infty$. 

Fix $p\in\Omega_{0}\setminus\Omega_{\dot{T}}$, let 
\[
\tau=\frac{1}{2}u\left(p\right)\,\in\,\left(0,\dot{T}\right)
\]
and choose $r>0$ so that 
\[
B_{2r}\left(p\right)\,\subset\,\left\{ \frac{3}{2}\tau<u<\frac{3}{2}\dot{T}\right\} .
\]
By Proposition \ref{C0 compactness}, when $k$ is large we have
\[
\left\{ \frac{3}{2}\tau<u<\frac{3}{2}\dot{T}\right\} \,\subset\,\left\{ \tau<u^{k}<2\dot{T}\right\} ;
\]
thus, by Corollary \ref{smooth convergence near initial time} the
flow $\left\{ \Sigma_{t}^{k}\right\} $ converges locally smoothly
to $\left\{ \Sigma_{t}\right\} $ in $B_{r}\left(p\right)$ as $k\rightarrow\infty$. 
\end{proof}
As a consequence, Corollary \ref{regular time} can be improved as
follows.
\begin{cor}
\label{closedness during regular times}If the flow $\left\{ \Sigma_{t}\right\} $
is regular during $t\in\left[a,b\right]$, where $0<a<b<T_{ext}$,
then $\left\{ \Sigma_{t}^{k}\right\} _{a\leq t\leq b}$ is a MCF when
$k$ is large and converges smoothly to $\left\{ \Sigma_{t}\right\} _{a\leq t\leq b}$
as $k\rightarrow\infty$.
\end{cor}

\begin{proof}
As in the proof of Corollary \ref{regular time}, choose 
\[
0<a'<a<b<b'<T_{ext}
\]
and $k_{0}\in\mathbb{N}$ such that 
\[
\left\{ a\leq u^{k}\leq b\right\} \,\subset\,\left\{ a'<u<b'\right\} \,\subset\,\Omega_{0}\setminus\mathcal{S}
\]
for every $k\geq k_{0}$. Then the conclusion follows from Theorem
\ref{stability theorem} that the flow $\left\{ \Sigma_{t}^{k}\right\} $
converges locally smoothly to $\left\{ \Sigma_{t}^{k}\right\} $ in
$\left\{ a'\leq u\leq b'\right\} $.
\end{proof}

\section{\uuline{Singular set of two-convex LSF}\label{singular set of two-convex LSF}}

\uline{From now on, the initial hypersurface \mbox{$\Sigma_{0}$}
is assumed to be two-convex.} It follows from \cite{CHN} that the
singular set $\mathcal{S}$ of $\left\{ \Sigma_{t}\right\} $ would
consist of round points and/or cylindrical points (i.e., $1$-cylindrical
points, see \cite{G2} for the definitions). 

In this section we aim to establish Theorem \ref{structure of singular set}.
Firstly, in Section \ref{round point} and Section \ref{cylindrical point}
we apply the theories in \cite{CM2} and \cite{CM3} to investigate
the asymptotic behavior of the flow $\left\{ \Sigma_{t}\right\} $
and the singular set $\mathcal{S}$ near a round point and a cylindrical
point, respectively. Then on the basis of these findings, in Section
\ref{singular set at the first singular time} we study the structure
of the singular set at the first singular time $T_{1}$ under the
assumption that that $T_{1}$ is an isolated singular time. Lastly,
in Section \ref{singular set at subsequent singular times} we give
Proposition \ref{domain after first singular time}, by which we can
apply the result in Section \ref{singular set at the first singular time}
recursively to prove Theorem \ref{structure of singular set}.

\subsection{Round points\label{round point}}

In this subsection we will briefly review the behavior of mean-convex
LSF near a round point and then conclude with Corollary \ref{isolated round point}.

Now let $p$ be a round point of $\left\{ \Sigma_{t}\right\} $. By
Brakke's regularity theorem (cf. \cite{B}, \cite{I1}), in a small
neighborhood $U\subset\Omega_{0}$ of $p$, $\Sigma_{t}$ would be
asymptotically spherical around $p$ for every $t<u\left(p\right)$.
Choose $t_{0}$ less than and sufficiently close to $u\left(p\right)$
such that 
\[
\hat{\Sigma}_{t_{0}}\coloneqq\Sigma_{t_{0}}\cap U
\]
is a  convex closed connected hypersurface. Let $\hat{\Omega}_{t_{0}}$
be the region bounded by $\hat{\Sigma}_{t_{0}}$.\footnote{Note that $\hat{\Omega}_{t_{0}}$ is an open neighborhood of $p$
in $\Omega_{0}$.} As $u$ satisfy (\ref{level set flow}) in the viscosity sense on
$\hat{\Omega}_{t_{0}}$ with the boundary condition $u=t_{0}$ on
$\partial\hat{\Omega}_{t_{0}}=\hat{\Sigma}_{t_{0}}$, it follows from
the existence and uniqueness theorem in \cite{ES} for solutions to
the Dirichlet problem (\ref{level set flow}) that 
\[
\left\{ \Sigma_{t}\cap\hat{\Omega}_{t_{0}}\right\} _{t>t_{0}}
\]
is indeed the mean-convex LSF starting from $\hat{\Sigma}_{t_{0}}$
at time $t_{0}$.

On the other hand, by \cite{H1} the MCF $\left\{ \tilde{\Sigma}_{t}\right\} _{t\geq t_{0}}$
starting from $\hat{\Sigma}_{t_{0}}$ at time $t_{0}$ would shrink
monotonically into a point as $t\nearrow\tilde{T}$, where $\tilde{T}$
is the first singular time of $\left\{ \tilde{\Sigma}_{t}\right\} $.
By \cite{ES}, 
\[
\Sigma_{t}\cap\hat{\Omega}_{t_{0}}=\tilde{\Sigma}_{t}
\]
for every $t\in\left[t_{0},\tilde{T}\right)$. Thus, we infer that
$\Sigma_{t}\cap\hat{\Omega}_{t_{0}}$ shrinks to the round point $p$
as $t\nearrow u\left(p\right)=\tilde{T}$. 

Furthermore, when $u\left(p\right)=T_{1}$ (i.e., the first singular
time of the flow), then by the connectedness of $\Sigma_{t}$ for
$t\in\left[0,T_{1}\right)$, $\Sigma_{t}\cap\hat{\Omega}_{t_{0}}$
must be the only component of $\Sigma_{t}$ for every $t$ less than
and sufficiently close to $T_{1}$. In other words, the whole $\Sigma_{t}$
would be asymptotically spherical around $p$ and shrinks to $p$
as $t\nearrow T_{1}.$
\begin{cor}
\label{isolated round point}The round point $p$ of $\left\{ \Sigma_{t}\right\} $
is an isolated singular point of the flow and a local maximum point
of $u$. Moreover, if $u\left(p\right)$ is the first singular time
of the flow, then $p$ would be the unique singular point of $\left\{ \Sigma_{t}\right\} $
and a global maximum point of $u$; that is to say, $\left\{ \Sigma_{t}\right\} $
would shrink to the point $p$ at time $u\left(p\right)$ and then
vanish.
\end{cor}

\subsection{Cylindrical points\label{cylindrical point}}

We shall begin this subsection with extracting some crucial facts
concerning cylindrical points from \cite{CM2} and \cite{CM3} (see
also Section 2 in \cite{G2}), and then proceed to prove the main
results of this subsection, Proposition \ref{one-sided structure of singular set}
and Corollary \ref{local structure of singular set}, which give the
local structure of the singular set at the first singular time near
a cylindrical point under certain assumptions. The key to establish
Proposition \ref{one-sided structure of singular set} is to prove
two critical lemmas: Lemma \ref{local max scale} and Lemma \ref{saddle scale},
which are based upon the properties of cylindrical points referred
at the beginning of the subsection.

Now let $p$ be a cylindrical point. \uline{For ease of notations
(which is valid only for this subsection), let us perform a rigid
motion to the flow in space and do a translation in time so as to
assume that \mbox{$p$} is the origin, the singular time \mbox{$u\left(p\right)$}
is \mbox{$0$}, and that the tangent flow at \mbox{$0$} is \mbox{$\left\{ \sqrt{-t}\,\mathcal{C}\right\} _{t<0}$}},
where 
\[
\mathcal{C}\coloneqq S_{\sqrt{2\left(n-2\right)}}^{n-2}\times\mathbb{R}.
\]
In this case, it is convenient to adopt the following coordinates:
\[
x=\left(y,z\right)\,\in\,\mathbb{R}^{n-1}\times\mathbb{R}.
\]
By \cite{CM2} (see also \cite{CM3}), given small positive constants
$\phi$ and $\epsilon$, there are $\delta>0$ and $L\geq1$ (depending
on $n$, $\lambda$,\footnote{See Proposition \ref{entropy}.} $\phi$,
$\epsilon$) so that if for some $t_{0}<0$ the rescaled level set
\begin{equation}
\left(\frac{1}{\sqrt{-2\,t_{0}}}\Sigma_{2t_{0}}\right)\cap B_{L}\label{initial approximately cylindrical}
\end{equation}
is $\delta$-close in the $C^{\dot{m}}$ topology to some cylinder
congruent to $\mathcal{C}$ in $B_{L}$, where $\dot{m}\geq2$ is
an absolute constant (depending on $n$, $\lambda$), then for every
$t\in\left[t_{0},0\right)$, the rescaled level set
\[
\left(\frac{1}{\sqrt{-t}}\Sigma_{t}\right)\cap B_{M}
\]
would be $\epsilon$-close in the $C^{\dot{m}}$ topology\footnote{The closeness in the higher order topology would follow from \cite{EH}
and interpolation.} to $\mathcal{C}$ in $B_{M}$, where 
\[
M=\frac{\sqrt{2\left(n-2\right)}}{\sin\phi};
\]
namely, $\frac{1}{\sqrt{-t}}\Sigma_{t}$ in $B_{M}$ can be parametrized
as a normal graph of a function $f_{t}\left(\omega,z\right)$ over
$\mathcal{C}$ whose $C^{\dot{m}}$ norm is less than $\epsilon$,
where $\left(\omega,z\right)\in S^{n-2}\times\mathbb{R}$ are the
cylindrical coordinates of $\mathcal{C}$.\footnote{That is, $\mathcal{C}$ can be typically parametrized as 
\[
x\left(\omega,z\right)=\left(\sqrt{2\left(n-2\right)}\,\omega,\,z\right),\quad\omega\in S^{n-2},\,z\in\mathbb{R}.
\]
} After undoing the rescaling, $\Sigma_{t}$ in $B_{M\sqrt{-t}}$ can
be parametrized as 
\begin{equation}
x_{t}\left(\omega,z\right)=\left(\left(\sqrt{2\left(n-2\right)}+f_{t}\left(\omega,z\right)\right)\sqrt{-t}\,\omega,\,z\right),\quad\omega\in S^{n-2},\,\,\left|z\right|\lesssim M\sqrt{-t}.\label{cylindrical coordinates}
\end{equation}
Note that such $t_{0}<0$ exists due to Brakke's regularity theorem
(cf. \cite{B},  \cite{I1}).

Let 
\[
r=\sqrt{2\left(n-2\right)\left(-t_{0}\right)}.
\]
It follows that in $B_{r}\setminus\mathscr{C}_{\phi}$, where\footnote{$\mathscr{C}_{\phi}$ is a solid double cone with apex $0$, axis
$\left\{ 0\right\} \times\mathbb{R}$ (i.e., the $z$-axis), and angle
$\phi$.} 
\begin{equation}
\mathscr{C}_{\phi}=\left\{ \,\left|y\right|\leq\left|z\right|\tan\phi\right\} ,\label{cone}
\end{equation}
$\Sigma_{t}=\left\{ u=t\right\} $ would be $\epsilon$-close ``relative
to the scale $\sqrt{-t}$'' \footnote{In the sense that after rescaling by the factor $\frac{1}{\sqrt{-t}}$,
the two ``normalized'' hypersurfaces would be $\epsilon$-close
in the $C^{\dot{m}}$ topology.}in the $C^{\dot{m}}$ topology to 
\[
\sqrt{-t}\,\mathcal{C}=\left\{ \left|y\right|=\sqrt{2\left(n-2\right)\left(-t\right)}\right\} 
\]
for every $t<0$. Such a number $r$ is called a $\left(\phi,\epsilon\right)$-\textbf{cylindrical
scale} of the flow $\left\{ \Sigma_{t}\right\} $ at the cylindrical
point $0$. 

Before getting into more details about the cylindrical points from
\cite{CM2} and \cite{CM3}, let us digress for a moment to discuss
the cylindrical points as critical points of the arrival time function
$u$. This is indispensable for Section \ref{types of singular components}.
\begin{cor}
\label{saddle criterion}The cylindrical point $0$ would never be
a local minimum point of $u$ in view of
\begin{equation}
\sup_{B_{r}\setminus\mathscr{C}_{\phi}}\,u\,\leq\,u\left(0\right)=0;\label{saddle criterion: max}
\end{equation}
thus, it is either a local maximum point or a saddle point of $u$. 

As a consequence, the cylindrical point $0$ is a saddle point of
$u$ if and only if it is not a local maximum point of $u$, that
is,  for every $\varepsilon>0$ there is $x\in B_{\varepsilon}$ such
that $u\left(x\right)>u\left(0\right)=0$ (in other words, the cylindrical
point $0$ can be approached by points from the superlevel set $\left\{ u>u\left(0\right)=0\right\} $). 
\end{cor}

Recall that from the asymptotically cylindrical behavior of the flow
$\left\{ \Sigma_{t}\right\} $ in $B_{r}\setminus\mathscr{C}_{\phi}$,
we have
\[
\left\{ u>u\left(0\right)=0\right\} \cap B_{r}\,\subset\,\mathscr{C}_{\phi}\setminus\left\{ 0\right\} .
\]
Note that $\mathscr{C}_{\phi}\setminus\left\{ 0\right\} $ has two
components: the upside and the downside. When the cylindrical points
$0$ is a saddle point, depending on whether it can be approached
by points in $\left\{ u>u\left(0\right)=0\right\} $ from only one
side or both sides, we have the following definition:
\begin{defn}
\label{one-sided saddle}When the cylindrical point $0$ is a saddle
point of $u$, we call it a \textbf{one-sided saddle point} of $u$
provided it can be approached by points in $\left\{ u>u\left(0\right)=0\right\} $
from only one side; in other words, it must be a local maximum point
on either the upside or the downside, say the upside, in the sense
that there exists $\varepsilon_{0}>0$ so that $u\left(0\right)=0$
is the maximum value of $u$ on $\bar{B}_{\varepsilon_{0}}^{n-1}\times\left[0,\varepsilon_{0}\right]$. 

When the saddle point $0$ can be approached by points in $\left\{ u>u\left(0\right)=0\right\} $
from both sides, it would be called a \textbf{two-sided saddle point}
of $u$.
\end{defn}

Below we have a remark, which will be used in the discussion of singular
components of the bumpy type in Section \ref{types of singular components}.
\begin{rem}
\label{one-sided isolated local max}Suppose that the cylindrical
point $0$ is a local maximum point of $u$ on the upside, namely,
there exists $\varepsilon_{0}\in\left(0,r\right)$ so that $u\left(0\right)=0$
is the maximum value of $u$ on $\bar{B}_{\varepsilon_{0}}^{n-1}\times\left[0,\varepsilon_{0}\right]$.
By (\ref{saddle criterion: max}) we actually have 
\[
u\leq u\left(0\right)=0\quad\textrm{on}\,\,B_{r}\cap\left\{ 0\leq z\leq\varepsilon_{0}\right\} .
\]
Assume further that there is no other singular points of $\left\{ \Sigma_{t}\right\} $
in $B_{r}\cap\left\{ 0\leq z\leq\varepsilon_{0}\right\} $ except
the cylindrical point $0$. Then we have
\[
u<u\left(0\right)=0\quad\textrm{on}\,\,B_{r}\cap\left\{ 0\leq z<\varepsilon_{0}\right\} .
\]
Because otherwise there would exist $\left(\check{y},\check{z}\right)\in\left(B_{r}\setminus\left\{ 0\right\} \right)\cap\left\{ 0\leq z<\varepsilon_{0}\right\} $
such that $u\left(\check{y},\check{z}\right)=0$. In view of the asymptotically
cylindrical behavior of $\left\{ \Sigma_{t}\right\} $, we infer that
\[
\left(\check{y},\check{z}\right)\in\left(B_{r}\setminus\left\{ 0\right\} \right)\cap\left\{ 0\leq z<\varepsilon_{0}\right\} \cap\mathscr{C}_{\phi}\,\,\subset\,\,B_{r}\cap\left\{ 0<z<\varepsilon_{0}\right\} ,
\]
which yields that $\left(\check{y},\check{z}\right)$ is an interior
local maximum point of $u$ (and hence a critical point of $u$).
This is in contradiction to the assumption that the cylindrical point
$0$ is the only singular point of $\left\{ \Sigma_{t}\right\} $
in $B_{r}\cap\left\{ 0\leq z\leq\varepsilon_{0}\right\} $. Therefore,
(in view of the continuity of $u$ at the cylindrical point $0$)
we have
\begin{equation}
\max_{B_{r}\cap\left\{ z=z_{0}\right\} }u<u\left(0\right)\quad\forall\,z_{0}\in\left(0,\varepsilon_{0}\right);\quad\max_{B_{r}\cap\left\{ z=z_{0}\right\} }u\,\rightarrow u\left(0\right)\quad\textrm{as}\,\,z_{0}\searrow0.\label{one-sided isolated local max: level}
\end{equation}
\end{rem}

Now let us proceed to the further details regarding the cylindrical
points from \cite{CM2} and \cite{CM3}. Firstly, owing to Brakke's
regularity theorem, we can rechoose $t_{0}$ in (\ref{initial approximately cylindrical})
even closer to $0$ so that 
\[
\left(\frac{1}{\sqrt{-2\,t_{0}}}\left\{ u=2\,t_{0}\right\} \right)\cap B_{2L}
\]
is $\frac{\delta}{2}$-close in the $C^{\dot{m}}$ topology to $\mathcal{C}$
in $B_{2L}$. Applying the theorems in \cite{CM2} to every cylindrical
point near $0$ (if any) yields the following: there exists $\rho>0$
(depending on $n$, $\lambda$, $\phi$, $\epsilon$, $\left|t_{0}\right|$,
and the Lipschitz constant of $u$)\footnote{\label{scale of cylinder neighborhood}The intention is to ensure
that for any cylindrical point $q\in\bar{B}_{\rho}^{n-1}\times\left[-\rho,\rho\right]$
there holds 
\[
\frac{\left|q\right|+\sqrt{\left|u\left(q\right)\right|}}{\sqrt{-t_{0}}}
\]
is sufficiently small (depending on $n$, $\dot{m}$, $\delta$).} so that every cylindrical point in $\bar{B}_{\rho}^{n-1}\times\left[-\rho,\rho\right]$
has a uniform $\left(\phi,\epsilon\right)$-cylindrical scale $r=\frac{1}{2}\sqrt{2\left(n-2\right)\left(-t_{0}\right)}$.\footnote{By Footnote \ref{scale of cylinder neighborhood}, $r$ is much larger
than $\rho$, so we may assume that $\bar{B}_{2\rho}^{n-1}\times\left[-2\rho,2\rho\right]\,\subset B_{r}$.}

By \cite{CM3} (see also \cite{CM5}), let $\theta=\theta\left(n\right)\in\left(0,1\right)$
be an appropriately small constant, there would be small constants
$\phi=\phi\left(n\right)$ and $\epsilon=\epsilon\left(n\right)$
so that every cylindrical point $\left(y_{0},z_{0}\right)$ in $\bar{B}_{\rho}^{n-1}\times\left[-\rho,\rho\right]$
has a uniform $\left(\phi\left(n\right),\epsilon\left(n\right)\right)$-cylindrical
scale\footnote{In the rest of the paper, this would be simply referred to as a cylindrical
scale.}and is contained in a graph $y=\psi\left(z\right)$, where 
\[
\psi\,:\,\left[-\rho,\rho\right]\rightarrow\bar{B}_{\rho}^{n-1}
\]
is a $\theta$-Lipschitz function with $\psi\left(0\right)=0$,\footnote{\label{contained in the interior}With the $\theta$-Lipschitz condition
of $\psi$ and $\psi\left(0\right)=0$, we actually have that 
\[
\psi\left(z\right)\in\bar{B}_{\theta\rho}^{n-1}\subset B_{\rho}^{n-1}
\]
for every $z\in\left[-\rho,\rho\right]$.}and the direction of the $1$-dimensional axis of the tangent cylinder
at $\left(y_{0},z_{0}\right)$ would be $\theta$-close to the direction
of the $z$-axis. 

Moreover, in case there exist $-\rho\leq a<b\leq\rho$ so that every
point on the curve segment 
\begin{equation}
\left\{ \left(y,z\right)\,:\,y=\psi\left(z\right),\,z\in\left[a,b\right]\right\} \label{singular curve}
\end{equation}
is a cylindrical point,\footnote{By the last paragraph, it is equivalent to say that for every $z\in\left[a,b\right]$,
there exists a cylindrical point in $B_{\rho}^{n-1}\times\left\{ z\right\} $.} then $\psi\in C^{1}\left[a,b\right]$ and that there are no other
singular points in $\bar{B}_{\rho}^{n-1}\times\left[a,b\right]$.
Furthermore, in this case the tangent line to the $C^{1}$ curve segment
at each point would be the $1$-dimensional axis of the tangent cylinder
at the corresponding point. Notice that 
\[
u\left(\psi\left(z\right),z\right)=u\left(0\right)\quad\forall\,z\in\left[a,b\right]
\]
since $\nabla u$ vanishes along the curve (\ref{singular curve}).
Consequently, $u\left(0\right)=0$ is the unique singular time of
the flow in $\bar{B}_{\rho}^{n-1}\times\left[a,b\right]$.
\begin{rem}
\label{one-sided local max}In the above scenario, for each $z_{0}\in\left[a,b\right]$,
the flow is asymptotically cylindrical in $B_{r}\left(\psi\left(z_{0}\right),z_{0}\right)\setminus\mathscr{C}_{\phi}^{z_{0}}$,
where $\mathscr{C}_{\phi}^{z_{0}}$ is the solid cone with apex $\left(\psi\left(z_{0}\right),z_{0}\right)$,
axis aligned with the axis of the tangent cylinder at $\left(\psi\left(z_{0}\right),z_{0}\right)$
(whose direction is $\theta$-close to the direction of the $z$-axis),
and angle $\phi$. Particularly, we deduce that
\[
\max_{\bar{B}_{\rho}^{n-1}}\,u\left(\cdot,z_{0}\right)=u\left(0\right)=0\quad\forall\,z_{0}\in\left[a,b\right].
\]
Therefore, $u\left(0\right)=0$ is the maximum value of $u$ on $\bar{B}_{\rho}^{n-1}\times\left[a,b\right].$
In the case where $a=0$ and $b\in\left(0,\rho\right]$, the cylindrical
point $0$ would be either a local maximum point or a one-sided saddle
point of $u$ such that it is a local maximum point on the upside
(see Definition \ref{one-sided saddle}).
\end{rem}

Next, we are going to prove two critical lemmas, Lemma \ref{local max scale}
and Lemma \ref{saddle scale}, from which the main results of this
subsection (Proposition \ref{one-sided structure of singular set}
and Corollary \ref{local structure of singular set}) follow. The
intuition behind Lemma \ref{local max scale} is that the ``neck''
between two cylindrical points should not break before the first singular
time.
\begin{lem}
\label{local max scale}Let the notations be as stated in the above.
Note that the parameters $\phi$, $\epsilon$, $\theta$ are small
and that the scale $r\sim\sqrt{-t_{0}}\gg\rho$. 

Suppose that $u\left(0\right)=0$ is the first singular time of the
flow $\left\{ \Sigma_{t}\right\} $ and that the cylindrical point
$0$ is a maximum point of $u$ in $\bar{B}_{\varrho}^{n-1}\times\left[0,\varrho\right]$
for some $\varrho\in\left(0,\rho\right]$. If there exists another
cylindrical point $\left(y_{*},z_{*}\right)\in\bar{B}_{\varrho}^{n-1}\times\left(0,\varrho\right]$,
then the singular set of $\left\{ \Sigma_{t}\right\} $ in $\bar{B}_{\varrho}^{n-1}\times\left[0,z_{*}\right]$
would be a $C^{1}$ embedded curve that is composed of cylindrical
points and is given by 
\[
\left\{ \left(y,z\right)\,:\,y=\psi\left(z\right),\,z\in\left[0,z_{*}\right]\right\} \,\subset\,\left\{ u=0\right\} .
\]
\end{lem}

\begin{proof}
First of all, note that the hypotheses implies that $t=0$ would be
the unique singular time of the flow in $\bar{B}_{\varrho}^{n-1}\times\left[0,\varrho\right]$;
particularly, $u\left(y_{*},z_{*}\right)=0$. In addition, in view
of the paragraphs  between Remark \ref{one-sided isolated local max}
and Remark \ref{one-sided local max}, to prove the lemma, it suffices
to show that for every $z\in\left(0,z_{*}\right)$ there exists a
cylindrical point in $B_{\varrho}^{n-1}\times\left\{ z\right\} $. 

Observe that if $u\left(y_{0},z_{0}\right)=0$ for some $\left(y_{0},z_{0}\right)\in B_{\varrho}^{n-1}\times\left(0,z_{*}\right)$,
then it must be a local maximum point of $u$ since $u\left(0\right)=0$
is the maximum value of $u$ on $\bar{B}_{\varrho}^{n-1}\times\left[0,\varrho\right]$;
hence $\left(y_{0},z_{0}\right)$ would be a singular point of the
flow, which by Corollary \ref{isolated round point} is indeed a cylindrical
point. Thus, our goal is to show that for every $z\in\left(0,z_{*}\right)$
there exists a point in $B_{\varrho}^{n-1}\times\left\{ z\right\} $
at which $u$ vanishes. 

Next, since $\Sigma_{t}$ approximates to the cylinder $\sqrt{-t}\,\mathcal{C}$
in $B_{r}\setminus\mathscr{C}_{\phi}$ for $t<0$, we can choose $t_{1}\in\left(t_{0},0\right)$
so that 
\begin{equation}
\Sigma_{t_{1}}\cap\left(\bar{B}_{\varrho}^{n-1}\times\left\{ z\right\} \right)\,\subset\,\left(B_{\varrho}^{n-1}\setminus\bar{B}_{z\tan\phi}^{n-1}\right)\times\left\{ z\right\} \,\,\quad\forall\,z\in\left[-2\varrho,2\varrho\right].\label{local max scale: lateral boundary}
\end{equation}
Let $\Theta$ be the ``tubular'' closed region in $B_{\varrho}^{n-1}\times\left[0,z_{*}\right]$
that is bounded from below, above, and lateral by respectively $z=0$,
$z=z_{*}$, and $\Sigma_{t_{1}}\approx\sqrt{-t_{1}}\,\mathcal{C}$.
Consider the set $\mathfrak{T}$ consisting of $t\in\left[t_{1},0\right)$
for which there exists a continuous curve on $\Sigma_{t}\cap\Theta$
that goes from the bottom $z=0$ to the top $z=z_{*}$. It is clear
that $t_{1}\in\mathfrak{T}$ since $\Sigma_{t_{1}}\approx\sqrt{-t_{1}}\,\mathcal{C}$
in $\Theta$. More generally, notice that we can actually find some
$t_{2}\in\left(t_{1},0\right)$ so that for every $t\in\left[t_{1},t_{2}\right]$,
the condition (\ref{local max scale: lateral boundary}) holds with
$t$ in place of $t_{1}$ (based on the asymptotically cylindrical
behavior of the flow in $B_{r}\setminus\mathscr{C}_{\phi}$), which
yields that $\left[t_{1},t_{2}\right]\subset\mathfrak{T}$. 

Let
\begin{equation}
\hat{T}=\sup\left\{ t\in\left[t_{1},0\right)\,:\,\left[t_{1},t\right]\subset\mathfrak{T}\right\} .\label{local max scale: method of continuity}
\end{equation}
We claim that $\hat{T}=0$. If so, fix $z_{0}\in\left(0,z_{*}\right)$,
then for every $t\in\left[t_{1},0\right)\subset\mathfrak{T}$, as
there exits a continuous path on $\Sigma_{t}\cap\Theta$ going from
the bottom to the top and hence must passing through 
\begin{equation}
\Theta_{z_{0}}\,\coloneqq\,\Theta\cap\left\{ z=z_{0}\right\} ,\label{local max scale: level}
\end{equation}
we can find a point $\left(y_{t},z_{0}\right)\in\Sigma_{t}\cap\Theta_{z_{0}}$,
which gives that $u\left(y_{t},z_{0}\right)=t$. By the compactness
of $\Theta_{z_{0}}$ and the continuity of $u$, we then obtain a
point 
\[
\left(y_{0},z_{0}\right)\,\in\,\Theta_{z_{0}}\,\subset\,B_{\varrho}^{n-1}\times\left\{ z_{0}\right\} 
\]
 satisfying $u\left(y_{0},z_{0}\right)=0$, proving the lemma. In
the rest of the proof, we shall verify the claim. 

Suppose the contrary that $\hat{T}<0$. Let $\hat{\Theta}$ be the
extended region of $\Theta$ by adding two conical layers $\Delta_{0}$
and $\Delta_{*}$ in the bottom and top, respectively. The two layers
are defined as follows: let $\varphi\in\left(0,\phi\right]$ be a
small constant (to be determined); let $\Delta_{0}$ be the closed
region that is bounded from below, above, and lateral by respectively
$\left\{ z=-\left|y\right|\tan\varphi\right\} ,$ $\left\{ z=\left|y\right|\tan\varphi\right\} $,
and $\Sigma_{t_{1}}\approx\sqrt{-t_{1}}\,\mathcal{C}$;\footnote{Because $\Sigma_{t_{1}}\cap\Delta_{0}\,\subset\,\Sigma_{t_{1}}\cap\left(B_{r}\setminus\mathscr{C}_{\phi}\right)$.}
let $\Delta_{*}$ be the closed region that is bounded from below,
above, and lateral by respectively $\left\{ z-z_{*}=-\left|y-y_{*}\right|\tan\varphi\right\} ,$
$\left\{ z-z_{*}=\left|y-y_{*}\right|\tan\varphi\right\} $, and $\Sigma_{t_{1}}\approx\sqrt{-t_{1}}\,\mathcal{C}$.\footnote{Because $\Sigma_{t_{1}}\cap\Delta_{*}\,\subset\,\Sigma_{t_{1}}\cap\left(B_{r}\setminus\mathscr{C}_{\phi}\right)$
by (\ref{local max scale: lateral boundary}).} Now let us choose $\varphi\in\left(0,\phi\right]$ sufficiently small
(depending also on $z_{*}$) so that $\Delta_{0}$ and $\Delta_{*}$
are disjoint.

Recall that both of the cylindrical points $0$ and $\left(y_{*},z_{*}\right)$
have cylindrical scale $r$, so the flow is asymptotically cylindrical
in both $B_{r}\setminus\mathscr{C}_{\phi}$ and $B_{r}\left(y_{*},z_{*}\right)\setminus\mathscr{C}_{\phi}^{*}$,
where $\mathscr{C}_{\phi}^{*}$ is the counterpart of $\mathscr{C}_{\phi}$
for the cylindrical point $\left(y_{*},z_{*}\right)$, namely, $\mathscr{C}_{\phi}^{*}$
is the the solid cone with apex $\left(y_{*},z_{*}\right)$, axis
aligned with the axis of the tangent cylinder at $\left(y_{*},z_{*}\right)$
(whose direction is $\theta$-close to the direction of the $z$-axis),
and angle $\phi$. Particularly, the flow is asymptotically cylindrical
in both $\Delta_{0}$ and $\Delta_{*}$. 

In view of $\Sigma_{t}\cap\Theta_{0}\approx\sqrt{-t}\,\mathcal{C}\cap\Theta_{0}$
for every $t\in\left[t_{1},\hat{T}/2\right]$, and that $\partial\Delta_{0}\setminus\Sigma_{t_{1}}$
\footnote{It is the union of the top and bottom boundaries of $\Delta_{0}$.}is
contained in $z=\pm\left|y\right|\tan\varphi,$ there is a uniform
distance between $\Sigma_{t}\cap\Theta_{0}$ and $\partial\Delta_{0}\setminus\Sigma_{t_{1}}$
for every $t\in\left[t_{1},\hat{T}/2\right]$. The same is true if
we replace $\Theta_{0}$ and $\Delta_{0}$ by respectively $\Theta_{z_{*}}$
and $\Delta_{*}$. Specifically, there is $\sigma>0$ so that 
\begin{equation}
\textrm{dist}\left(\Sigma_{t}\cap\Theta_{0},\,\partial\Delta_{0}\setminus\Sigma_{t_{1}}\right)\,\geq\,\sigma,\quad\textrm{dist}\left(\Sigma_{t}\cap\Theta_{z_{*}},\,\partial\Delta_{*}\setminus\Sigma_{t_{1}}\right)\,\geq\,\sigma\label{local max scale: distance}
\end{equation}
for every $t\in\left[t_{1},\hat{T}/2\right]$.

In addition, the asymptotically cylindrical behavior of the flow in
$\Delta_{0}$ and $\Delta_{*}$ implies that there is not any singular
points in 
\[
\left(\Delta_{0}\cup\Delta_{*}\right)\cap\left\{ t_{1}\leq u\leq\hat{T}/2\right\} .
\]
Also, recall that the only singular time of the flow in 
\[
\Theta\,\subset\,\bar{B}_{\varrho}^{n-1}\times\left[0,\varrho\right]
\]
is $0$. Thus, we infer that $\hat{\Theta}\cap\left\{ t_{1}\leq u\leq\hat{T}/2\right\} $
is a compact set consisting of regular points (so $\nabla u\neq0)$,
which yields that 
\begin{equation}
K\coloneqq\max_{\hat{\Theta}\cap\left\{ t_{1}\leq u\leq\hat{T}/2\right\} }\left|\nabla u\right|^{-1}\label{local max scale: speed}
\end{equation}
is a finite positive number.

Now choose $t_{3}\in\left[t_{2},\hat{T}\right)$ and $t_{4}\in\left(\hat{T},\hat{T}/2\right)$
so that 
\begin{equation}
t_{4}-t_{3}\,<\,\frac{\sigma}{K}.\label{local max scale: time}
\end{equation}
Since $t_{3}\in\mathfrak{T}$, there exists a continuous curve 
\[
\gamma_{t_{3}}\,:\,\left[0,1\right]\rightarrow\Sigma_{t_{3}}\cap\Theta
\]
such that 
\begin{equation}
\gamma_{t_{3}}\left(0\right)\in\Sigma_{t_{3}}\cap\Theta_{0},\quad\gamma_{t_{3}}\left(1\right)\in\Sigma_{t_{3}}\cap\Theta_{z_{*}}.\label{local max scale: initial point}
\end{equation}
Let $\Phi_{\tau}$ be the (local) flow generated by the vector field
$\frac{\nabla u}{\left|\nabla u\right|^{2}}$.\footnote{For a regular point $x$, $\tau\mapsto\Phi_{\tau}\left(x\right)$
is the unique integral curve of the vector field $\frac{\nabla u}{\left|\nabla u\right|^{2}}$
such that $\Phi_{0}\left(x\right)=x$, which has short time existence.
Note that if $u\left(x\right)=t$, then $u\left(\Phi_{\tau}\left(x\right)\right)=t+\tau$.} Because $\gamma_{t_{3}}\left[0,1\right]\subset\Sigma_{t_{3}}\cap\Theta$
is a compact set consisting of regular points (see the last paragraph),
$\Phi_{\tau}$ acts on $\gamma_{t_{3}}\left[0,1\right]$ for $\tau\geq0$
sufficiently small with $\Phi_{\tau}\circ\gamma_{t_{3}}$ being a
continuous curve lying on $\Sigma_{t_{3}+\tau}\cap\,\textrm{int}\,\hat{\Theta}$.
In fact, by noticing that the curve $\Phi_{\tau}\circ\gamma_{t_{3}}\left(s\right)$
will never hit the lateral boundary of $\hat{\Theta}$ (which is contained
in $\Sigma_{t_{1}}$) and taking (\ref{local max scale: distance}),
(\ref{local max scale: speed}), and (\ref{local max scale: time})
into account, we infer that so long as $\tau\in\left[0,t_{4}-t_{3}\right]$,
$\Phi_{\tau}$ can act on $\gamma_{t_{3}}\left[0,1\right]$ with 
\[
\gamma_{t_{3}+\tau}^{\left(0\right)}\coloneqq\Phi_{\tau}\circ\gamma_{t_{3}}\,:\,\left[0,1\right]\rightarrow\,\Sigma_{t_{3}+\tau}\cap\,\textrm{int}\,\hat{\Theta}
\]
being a continuous curve that, in view of (\ref{local max scale: distance}),
(\ref{local max scale: speed}), (\ref{local max scale: time}), and
(\ref{local max scale: initial point}), satisfies
\begin{equation}
\gamma_{t_{3}+\tau}^{\left(0\right)}\left(0\right)\,\in\,\Delta_{0},\quad\gamma_{t_{3}+\tau}^{\left(0\right)}\left(1\right)\,\in\,\Delta_{*}.\label{local max scale: moved initial point}
\end{equation}
To finish the proof, we need to modify $\gamma_{t}^{\left(0\right)}$
for each $t\in\left(t_{3},t_{4}\right]$ so as to make it stay on
$\Sigma_{t}\cap\Theta$ (instead of $\Sigma_{t}\cap\hat{\Theta}$),
start from somewhere on $\Sigma_{t}\cap\Theta_{0}$, and end up on
$\Sigma_{t}\cap\Theta_{z_{*}}$. If so, we would have $\left[t_{1},t_{4}\right]\subset\mathfrak{T}$,
which gives the desired contradiction that $t_{4}>\hat{T}$ (see (\ref{local max scale: method of continuity})). 

The aforementioned modification of $\gamma_{t}^{\left(0\right)}$
for each $t\in\left(t_{3},t_{4}\right]$ will be done in two steps.
In the first step we would deal with the ``bottom'' part. Specifically,
we shall find a continuous map $\mathfrak{I}{}_{t}:\Sigma_{t}\cap\hat{\Theta}\rightarrow\Sigma_{t}\cap\left(\Theta\cup\Delta_{*}\right)$
(as will be seen in the next paragraph) that pushes the portions of
$\gamma_{t}^{\left(0\right)}$ lying on $\Sigma_{t}\cap\Delta_{0}\setminus\Theta$
continuously to $\Sigma_{t}\cap\Delta_{0}\cap\Theta$ and keeps the
rest invariant, so we would obtain a new continuous curve:
\[
\mathfrak{I}{}_{t}\circ\gamma_{t}^{\left(0\right)}\,:\,\left[0,1\right]\rightarrow\,\Sigma_{t}\cap\left(\Theta\cup\Delta_{*}\right).
\]
Notice that $\mathfrak{I}{}_{t}\circ\gamma_{t}^{\left(0\right)}\left(0\right)$
might not be on $\Sigma_{t}\cap$$\Theta_{0}$; instead, by (\ref{local max scale: moved initial point})
we have 
\[
\mathfrak{I}{}_{t}\circ\gamma_{t}^{\left(0\right)}\left(0\right)\,\in\,\Sigma_{t}\cap\Theta\cap\Delta_{0}.
\]
Since $\Sigma_{t}\cap\Theta\cap\Delta_{0}\approx\sqrt{-t}\,\mathcal{C}\cap\Theta\cap\Delta_{0}$,
we can find a curve $\varsigma_{t}:\left[0,1\right]\rightarrow\Sigma_{t}\cap\Theta\cap\Delta_{0}$
so that 
\[
\varsigma_{t}\left(0\right)\,\in\,\Sigma_{t}\cap\Theta_{0},\quad\varsigma\left(1\right)\,=\,\mathfrak{I}{}_{t}\circ\gamma_{t}^{\left(0\right)}\left(0\right).
\]
Then jointing this two paths together gives a new curve
\[
\gamma_{t}^{\left(1\right)}\left(s\right)=\left\{ \begin{array}{c}
\varsigma_{t}\left(2s\right),\quad s\in\left[0,\frac{1}{2}\right]\\
\mathfrak{I}{}_{t}\circ\gamma_{t}^{\left(0\right)}\left(2s-1\right),\quad s\in\left[\frac{1}{2},1\right]
\end{array}\right.
\]
satisfying 
\[
\gamma_{t}^{\left(1\right)}\,:\,\left[0,1\right]\rightarrow\,\Sigma_{t}\cap\left(\Theta\cup\Delta_{*}\right)
\]
\[
\gamma_{t}^{\left(1\right)}\left(0\right)\,\in\,\Sigma_{t}\cap\Theta_{0},\quad\gamma_{t}^{\left(1\right)}\left(1\right)\,\in\,\Sigma_{t}\cap\Delta_{*},
\]
completing the first step. In the second step, we deal with the ``top''
part using essentially the same trick as in the first step, so we
are able to transform $\gamma_{t}^{\left(1\right)}$ into a new curve
$\gamma_{t}^{\left(2\right)}$ satisfying 
\[
\gamma_{t}^{\left(2\right)}\,:\,\left[0,1\right]\rightarrow\,\Sigma_{t}\cap\Theta
\]
\[
\gamma_{t}^{\left(2\right)}\left(0\right)\,\in\,\Sigma_{t}\cap\Theta_{0},\quad\gamma_{t}^{\left(2\right)}\left(1\right)\,\in\,\Sigma_{t}\cap\Theta_{z_{*}}.
\]
Finally, $\gamma_{t}=\gamma_{t}^{\left(2\right)}$ is the desired
continuous curve. 

The last piece to complete the entire proof is to construct the map
\[
\mathfrak{I}{}_{t}\,:\,\Sigma_{t}\cap\hat{\Theta}\rightarrow\,\Sigma_{t}\cap\left(\Theta\cup\Delta_{*}\right)
\]
for each $t\in\left(t_{3},t_{4}\right]$ as mentioned in the last
paragraph. As $\Sigma_{t}\cap\Delta_{0}\approx\sqrt{-t}\,\mathcal{C}\cap\Delta_{0}$,
we can parametrize $\Sigma_{t}\cap\Delta_{0}$ using the cylindrical
coordinates $\left(\omega,z\right)$ defined in (\ref{cylindrical coordinates}).
Then we define a map (between submanifolds)
\[
\mathfrak{I}{}_{t}\,:\,\Sigma_{t}\cap\Delta_{0}\rightarrow\,\Sigma_{t}\cap\Delta_{0}\cap\Theta
\]
in terms of local coordinates as 
\[
\mathfrak{I}{}_{t}\left(\omega,z\right)=\left(\omega,z_{+}\right),
\]
where $z_{+}=\max\left\{ z,0\right\} $. It is not hard to see that
the map $\mathfrak{I}{}_{t}$ can be extended on $\Sigma_{t}\cap\hat{\Theta}$
in a continuous manner such that $\mathfrak{I}{}_{t}=\textrm{id}$
on $\Sigma_{t}\cap\hat{\Theta}\setminus\Delta_{0}$.
\end{proof}
In contrast to Lemma \ref{local max scale}, in which the cylindrical
point $0$ is assumed to be a local maximum point of $u$ on the upside,
in Lemma \ref{saddle scale} we consider the case where the cylindrical
point $0$ is not a local maximum point on the upside. The additional
assumption that $u\left(0\right)=0$ is locally the unique singular
time on the upside is of vital importance for the lemma.
\begin{lem}
\label{saddle scale}Let notations be as defined in Lemma \ref{local max scale}.
Suppose that $0$ is the first singular time of the flow $\left\{ \Sigma_{t}\right\} $
and that the cylindrical point $0$ is not a local maximum point of
$u$ in $\bar{B}_{\rho}^{n-1}\times\left[0,\rho\right]$ in the sense
that for every $\varepsilon>0$, there is $x\in B_{\varepsilon}\cap\left(\bar{B}_{\rho}^{n-1}\times\left[0,\rho\right]\right)$
such that $u\left(x\right)>u\left(0\right)=0$. 

If there exists $\varrho\in\left(0,\rho\right]$ so that $u\left(0\right)=0$
is the unique singular time of the flow in $\bar{B}_{\varrho}^{n-1}\times\left[0,\varrho\right]$
for some $\varrho\in\left(0,\rho\right]$, then the cylindrical point
$0$ would be the unique singular point of the flow in $\bar{B}_{\varrho}^{n-1}\times\left[0,\varrho\right]$.
\end{lem}

\begin{proof}
Suppose the contrary that there is a another singular point $\left(y_{*},z_{*}\right)$
in $\bar{B}_{\varrho}^{n-1}\times\left[0,\varrho\right]$. Note that
$u\left(y_{*},z_{*}\right)=0$ by hypothesis and that $\left(y_{*},z_{*}\right)$
must be a cylindrical point by Corollary \ref{isolated round point}.
Moreover, in view of the asymptotically cylindrical behavior of the
flow in $B_{r}\setminus\mathscr{C}_{\phi}$, we infer that $\left(y_{*},z_{*}\right)\in\mathscr{C}_{\phi}$
with $z_{*}>0$. 

Recall that every cylindrical point in $\bar{B}_{\varrho}^{n-1}\times\left[0,\varrho\right]$
is located on a $\theta$-Lipschitz graph $y=\psi\left(z\right)$
and has a cylindrical scale $r$. It follows that the flow is asymptotically
cylindrical in $B_{r}\left(y_{*},z_{*}\right)\setminus\mathscr{C}_{\phi}^{*}$,
where $\mathscr{C}_{\phi}^{*}$ is the the solid cone with apex $\left(y_{*},z_{*}\right)$,
axis aligned with the axis of the tangent cylinder at $\left(y_{*},z_{*}\right)$
(whose direction is $\theta$-close to the direction of the $z$-axis),
and angle $\phi$. Thus, we obtain
\begin{equation}
\bar{B}_{\varrho}^{n-1}\times\left\{ 0\right\} \,\subset\,\left(B_{r}\setminus\mathscr{C}_{\phi}\right)\cup\left\{ 0\right\} \,\subset\,\left\{ u\leq u\left(0\right)=0\right\} ,\label{saddle scale: bottom}
\end{equation}
\begin{equation}
\bar{B}_{\varrho}^{n-1}\times\left\{ z_{*}\right\} \,\subset\,\left(B_{r}\left(y_{*},z_{*}\right)\setminus\mathscr{C}_{\phi}^{*}\right)\cup\left\{ \left(y_{*},z_{*}\right)\right\} \,\subset\,\left\{ u\leq u\left(y_{*},z_{*}\right)=0\right\} .\label{saddle scale: top}
\end{equation}
Now (using the asymptotically cylindrical behavior of the flow in
$B_{r}\setminus\mathscr{C}_{\phi}$) choose $t_{1}\in\left(t_{0},0\right)$
such that 
\[
\Sigma_{t_{1}}\cap\left(\bar{B}_{\varrho}^{n-1}\times\left\{ z\right\} \right)\,\subset\,\left(B_{\varrho}^{n-1}\setminus\bar{B}_{z\tan\phi}^{n-1}\right)\times\left\{ z\right\} \quad\forall\,z\in\left[0,\varrho\right].
\]
Let $\Theta$ be the closed region in $B_{\varrho}^{n-1}\times\left[0,z_{*}\right]$
that is bounded from below, above, and lateral by respectively $z=0$,
$z=z_{*}$, and $\Sigma_{t_{1}}\approx\sqrt{-t_{1}}\,\mathcal{C}$.
Since $u\left(0\right)=0$ is not a local maximum value of $u$ on
$\bar{B}_{\varrho}^{n-1}\times\left[0,\varrho\right]$, we have
\[
\hat{T}=\max_{\Theta}\,u>0.
\]
Choose $p\in\Theta$ so that $u\left(p\right)=\hat{T}$. In view of
(\ref{saddle scale: bottom}), (\ref{saddle scale: top}), and that
the lateral boundary of $\Theta$ is contained in $\Sigma_{t_{1}}$,
$p$ is not on $\partial\Theta$ and therefore is an interior maximum
point (and hence a critical point) of $u$. Thus, $\hat{T}$ is another
singular time of the flow, contradicting the the hypothesis.
\end{proof}
Before moving on to the primary conclusions of this subsection, let
us make the following remark, which will be used in the discussion
of singular components of the splitting/bumpy type in Section \ref{types of singular components}.
\begin{rem}
\label{flow near saddle} In Lemma \ref{saddle scale}, note that
the cylindrical point $0$ must be either a two-sided saddle point
of $u$ or a one-sided saddle point of $u$ such that it is a local
maximum point on the downside (see Definition \ref{one-sided saddle}).
Because of the asymptotically cylindrical behavior of $\left\{ \Sigma_{t}\right\} $
in $B_{r}\setminus\mathscr{C}_{\phi}$, we have $\Sigma_{\tau}\approx\sqrt{-\tau}\,\mathcal{C}$
in $B_{r}$, where 
\[
\tau\,=\,u\left(0\right)-\frac{r^{2}}{4\left(n-2\right)}=-\frac{r^{2}}{4\left(n-2\right)}.
\]
By the continuity of $u$ at the cylindrical point $0$, we can find
$\mathring{r}\in\left(0,r\right)$ so that 
\begin{equation}
\min_{\mathscr{C}_{\phi}\cap\left\{ 0\leq z\leq\mathring{r}\right\} }u\,>\,\tau.\label{flow near saddle: lower bound in cone}
\end{equation}
Moreover, since the cylindrical point $0$ is not a local maximum
point of $u$ on the top, we can choose $\mathring{\varepsilon}\in\left(0,\mathring{r}\right)$
such that 
\[
\mathring{t}\coloneqq\max_{\hat{\mathcal{T}}\cap\left\{ z=\mathring{\varepsilon}\right\} }u\,>\,u\left(0\right)=0,
\]
where 
\begin{equation}
\hat{\mathcal{T}}\,\coloneqq\,\bar{\Omega}_{\tau}\cap B_{r}\cap\left\{ 0\leq z\leq\mathring{\varepsilon}\right\} \label{flow near saddle: tubular region}
\end{equation}
is a ``tubular'' closed region bounded by $\Sigma_{\tau}$, $\left\{ z=0\right\} $,
and $\left\{ z=\mathring{\varepsilon}\right\} $ from the lateral,
bottom, and top, respectively. By the asymptotically cylindrical behavior
of $\left\{ \Sigma_{t}\right\} $ in $B_{r}\setminus\mathscr{C}_{\phi}$
and the intermediate value theorem, we then deduce that
\begin{equation}
\Sigma_{t}\cap\left\{ z=\mathring{\varepsilon}\right\} \cap\hat{\mathcal{T}}\neq\emptyset\quad\forall\,t\in\left[\tau,\mathring{t}\right],\label{flow near saddle: intersection}
\end{equation}
\begin{equation}
\Sigma_{t}\cap\hat{\mathcal{T}}\,\subset\,\textrm{int}\,\mathscr{C}_{\phi}\quad\forall\,t\in\left(0,\mathring{t}\right].\label{flow near saddle: cone}
\end{equation}
In case where $u\left(0\right)=0$ is the unique singular time of
$\left\{ \Sigma_{t}\right\} $ in $\bar{B}_{r}^{+}$, $\Sigma_{t}\cap\hat{\mathcal{T}}$
would be a hypersurface contained in $\textrm{int}\,\mathscr{C}_{\phi}$
for every $t\in\left(0,\mathring{t}\right]$. 
\end{rem}

Now we are all set to prove the main result of this subsection.
\begin{prop}
\label{one-sided structure of singular set}With the same notations
as in Lemma \ref{local max scale} and Lemma \ref{saddle scale}.
Assume that $u\left(0\right)=0$ is the first singular time of the
flow $\left\{ \Sigma_{t}\right\} $ and also the unique singular time
of the flow in $\bar{B}_{\varrho}^{n-1}\times\left[0,\varrho\right]$
for some $\varrho\in\left(0,\rho\right]$.

Then either the cylindrical point $0$ is an isolated\footnote{It means that $0$ is the unique singular point in $B_{\varepsilon}\cap\left(\bar{B}_{\varrho}^{n-1}\times\left[0,\varrho\right]\right)$
for some $\varepsilon>0$.} singular point in $\bar{B}_{\varrho}^{n-1}\times\left[0,\varrho\right]$,
or there exists $z_{*}\in\left(0,\varrho\right]$ so that the singular
set of the flow in $\bar{B}_{\varrho}^{n-1}\times\left[0,z_{*}\right]$
is a $C^{1}$ embedded curve that is comprised of cylindrical points
and is defined by 
\[
\left\{ \left(y,z\right):y=\psi\left(z\right),\,z\in\left[0,z_{*}\right]\right\} .
\]
\end{prop}

\begin{proof}
If the cylindrical point $0$ is an isolated singular point in $\bar{B}_{\varrho}^{n-1}\times\left[0,\varrho\right]$,
then we are done. So let us assume that $0$ is not an isolated singular
point in $\bar{B}_{\varrho}^{n-1}\times\left[0,\varrho\right]$; particularly,
$0$ is not the only singular points in $\bar{B}_{\varrho}^{n-1}\times\left[0,\varrho\right]$.
So by Lemma \ref{saddle scale} we infer that the cylindrical point
$0$ must be a local maximum point of $u$ in $\bar{B}_{\varrho}^{n-1}\times\left[0,\varrho\right]$
in the sense that for some $\tilde{\varrho}\in\left(0,\varrho\right]$
we have 
\[
u\left(x\right)\,\leq\,u\left(0\right)=0\quad\forall\,x\in\bar{B}_{\tilde{\rho}}^{n-1}\times\left[0,\tilde{\rho}\right].
\]
Namely, the cylindrical point $0$ is a maximum point of $u$ on $\bar{B}_{\tilde{\rho}}^{n-1}\times\left[0,\tilde{\rho}\right]$. 

Since $0$ is not an isolated singular point in $\bar{B}_{\varrho}^{n-1}\times\left[0,\varrho\right]$,
we can find a singular point $\left(y_{*},z_{*}\right)\neq0$ in $\bar{B}_{\tilde{\rho}}^{n-1}\times\left[0,\tilde{\rho}\right]$,
which has to be a cylindrical point by Corollary \ref{isolated round point}.
Recalling that all cylindrical points in $\bar{B}_{\varrho}^{n-1}\times\left[0,\varrho\right]$
must lie on the graph $y=\psi\left(z\right)$, we deduce that $z_{*}>0$
(because $\psi\left(0\right)=0$). Then it follows from Lemma \ref{local max scale}
that the singular set of the flow in $\bar{B}_{\varrho}^{n-1}\times\left[0,z_{*}\right]$
is a $C^{1}$ embedded curve consisting of purely cylindrical points
that is given by 
\[
\left\{ \left(y,z\right)\,:\,y=\psi\left(z\right),\,z\in\left[0,z_{*}\right]\right\} \,\subset\,\mathcal{S}\cap\Sigma_{0}.
\]
\end{proof}
Since the analogous result of Proposition \ref{one-sided structure of singular set}
holds for the flow on the downside $\bar{B}_{\rho}^{n-1}\times\left[-\rho,0\right]$
as well. Combining them together yields the following corollary.
\begin{cor}
\label{local structure of singular set}With the same notations as
in Proposition \ref{one-sided structure of singular set}. If $u\left(0\right)=0$
is the first singular time of the flow $\left\{ \Sigma_{t}\right\} $
and also the unique singular time of the flow in $\bar{B}_{\varrho}^{n-1}\times\left[-\varrho,\varrho\right]$
for some $\varrho>0$. 

Then either the cylindrical point $0$ is an isolated singular point
in $\bar{B}_{\varrho}^{n-1}\times\left[-\varrho,\varrho\right]$,
or there exist $z_{*}^{+}\in\left[0,\varrho\right]$ and $z_{*}^{-}\in\left[-\varrho,0\right]$,
at least one of which is nonzero, so that the singular set of the
flow in $\bar{B}_{\varrho}^{n-1}\times\left[z_{*}^{-},z_{*}^{+}\right]$
is a $C^{1}$ embedded curve consisting of purely cylindrical points
given by 
\[
\left\{ \left(y,z\right)\,:\,y=\psi\left(z\right),\,z\in\left[z_{*}^{-},z_{*}^{+}\right]\right\} ;
\]
moreover, in case that one of $\left\{ z_{*}^{-},z_{*}^{+}\right\} $
is $0,$ say $z_{*}^{-}=0$, then there exists $\varepsilon>0$ so
that the cylindrical point $0$ is the unique singular point in $B_{\varepsilon}\cap\left(\bar{B}_{\varrho}^{n-1}\times\left[-\varrho,0\right]\right)$.
\end{cor}

\subsection{Singular set at the first singular time\label{singular set at the first singular time}}

In this subsection we shall consider the structure of the singular
set at the first singular time $T_{1}$ under the assumption that
$T_{1}$ is an isolated singular time of the flow. The main result
(see the following proposition) is based on Corollary \ref{isolated round point}
(in Section \ref{round point}) and Corollary \ref{local structure of singular set}
(in Section \ref{cylindrical point}). This will be utilized in Section
\ref{singular set at subsequent singular times} to prove Theorem
\ref{structure of singular set}.
\begin{prop}
\label{structure of singular set at first singular time}Suppose either
$T_{1}=T_{ext}$, or that $T_{1}<T_{ext}$ is an isolated singular
time of the flow $\left\{ \Sigma_{t}\right\} $, that is, there exists
a ``second'' singular time $T_{2}\in\left(T_{1},T_{ext}\right]$
so that $\left\{ \Sigma_{t}\right\} $ is regular during $t\in\left(T_{1},T_{2}\right)$.
Then the singular set of $\left\{ \Sigma_{t}\right\} $ at the first
singular time, i.e.,
\[
\mathcal{S}\cap\Sigma_{T_{1}},
\]
has finitely many connected components, each of which is either a
single (round or cylindrical) point, or a compact $C^{1}$ embedded
curve (with or without boundary) consisting of cylindrical points. 
\end{prop}

\begin{proof}
Firstly, recall that $\mathcal{S}$ is a closed subset of $\Omega_{0}$
(see Lemma \ref{closedness of singular set}) and note that $\Sigma_{T_{1}}=\left\{ u=T_{1}\right\} $
is a closed subset of $\Omega_{0}$ owing to the continuity of $u$.
Consequently, $\mathcal{S}\cap\Sigma_{T_{1}}$ is a closed subset
of $\Omega_{0}$ and hence a compact set. 

Secondly, if $\mathcal{S}\cap\Sigma_{T_{1}}$ has a round point, then
by Corollary \ref{isolated round point}, this round point would be
the only element of $\mathcal{S}\cap\Sigma_{T_{1}}$ and hence the
proposition would be proved. So for the rest of the proof let us assume
that $\mathcal{S}\cap\Sigma_{T_{1}}$ comprises purely cylindrical
points.

Fix $p\in\mathcal{S}\cap\Sigma_{T_{1}}$. Note that we can find $r>0$
so that $T_{1}$ is the only singular time of the flow in $B_{r}\left(p\right)$.
This is obviously true when $T_{1}=T_{ext}$; in the case where $T_{1}<T_{ext}$,
since $\left\{ u\geq T_{2}\right\} $ is a compact set in $\Omega_{0}$
and $p\notin\left\{ u\geq T_{2}\right\} $, there exists $r>0$ such
that $B_{r}\left(p\right)\cap\left\{ u\geq T_{2}\right\} =\emptyset,$
which means that 
\begin{equation}
B_{r}\left(p\right)\subset\left\{ u<T_{2}\right\} ,\label{structure of singular set at first singular time: distance to next singualr time slice}
\end{equation}
and so there are no other singular times than $T_{1}$ in $B_{r}\left(p\right)$.
Thus, it follows from Corollary \ref{local structure of singular set}
that there is $\varepsilon\in\left(0,r\right]$ so that one of the
following holds:
\begin{enumerate}
\item $\mathcal{S}\cap B_{\varepsilon}\left(p\right)=\left\{ p\right\} $;
\item $\mathcal{S}\cap B_{\varepsilon}\left(p\right)$ is a $C^{1}$ embedded
curve with $p$ being an interior point;
\item $\mathcal{S}\cap B_{\varepsilon}\left(p\right)$ is a $C^{1}$ embedded
curve with $p$ being a boundary point.
\end{enumerate}
Notice that $\mathcal{S}\cap B_{\varepsilon}\left(p\right)=\mathcal{S}\cap\Sigma_{T_{1}}\cap B_{\varepsilon}\left(p\right)$
by (\ref{structure of singular set at first singular time: distance to next singualr time slice}). 

Therefore, by the compactness of $\mathcal{S}\cap\Sigma_{T_{1}}$,
we infer that $\mathcal{S}\cap\Sigma_{T_{1}}$ is a finite disjoint
union of points and/or compact $C^{1}$ embedded curves (with or without
boundary).
\end{proof}

\subsection{Singular set at subsequent singular times\label{singular set at subsequent singular times}}

We shall prove Theorem \ref{structure of singular set} in the end
of this subsection. The idea of the proof is as follows. By Proposition
\ref{domain after first singular time} (which will be proved in this
subsection), the superlevel set of the arrival time function $u$
at the first singular time $T_{1}$ is a finite disjoint union of
open connected sets, namely, 
\[
\Omega_{T_{1}}=\left\{ u>T_{1}\right\} =\bigsqcup_{i=1}^{m_{1}}\,\Omega_{T_{1}}^{\left(i\right)},
\]
and so 
\[
\Sigma_{t}=\bigsqcup_{i=1}^{m_{1}}\,\left(\Sigma_{t}\cap\Omega_{T_{1}}^{\left(i\right)}\right),\quad t>T_{1};
\]
moreover, each $\Sigma_{t}\cap\Omega_{T_{1}}^{\left(i\right)}$ is
itself a two-convex LSF for $t\in\left(T_{1},\infty\right)$ and a
MCF of closed connected hypersurfaces for $t\in\left(T_{1},T_{2}\right)$,
where $T_{2}$ is the second singular time of $\left\{ \Sigma_{t}\right\} $.
It follows that the singular set of $\left\{ \Sigma_{t}\right\} $
at time $T_{2}$ is indeed the union of the singular sets of $\left\{ \Sigma_{t}\cap\Omega_{T_{1}}^{\left(\hat{i}\right)}\right\} _{t>T_{1}}$'s
at their ``first'' singular time, where $\hat{i}$'s are those indices
such that $\left\{ \Sigma_{t}\cap\Omega_{T_{1}}^{\left(\hat{i}\right)}\right\} _{t>T_{1}}$
does become singular at $t=T_{2}$. In this way, Proposition \ref{structure of singular set at first singular time}
from Section \ref{singular set at the first singular time} can be
applied to each $\left\{ \Sigma_{t}\cap\Omega_{T_{1}}^{\left(\hat{i}\right)}\right\} _{t>T_{1}}$
to study the singular set of $\left\{ \Sigma_{t}\right\} $ at time
$T_{2}$. Such process can be repeated to study the singular set of
$\left\{ \Sigma_{t}\right\} $ at all subsequent singular times $T_{1}<T_{2}<\cdots<T_{ext}$. 

To prove Proposition \ref{domain after first singular time}, we need
the following lemma. 
\begin{lem}
\label{domain at regular time}For each regular time $t\in\left(0,T_{ext}\right)$,
$\Sigma_{t}=\left\{ x\in\Omega_{0}:u\left(x\right)=t\right\} $ is
a finite disjoint union of two-convex\footnote{With respect to the inward unit normal.}
closed connected hypersurfaces 
\[
\Sigma_{t}^{\left(1\right)},\cdots,\Sigma_{t}^{\left(m\right)};
\]
moreover, $\Omega_{t}=\left\{ x\in\Omega_{0}:u\left(x\right)>t\right\} $
is the disjoint union of 
\[
\Omega_{t}^{\left(1\right)},\cdots,\Omega_{t}^{\left(m\right)},
\]
where $\Omega_{t}^{\left(i\right)}$ is the open connected set bounded
by $\Sigma_{t}^{\left(i\right)}$ for $i\in\left\{ 1,\cdots,m\right\} $. 
\end{lem}

\begin{proof}
Let $t\in\left(0,T_{ext}\right)$ be a regular value of $u$. By the
implicit function theorem, $\Sigma_{t}$ is a closed  hypersurface
in $\Omega_{0}$. As $\Sigma_{t}$ is compact, it has finitely many
connected components, say $\Sigma_{t}^{\left(1\right)},\cdots,\Sigma_{t}^{\left(m\right)}$,
each of which is a closed connected hypersurface in $\Omega_{0}$.
By the Jordan-Brouwer separation theorem (see \cite{GP}), each $\Sigma_{t}^{\left(i\right)}$
would bound a region $\Omega_{t}^{\left(i\right)}$ contained\footnote{The mod 2 winding number of $\Sigma_{t}^{\left(i\right)}$ around
every point in $\mathbb{R}^{n}\setminus\bar{\Omega}_{0}$ is 0 (which
can be seen by taking a point far away from $\Sigma_{t}^{\left(i\right)}$
and noting that $\mathbb{R}^{n}\setminus\bar{\Omega}_{0}$ is path-connected
with $\left(\mathbb{R}^{n}\setminus\bar{\Omega}_{0}\right)\cap\Sigma_{t}^{\left(i\right)}=\emptyset$),
so $\Omega_{t}^{\left(i\right)}$ is contained in $\Omega_{0}$.} in $\Omega_{0}$. Because $\Sigma_{t}$ is  two-convex with respect
to $\frac{\nabla u}{\left|\nabla u\right|}$ (cf. \cite{CHN}), so
is $\Sigma_{t}^{\left(i\right)}$. Note that the unit normal $\frac{\nabla u}{\left|\nabla u\right|}$
is either entirely inward or entirely outward for each $\Sigma_{t}^{\left(i\right)}$.
In view of the fact that every closed hypersurface has a point where
it is  convex with respect to the inward unit normal vector, we conclude
that $\frac{\nabla u}{\left|\nabla u\right|}$ is entirely inward
for each $\Sigma_{t}^{\left(i\right)}$.

To see that $\Omega_{t}^{\left(i\right)}\subset\Omega_{t}$, firstly
note that because the vector field $\frac{\nabla u}{\left|\nabla u\right|}$
on $\Sigma_{t}^{\left(i\right)}$ points toward $\Omega_{t}^{\left(i\right)}$,
we are able to find a small tubular open neighborhood $U$ of $\Sigma_{t}^{\left(i\right)}$
so that $\Omega_{t}^{\left(i\right)}\cap U\subset\Omega_{t}$. Were
$\Omega_{t}^{\left(i\right)}\setminus\Omega_{t}\neq\emptyset$, we
would have 
\[
p\,\in\,\Omega_{t}^{\left(i\right)}\setminus\Omega_{t}\,\subset\,\bar{\Omega}_{t}^{\left(i\right)}\setminus U
\]
so that 
\[
u\left(p\right)=\min_{\bar{\Omega}_{t}^{\left(i\right)}\setminus U}\,u\,\leq\,t.
\]
It would follow that $p\in\Omega_{t}^{\left(i\right)}$ with 
\[
u\left(p\right)=\min_{\bar{\Omega}_{t}^{\left(i\right)}}\,u;
\]
particularly, $p$ is a critical point of $u$. However, $u$ has
no local minimum points (see Corollary \ref{isolated round point}
and Corollary \ref{saddle criterion}). Thus we obtain $\Omega_{t}^{\left(i\right)}\subset\Omega_{t}$. 

Recall that both $\Omega_{t}^{\left(i\right)}$ and $\mathbb{R}^{n}\setminus\bar{\Omega}_{t}^{\left(i\right)}$
are open connected (and hence path-connected) by the Jordan-Brouwer
separation theorem. When $m>1$, for every $j\neq i$,
\[
\Sigma_{t}^{\left(i\right)}\cap\Sigma_{t}^{\left(j\right)}=\emptyset,
\]
\[
\Sigma_{t}^{\left(i\right)}\cap\Omega_{t}^{\left(j\right)}\,\subset\,\Sigma_{t}\cap\Omega_{t}=\emptyset,
\]
giving that 
\[
\Sigma_{t}^{\left(i\right)}\,\subset\,\mathbb{R}^{n}\setminus\bar{\Omega}_{t}^{\left(j\right)}.
\]
Notice that $\mathbb{R}^{n}\setminus\bar{\Omega}_{t}^{\left(j\right)}$
is open connected. Then it follows from the path-connectedness of
$\Omega_{t}^{\left(i\right)}$ that 
\[
\bar{\Omega}_{t}^{\left(i\right)}\,\subset\,\mathbb{R}^{n}\setminus\bar{\Omega}_{t}^{\left(j\right)},
\]
that is, $\bar{\Omega}_{t}^{\left(i\right)}\cap\bar{\Omega}_{t}^{\left(j\right)}=\emptyset$.

To finish the proof, we have to show that 
\[
\Omega_{t}\subset\Omega_{t}^{\left(1\right)}\cup\cdots\cup\Omega_{t}^{\left(m\right)}.
\]
Suppose the contrary that there is $p\in\Omega_{t}$ such that $p\notin\Omega_{t}^{\left(i\right)}$
for every $i$. Let 
\[
R=\sup\left\{ r>0\,:\,B_{r}\left(p\right)\subset\Omega_{t}\right\} .
\]
Then $B_{R}\left(p\right)\cap\bar{\Omega}_{t}^{\left(i\right)}=\emptyset$
for every $i$ and $\partial B_{R}\left(p\right)\cap\Sigma_{t}\neq\emptyset$.
Choose $q\in\partial B_{R}\left(p\right)\cap\Sigma_{t}$. Then $q\in\partial B_{R}\left(p\right)\cap\Sigma_{t}^{\left(i_{0}\right)}$
for some $i_{0}$. Since $\nabla u\left(q\right)\neq0$ points toward
$\Omega_{t}^{\left(i_{0}\right)}$, there exists $\delta>0$ so that
\[
B_{\delta}\left(q\right)\setminus\bar{\Omega}_{t}^{\left(i_{0}\right)}\,\subset\,\left\{ u<t\right\} ,
\]
contradicting that $\emptyset\neq B_{R}\left(p\right)\cap B_{\delta}\left(q\right)\subset\Omega_{t}$.
\end{proof}
\begin{prop}
\label{domain after first singular time}Suppose that $T_{1}<T_{ext}$
is an isolated singular time of $\left\{ \Sigma_{t}\right\} $, that
is, there exists a singular time $T_{2}\in\left(T_{1},T_{ext}\right]$
so that the flow is regular during the time period $\left(T_{1},T_{2}\right)$.
Then $\Omega_{T_{1}}$ is a finite disjoint union of open connected
sets
\[
\Omega_{T_{1}}^{\left(1\right)},\cdots,\Omega_{T_{1}}^{\left(m_{1}\right)}.
\]
Moreover, for each $i\in\left\{ 1,\cdots,m_{1}\right\} $, 
\[
\Sigma_{t}^{\left(i\right)}\coloneqq\Sigma_{t}\cap\Omega_{T_{1}}^{\left(i\right)},\quad t>T_{1}
\]
is a  two-convex MCF of closed connected hypersurfaces for $t\in\left(T_{1},T_{2}\right)$
and a  two-convex LSF for $t\in\left(T_{1},\infty\right).$\footnote{In the sense that for every $t_{1}\in\left(T_{1},T_{2}\right)$, $\left\{ \Sigma_{t}^{\left(i\right)}\right\} _{t\geq t_{1}}$is
the LSF starting from $\Sigma_{t_{1}}^{\left(i\right)}$.} Notice that as $\Sigma_{t}\subset\Omega_{T_{1}}$ for $t>T_{1}$,
we have 
\[
\Sigma_{t}=\bigsqcup_{i=1}^{m_{1}}\,\Sigma_{t}^{\left(i\right)}\quad\forall\,t>T_{1}.
\]
\end{prop}

\begin{proof}
Fix $t_{1}\in\left(T_{1},T_{2}\right)$. By Lemma \ref{domain at regular time},
$\Sigma_{t_{1}}$ is a finite disjoint union of closed connected hypersurfaces
$\Sigma_{t_{1}}^{\left(1\right)},\ldots,\Sigma_{t_{1}}^{\left(m_{1}\right)}$
that are  two-convex with respect to the inward normal; moreover,
$\Omega_{t_{1}}$ is the disjoint union of $\Omega_{t_{1}}^{\left(1\right)},\cdots,\Omega_{t_{1}}^{\left(m_{1}\right)}$,
where $\Omega_{t_{1}}^{\left(i\right)}$ is the open connected set
bounded by $\Sigma_{t_{1}}^{\left(i\right)}$. 

Since $u$ satisfies (\ref{level set flow}) in the viscosity sense
on $\Omega_{t_{1}}^{\left(i\right)}$ with the boundary condition
$u=t_{1}$ on $\partial\Omega_{t_{1}}^{\left(i\right)}=\Sigma_{t_{1}}^{\left(i\right)}$,
the existence and uniqueness theorem in \cite{ES} yields that 
\[
\Sigma_{t}^{\left(i\right)}\,\coloneqq\,\Sigma_{t}\cap\Omega_{t_{1}}^{\left(i\right)}\quad t>t_{1}
\]
is indeed the mean-convex LSF starting from $\Sigma_{t_{1}}^{\left(i\right)}$
at time $t_{1}$. In addition, since the flow $\left\{ \Sigma_{t}\right\} $
is a  two-convex MCF for $t\in\left(T_{1},T_{2}\right),$ the flow
$\left\{ \Sigma_{t}^{\left(i\right)}\right\} $ is also a  two-convex
MCF for $t\in\left[t_{1},T_{2}\right)$. 

It follows that for every $t\in\left[t_{1},T_{2}\right)$, each $\Sigma_{t}^{\left(i\right)}$
is a closed connected hypersurface and hence is exactly a connected
component of $\Sigma_{t}$. Let $\Omega_{t}^{\left(i\right)}$ be
the open connected set bounded by $\Sigma_{t}^{\left(i\right)}$ for
$t\in\left(t_{1},T_{2}\right)$. Since $\left\{ \Sigma_{t}^{\left(i\right)}\right\} _{t\in\left[t_{1},T_{2}\right)}$
is a  mean-convex MCF, $\left\{ \Omega_{t}^{\left(i\right)}\right\} _{t\in\left[t_{1},T_{2}\right)}$
is contracting. By Lemma \ref{domain at regular time}, $\Omega_{t}$
is the disjoint union of $\Omega_{t}^{\left(1\right)},\cdots,\Omega_{t}^{\left(m_{1}\right)}$
for $t\in\left[t_{1},T_{2}\right)$.

Replacing $t_{1}\in\left(T_{1},T_{2}\right)$ in the a above by a
decreasing sequence tending to $T_{1}$, we can extend\footnote{On the basis of the uniqueness theorem for MCF and LSF, and also the
semigroup property of LSF operator (cf. \cite{ES}).} the flow $\left\{ \Sigma_{t}^{\left(i\right)}\right\} $ backward
in time so as to obtain that $\left\{ \Sigma_{t}^{\left(i\right)}\right\} _{t\in\left(T_{1},T_{2}\right)}$
is a  two-convex MCF of closed connected hypersurfaces, that $\left\{ \Sigma_{t}^{\left(i\right)}\right\} _{t>T_{1}}$
is a  two-convex LSF, and that $\Sigma_{t}$ is the disjoint union
of $\Sigma_{t}^{\left(1\right)},\ldots,\Sigma_{t}^{\left(m_{1}\right)}$
for every $t\in\left(T_{1},T_{2}\right)$. Also, $\Omega_{t}^{\left(i\right)}$
is defined for every $t\in\left(T_{1},T_{2}\right)$ as the open connected
set bounded by $\Sigma_{t}^{\left(i\right)}$. We then have that $\left\{ \Omega_{t}^{\left(i\right)}\right\} _{t\in\left(T_{1},T_{2}\right)}$
is contracting and that $\Omega_{t}$ is the disjoint union of $\Omega_{t}^{\left(1\right)},\cdots,\Omega_{t}^{\left(m_{1}\right)}$
for $t\in\left(T_{1},T_{2}\right)$. 

For each $i\in\left\{ 1,\cdots,m_{1}\right\} $ let

\[
\Omega_{T_{1}}^{\left(i\right)}\,=\bigcup_{t\in\left(T_{1},T_{2}\right)}\Omega_{t}^{\left(i\right)}.
\]
Since $\Omega_{t}^{\left(i\right)}$ is open connected for every $t\in\left(T_{1},T_{2}\right)$,
so is $\Omega_{T_{1}}^{\left(i\right)}$. When $m_{1}>1$, for every
$j\neq i$ we would have $\Omega_{T_{1}}^{\left(i\right)}\cap\Omega_{T_{1}}^{\left(j\right)}=\emptyset$;
otherwise, choose $p\in\Omega_{T_{1}}^{\left(i\right)}\cap\Omega_{T_{1}}^{\left(j\right)}$,
then by the contracting property of $\Omega_{t}^{\left(i\right)}$
and $\Omega_{t}^{\left(j\right)}$, there would exist $t_{1}\in\left(T_{1},T_{2}\right)$
so that $p\in\Omega_{t_{1}}^{\left(i\right)}\cap\Omega_{t_{1}}^{\left(j\right)}$,
contradicting that $\Omega_{t_{1}}^{\left(i\right)}\cap\Omega_{t_{1}}^{\left(j\right)}=\emptyset$.
In addition, note that
\[
\Omega_{T_{1}}=\left\{ u>T_{1}\right\} \,=\bigcup_{t\in\left(T_{1},T_{2}\right)}\left\{ u>t\right\} \,=\bigcup_{t\in\left(T_{1},T_{2}\right)}\Omega_{t}
\]
\[
=\bigcup_{t\in\left(T_{1},T_{2}\right)}\,\bigcup_{i=1}^{m_{1}}\,\,\Omega_{t}^{\left(i\right)}=\,\bigcup_{i=1}^{m_{1}}\,\bigcup_{t\in\left(T_{1},T_{2}\right)}\Omega_{t}^{\left(i\right)}=\,\bigcup_{i=1}^{m_{1}}\,\,\Omega_{T_{1}}^{\left(i\right)}.
\]
Lastly, given a sequence $\left\{ t_{k}\right\} _{k\in\mathbb{N}}\subset\left(T_{1},T_{2}\right)$
such that $t_{k}\searrow T_{1}$, by construction $\left\{ \Sigma_{t}^{\left(i\right)}\right\} _{t\geq t_{k}}$
is a  mean-convex LSF starting from $\Sigma_{t_{k}}^{\left(i\right)}$
at time $t_{k}$, so we have 
\[
\Sigma_{t}^{\left(i\right)}\,\subset\,\Omega_{t_{k}}^{\left(i\right)}\quad\forall\,t>t_{k}
\]
for each $i\in\left\{ 1,\cdots,m_{1}\right\} $. Since $\Omega_{t_{k}}^{\left(i\right)}\cap\Omega_{t_{k}}^{\left(j\right)}=\emptyset$
whenever $i\neq j$ and that $\Sigma_{t}$ is the disjoint union of
$\Sigma_{t}^{\left(1\right)},\cdots,\Sigma_{t}^{\left(m_{1}\right)}$,
we infer that
\[
\Sigma_{t}^{\left(i\right)}\,=\,\Sigma_{t}\cap\Omega_{t_{k}}^{\left(i\right)}
\]
so long as $t>t_{k}$. Now given $t>T_{1}$, we can choose $k_{0}\in\mathbb{N}$
such that $t>t_{k_{0}}$, then we have 
\[
\Sigma_{t}^{\left(i\right)}\,=\,\bigcup_{k=k_{0}}^{\infty}\Sigma_{t}\cap\Omega_{t_{k}}^{\left(i\right)}\,=\,\Sigma_{t}\cap\Omega_{T_{1}}^{\left(i\right)}
\]
for every $i\in\left\{ 1,\cdots,m_{1}\right\} $.
\end{proof}
With Proposition \ref{structure of singular set at first singular time}
and Proposition \ref{domain after first singular time}, we are able
to prove Theorem \ref{structure of singular set}. 
\begin{proof}
(\textit{of Theorem \ref{structure of singular set}}) If $T_{1}<T_{ext}$,
choose $t_{1}\in\left(T_{1},T_{2}\right)$. Note that 
\[
\Sigma_{t}=\left\{ u=t\right\} \,\subset\,\left\{ u>T_{1}\right\} =\Omega_{T_{1}}\quad\forall\,t\geq t_{1}.
\]
By Proposition \ref{domain after first singular time}, $\Omega_{T_{1}}$
is a finite disjoint union of open connected sets 
\[
\Omega_{T_{1}}^{\left(1\right)},\cdots,\Omega_{T_{1}}^{\left(m_{1}\right)};
\]
moreover, for each $i\in\left\{ 1,\cdots,m_{1}\right\} $, 
\[
\Sigma_{t}\cap\Omega_{T_{1}}^{\left(i\right)},\quad t\geq t_{1}
\]
is the (two-convex) LSF\footnote{\label{restriction of arrival time}The restriction of $u$ on $\bar{\Omega}_{t_{1}}^{\left(i\right)}$
serves as the arrival time function of the two-convex LSF.} starting at time $t_{1}$ from $\Sigma_{t_{1}}\cap\Omega_{T_{1}}^{\left(i\right)}$,
which is a two-convex closed connected hypersurface. Note that for
each $i\in\left\{ 1,\cdots,m_{1}\right\} $, all the singular times
of the flow $\left\{ \Sigma_{t}\cap\Omega_{T_{1}}^{\left(i\right)}\right\} _{t\geq t_{1}}$
must be contained in $\left\{ T_{2}\leq\cdots\leq T_{ext}\right\} $,
and that at time $T_{2}$, there must be some $i\in\left\{ 1,\cdots,m_{1}\right\} $
so that $\Sigma_{T_{2}}\cap\Omega_{T_{1}}^{\left(i\right)}$ has singular
points.\footnote{For such $i$'s, $T_{2}$ is actually the first singular time of the
flow $\left\{ \Sigma_{t}\cap\Omega_{T_{1}}^{\left(i\right)}\right\} _{t\geq t_{1}}$.} Then it follows from Proposition \ref{structure of singular set at first singular time}
that for each $i\in\left\{ 1,\cdots,m_{1}\right\} $, the singular
set of the LSF $\left\{ \Sigma_{t}\cap\Omega_{T_{1}}^{\left(i\right)}\right\} _{t\geq t_{1}}$
at the time $T_{2}$, which is the set of singular points in $\Sigma_{T_{2}}\cap\Omega_{T_{1}}^{\left(i\right)}$,
is either empty or a finite disjoint union of points and/or compact
$C^{1}$ embedded curves (with or without boundary). Because the singular
points of $\Sigma_{T_{2}}$ is indeed the union of singular points
of $\Sigma_{T_{2}}\cap\Omega_{T_{1}}^{\left(i\right)}$ for $i\in\left\{ 1,\cdots,m_{1}\right\} $,
we infer that the singular set of the flow $\left\{ \Sigma_{t}\right\} $
at the second singular time $T_{2}$ is a finite disjoint union of
points and/or compact $C^{1}$ embedded curves.

If $T_{2}<T_{ext}$, choose $t_{2}\in\left(T_{2},T_{3}\right)$. Note
that $\Sigma_{t}\subset\Omega_{T_{2}}\subset\Omega_{T_{1}}$ for $t\geq t_{2}$.
For each $i\in\left\{ 1,\cdots,m_{1}\right\} $, by Proposition \ref{domain after first singular time},
$\Omega_{T_{2}}\cap\Omega_{T_{1}}^{\left(i\right)}$ would be a finite
disjoint union of open connected sets.\footnote{If the LSF $\left\{ \Sigma_{t}\cap\Omega_{T_{1}}^{\left(i\right)}\right\} _{t\geq t_{1}}$
does not become singular at time $T_{2}$, then $\Omega_{T_{2}}\cap\Omega_{T_{1}}^{\left(i\right)}$
is itself an open connected set, which is bounded by the closed connected
hypersurface $\Sigma_{T_{2}}\cap\Omega_{T_{1}}^{\left(i\right)}$.} Since $\Omega_{T_{1}}$ is the disjoint union of $\Omega_{T_{1}}^{\left(1\right)},\cdots,\Omega_{T_{1}}^{\left(m_{1}\right)}$,
the set $\Omega_{T_{2}}$ is a finite disjoint union of open connected
sets
\[
\Omega_{T_{2}}^{\left(1\right)},\cdots,\Omega_{T_{2}}^{\left(m_{2}\right)}.
\]
Note that for each $j\in\left\{ 1,\cdots,m_{2}\right\} $, there exists
$i\in\left\{ 1,\cdots,m_{1}\right\} $ such that $\Omega_{T_{2}}^{\left(j\right)}$
is a connected component of $\Omega_{T_{2}}\cap\Omega_{T_{1}}^{\left(i\right)}$;
it follows from Proposition \ref{domain after first singular time}
that 
\[
\left(\Sigma_{t}\cap\Omega_{T_{1}}^{\left(i\right)}\right)\cap\Omega_{T_{2}}^{\left(j\right)}\,=\,\Sigma_{t}\cap\Omega_{T_{2}}^{\left(j\right)},\quad t\geq t_{2}
\]
is the (two-convex) LSF starting at time $t_{2}$ from $\Sigma_{t_{2}}\cap\Omega_{T_{2}}^{\left(j\right)}$,
which is a two-convex closed connected hypersurface. Since 
\[
\Sigma_{t}=\bigsqcup_{j=1}^{m_{2}}\,\Sigma_{t}\cap\Omega_{T_{2}}^{\left(j\right)}\quad\forall\,t\geq t_{2},
\]
the singular times of the flow $\left\{ \Sigma_{t}\cap\Omega_{T_{2}}^{\left(j\right)}\right\} _{t\geq t_{2}}$
must be contained in $\left\{ T_{3}\leq\cdots\leq T_{ext}\right\} $
for every $j$; at time $T_{3}$, there must be some $j\in\left\{ 1,\cdots,m_{2}\right\} $
so that $\Sigma_{T_{3}}\cap\Omega_{T_{2}}^{\left(j\right)}$ has singular
points. So it follows from Proposition \ref{structure of singular set at first singular time}
that for each $j\in\left\{ 1,\cdots,m_{2}\right\} $, the singular
set of the LSF $\left\{ \Sigma_{t}\cap\Omega_{T_{2}}^{\left(j\right)}\right\} _{t\geq t_{2}}$
at the time $T_{3}$, which is the set of singular points in $\Sigma_{T_{3}}\cap\Omega_{T_{2}}^{\left(j\right)}$,
is either empty or a finite disjoint union of points and/or compact
$C^{1}$ embedded curves. Thus, the singular set of the flow $\left\{ \Sigma_{t}\right\} $
at the third singular time $T_{3}$ is a finite disjoint union of
points and/or compact $C^{1}$ embedded curves.

The conclusion would follow after repeating this process for finitely
many times.
\end{proof}

\section{\uuline{Types of singular components}\label{types of singular components}}

\uline{In this section the two-convex LSF \mbox{$\left\{ \Sigma_{t}\right\} $}
is assumed to have finitely many singular times in order that Theorem
\mbox{\ref{structure of singular set}} holds.} Let $\left\{ \Sigma_{t}^{k}\right\} $,
$k\in\mathbb{N}$, be the LSFs in Theorem \ref{stability theorem}.
In view of Proposition \ref{two-convexity of approximations} (in
Section \ref{appendix:two-convexity}), we may assume for simplicity
that $\left\{ \Sigma_{t}^{k}\right\} $ is a two-convex LSF for every
$k\in\mathbb{N}$. 

Recall that when $k$ is large, by Theorem \ref{stability theorem}
(see also Corollary \ref{regular space}) $\left\{ \Sigma_{t}^{k}\right\} $
is regular and close in the smooth topology to $\left\{ \Sigma_{t}\right\} $
away from $\mathcal{S}$. The target of this section is to analyze
the singular set of $\left\{ \Sigma_{t}^{k}\right\} $ near $\mathcal{S}$
with results in Theorem \ref{stability of singular types}. To achieve
that, we will firstly classify the singular components (i.e., connected
components of the singular set) in Definition \ref{singularity types}.
Then we will make adequate assumptions, including Assumption \ref{local singular times hypothesis},
Assumption \ref{not getting into} (in Section \ref{splitting type}),
and Assumption \ref{existence of singularities near bumpy} (in Section
\ref{bumpy type}), so as to ensure that near each singular component
of $\left\{ \Sigma_{t}\right\} $, the flow $\left\{ \Sigma_{t}^{k}\right\} $
would have exactly the same type of singular set as that singular
component. The proof of Theorem \ref{stability of singular types}
is composed of three parts: Proposition \ref{stability of vanishing type}
in Section \ref{vanishing type}, Proposition \ref{rule out bumpy type}
in Section \ref{splitting type}, and Proposition \ref{uniqueness of bumpy}
in Section \ref{bumpy type}.

To start with, let us write the singular set $\mathcal{S}$ of $\left\{ \Sigma_{t}\right\} $
as a finite disjoint union of singular components, namely,
\[
\mathcal{S}\,=\,\bigsqcup_{j}\,\mathcal{S}_{j},
\]
where each singular component $\mathcal{S}_{j}$ is either a point
or a compact $C^{1}$ embedded curve (with or without boundary). Note
that $u$ is constant on each $\mathcal{S}_{j}$,\footnote{When $\mathcal{S}_{j}=\left\{ x\left(t\right):t\in\left[a,b\right]\right\} $
is a curve, we have $\frac{d}{dt}\,u\left(x\left(t\right)\right)=\nabla u\left(x\left(t\right)\right)\cdot x'\left(t\right)=0$.} that is to say, every singularity of $\mathcal{S}_{j}$ occurs at
the same time. In addition, when $\mathcal{S}_{j}$ is a curve, by
Remark \ref{one-sided local max} every interior point of $\mathcal{S}_{j}$
is a local maximum point of $u$, and every boundary point (i.e.,
endpoint) of $\mathcal{S}_{j}$ is either a local maximum point or
a one-sided saddle point of $u$. We then classify singular components,
according to their endpoints, into the following three types:
\begin{defn}
\label{singularity types}A singular component of $\left\{ \Sigma_{t}\right\} $
belongs to 
\begin{enumerate}
\item the \textbf{vanishing type} - if it is a single local maximum point
of $u$, or a compact $C^{1}$ embedded curve whose endpoints are
both local maximum points of $u$, or a closed $C^{1}$ embedded curve
(i.e., no endpoints);
\item the \textbf{splitting type} - if it is either a single two-sided saddle
point of $u$, or a compact $C^{1}$ embedded curve whose endpoints
are both one-sided saddle points of $u$;
\item the \textbf{bumpy type} - if it is either a single one-sided saddle
point of $u$, or a compact $C^{1}$ embedded curve with one endpoint
being a one-sided saddle point while the other being a local maximum
point of $u$.
\end{enumerate}
\end{defn}

The reason why the two-sided saddle points and the one-sided saddle
points are classified as the splitting type and the bumpy type, respectively,
is as follows. Recall that by Definition \ref{one-sided saddle},
a saddle point $p$ is two-sided or one-sided depends on how many
``sides'' the point can be approached by the superlevel set $\Omega_{u\left(p\right)}=\left\{ u>u\left(p\right)\right\} $.
From this perspective, and considering the fact that a splitting curve
can be (and can only be) approached by the superlevel set near the
two ``ends,'' a two-sided saddle point can be regarded as a ``degenerate''
curve of the splitting type. Likewise, a one-sided saddle point can
be viewed as a degenerate curve of the bumpy type because the superlevel
set can be approached from only one side/end.

Now let $\hat{\delta}>0$ be a sufficiently small constant such that
the following hold:
\begin{itemize}
\item For every $j\neq j'$,
\begin{equation}
\mathcal{S}_{j}^{\hat{\delta}}\cap\mathcal{S}_{j'}^{\hat{\delta}}=\emptyset,\label{delta 1}
\end{equation}
where 
\[
\mathcal{S}_{j}^{\hat{\delta}}=\left\{ x\in\mathbb{R}^{n}\,:\,\textrm{dist}\left(x,\mathcal{S}_{j}\right)<\hat{\delta}\right\} 
\]
is the $\hat{\delta}$-neighborhood of $\mathcal{S}_{j}$.
\item For every $l\geq1$,
\begin{equation}
T_{l-1}+\hat{\delta}\,<\,T_{l}-\hat{\delta},\label{delta 2}
\end{equation}
where $0=T_{0}<T_{1}\leq\cdots\leq T_{m}=T_{ext}$ are the set of
singular times of $\left\{ \Sigma_{t}\right\} $.
\item If $u\left(\mathcal{S}_{j}\right)=T_{l}$ (i.e., $\mathcal{S}_{j}$
occurs at time $T_{l}$), then 
\begin{equation}
\overline{\mathcal{S}_{j}^{\hat{\delta}}}\,\subset\,\Omega_{T_{l-1}}^{\left(i\right)},\label{delta 3}
\end{equation}
where $\Omega_{T_{l-1}}^{\left(i\right)}$ is some connected component
of $\Omega_{T_{l-1}}$ (see Proposition \ref{domain after first singular time}). 
\end{itemize}
The following proposition is a collection of Corollary \ref{extinction time},
Corollary \ref{regular space} (see also (\ref{delta 1})), and Corollary
\ref{closedness during regular times} (see also Proposition \ref{domain after first singular time}
and (\ref{delta 2})).
\begin{prop}
\label{pre-stability}Given $\epsilon>0$ and $0<\delta<\hat{\delta}$,
there exists $k_{\epsilon,\delta}\in\mathbb{N}$ so that for every
$k\geq k_{\epsilon,\delta}$ the following hold: 
\begin{enumerate}
\item The extinction time $T_{ext}^{k}$ of $\left\{ \Sigma_{t}^{k}\right\} $
satisfies 
\[
T_{ext}^{k}\,\in\,\left(T_{ext}-\delta,\,T_{ext}+\delta\right).
\]
\item The singular set $\mathcal{S}^{k}$ of $\left\{ \Sigma_{t}^{k}\right\} $
satisfies 
\[
\mathcal{S}^{k}\,\subset\,\bigsqcup_{j}\,\mathcal{S}_{j}^{\delta}.
\]
\item For every $l\geq1$, $\left\{ \Sigma_{t}^{k}\right\} _{T_{l-1}+\delta\,\leq\,t\,\leq\,T_{l}-\delta}$
is a two-convex MCF that is $\epsilon$-close in the $C^{\dot{m}}$
topology to $\left\{ \Sigma_{t}\right\} _{T_{l-1}+\delta\,\leq\,t\,\leq\,T_{l}-\delta}$,
where $\dot{m}=\dot{m}\left(n,\lambda\right)$ is the constant in
Section \ref{cylindrical point}. 
\end{enumerate}
\end{prop}

We will make good use of the above proposition in Section \ref{vanishing type},
Section \ref{splitting type}, and Section \ref{bumpy type} with
$\epsilon$ and $\delta$ chosen sufficiently small subject to the
asymptotic behavior of the flow $\left\{ \Sigma_{t}\right\} $ near
each of its singular component. Throughout this section we assume
that 
\begin{assumption}
\label{local singular times hypothesis}When $k$ is sufficiently
large, the flow $\left\{ \Sigma_{t}^{k}\right\} $ has at most one
singular time in $\mathcal{S}_{j}^{\hat{\delta}}$ for every $j$;
in particular, Theorem \ref{structure of singular set} and Definition
\ref{singularity types} are applicable to $\left\{ \Sigma_{t}^{k}\right\} $
as well.
\end{assumption}

Two supplementary assumptions will be made for the splitting case
(see Assumption \ref{not getting into}) in Section \ref{splitting type}
and the bumpy case (see Assumption \ref{existence of singularities near bumpy})
in Section \ref{bumpy type}.
\begin{rem}
\label{assuming first singular time}On account of (\ref{delta 3}),
each singular component $\mathcal{S}_{j}$ of $\left\{ \Sigma_{t}\right\} $
is indeed the singular set of the flow
\begin{equation}
\left\{ \Sigma_{t}\cap\Omega_{T_{l-1}}^{\left(i\right)}\right\} _{t>T_{l-1}}\label{component of flow}
\end{equation}
at its ``first'' singular time (see also the exposition at the beginning
of Section \ref{singular set at subsequent singular times}). \uline{Thus,
throughout Section \mbox{\ref{types of singular components}}, upon
replacing the flow by some of its connected component (e.g., (\mbox{\ref{component of flow}}))
if necessary, we may assume without loss of generality that every
singular component under discussion occurs at the first singular time.}
\end{rem}

\subsection{Vanishing type\label{vanishing type}}

In this subsection we shall give a criteria (see Proposition \ref{global maximum points })
to distinguish between the vanishing type and the splitting/bumpy
type. Then we will prove the ``stability'' of singular components
of the vanishing type in Proposition \ref{stability of vanishing type}.

Let us begin by recalling that a singular component $\mathcal{S}_{j}$
of $\left\{ \Sigma_{t}\right\} $ is of the vanishing type if and
only if it comprises local maximum points of $u$ (see Definition
\ref{singularity types}). In the following proposition, we show that
the flow $\left\{ \Sigma_{t}\right\} $ would shrink to $\mathcal{S}_{j}$
at time $T_{1}=u\left(\mathcal{S}_{j}\right)$ (see Remark \ref{assuming first singular time})
and then vanish completely. A typical example is when $\mathcal{S}_{j}$
is a single round point (see Corollary \ref{isolated round point}). 
\begin{prop}
\label{global maximum points }A singular component $\mathcal{S}_{j}$
of $\left\{ \Sigma_{t}\right\} $ is of the vanishing type if and
only if $T_{1}=T_{ext}$ (in the setting of Remark \ref{assuming first singular time});
in that case, $\mathcal{S}_{j}=\mathcal{S}$ (that is, $\mathcal{S}$
has only one component and hence is connected). 

In other words, a singular component $\mathcal{S}_{j}$ of $\left\{ \Sigma_{t}\right\} $
is of the splitting/bumpy type if and only if $T_{1}<T_{ext}$.
\end{prop}

\begin{proof}
Let $\mathcal{S}_{j}$ be a singular component of the vanishing type.
By Definition \ref{singularity types}, $\mathcal{S}_{j}$ consists
of local maximum points of $u$ and is compact, so there exists $\delta>0$
sufficiently small such that 
\begin{enumerate}
\item the $\delta$-neighborhood $\mathcal{S}_{j}^{\delta}$ of $\mathcal{S}_{j}$
is strictly contained in $\Omega_{0}$; 
\item in $\mathcal{S}_{j}^{\delta}$ there are no other singular points
of the flow $\left\{ \Sigma_{t}\right\} $ than $\mathcal{S}_{j}$;
\item $u\leq T_{1}$ in $\mathcal{S}_{j}^{\delta}$.
\end{enumerate}
Then the value of $u$ at every point in $\mathcal{\bar{S}}_{j}^{\delta/2}\setminus\mathcal{S}_{j}$
is strictly less than $T_{1}$; otherwise, that point would be a local
maximum point of $u$ and hence a singular point of $\left\{ \Sigma_{t}\right\} $,
contradicting the second condition in the above. In particular, we
have 
\[
\tau\coloneqq\max_{\partial\mathcal{S}_{j}^{\delta/2}}\,u\,<\,T_{1}.
\]
Note that for every $t\in\left(\tau,T_{1}\right)$, $\Sigma_{t}\cap\mathcal{S}_{j}^{\delta/2}$
is a nonempty (by the intermediate value theorem) closed hypersurface
since $t\in\left(0,T_{1}\right)$ is a regular value of $u$ and 
\[
\Sigma_{t}\,\cap\,\partial\mathcal{S}_{j}^{\delta/2}=\emptyset.
\]
As $\Sigma_{t}$ is connected, we infer that $\Sigma_{t}\subset\mathcal{S}_{j}^{\delta/2}$.
It follows that $\Omega_{t}\subset\mathcal{S}_{j}^{\delta/2}$ and
hence 
\[
T_{ext}=\max_{\Omega_{t}}\,u\,\leq\,\max_{\mathcal{S}_{j}^{\delta/2}}\,u=T_{1}.
\]
Thus, $T_{1}=T_{ext}$. Additionally, since 
\[
\mathcal{S}=\mathcal{S}\cap\Omega_{t}\,\subset\,\mathcal{S}\cap\mathcal{S}_{j}^{\delta}=\mathcal{S}_{j},
\]
we conclude that $\mathcal{S}_{j}=\mathcal{S}$.

Conversely, if $T_{1}=T_{ext}$, then every component of $\mathcal{S}$
is obviously of the vanishing type since $\mathcal{S}$ consists of
global maximum points of $u$. Let $\mathcal{S}_{j}$ be one of the
components of $\mathcal{S}$. Then by the above argument we obtain
$\mathcal{S}_{j}=\mathcal{S}$.
\end{proof}
The following result is based on Proposition \ref{pre-stability}
and Proposition \ref{global maximum points }.
\begin{prop}
\label{stability of vanishing type}If $\mathcal{S}_{j}$ is a singular
component of $\left\{ \Sigma_{t}\right\} $ that belongs to the vanishing
type, then for every sufficiently large $k$, in $\mathcal{S}_{j}^{\hat{\delta}}$
there is precisely one singular component of $\left\{ \Sigma_{t}^{k}\right\} $,
which is of the vanishing type.
\end{prop}

\begin{proof}
By the argument in the proof of Proposition \ref{global maximum points },
there is $\delta>0$ sufficiently small that 
\[
\bar{\Omega}_{T_{1}-\delta}\,\subset\,\mathcal{S}_{j}^{\hat{\delta}}\subset\,\overline{\mathcal{S}_{j}^{\hat{\delta}}}\,\subset\,\Omega_{\delta}.
\]
By Proposition \ref{pre-stability}, when $k$ is large, $\left\{ \Sigma_{t}^{k}\right\} _{\delta\,\leq\,t\,\leq\,T_{1}-\delta}$
is a two-convex MCF of closed connected hypersurfaces so that 
\[
\bar{\Omega}_{T_{1}-\delta}^{k}\,\subset\,\mathcal{S}_{j}^{\hat{\delta}}\,\subset\,\overline{\mathcal{S}_{j}^{\hat{\delta}}}\,\subset\,\Omega_{\delta}^{k}.
\]
In view of Assumption \ref{local singular times hypothesis}, 
\[
T_{ext}^{k}=\,\max_{\bar{\Omega}_{T_{1}-\delta}}\,u^{k}
\]
is the unique singular time of the flow $\left\{ \Sigma_{t}^{k}\right\} $
in $\mathcal{S}_{j}^{\hat{\delta}}$. Then applying Proposition \ref{global maximum points }
to the LSF $\left\{ \Sigma_{t}^{k}\right\} _{t\geq\delta}$ gives
that the singular set of $\left\{ \Sigma_{t}^{k}\right\} $ in $\mathcal{S}_{j}^{\hat{\delta}}$
is connected and of the vanishing type.
\end{proof}

\subsection{Splitting type\label{splitting type}}

In this subsection we will be devoted to show the ``stability''
of singular components of the splitting type (see Proposition \ref{classification of singularities near splitting type }
and Proposition \ref{rule out bumpy type}) under Assumption \ref{not getting into}.
In view of Proposition \ref{global maximum points }, it can be assumed
throughout this (and also the next) subsection that $T_{1}<T_{ext}$.
Additionally, note that a singular component of the splitting/bumpy
type consists of cylindrical points. 

Prior to beginning the subsection, let us give a brief overview as
follows. Let $\mathcal{S}_{j}$ be a singular component of $\left\{ \Sigma_{t}\right\} $
belonging to the splitting type. When $k$ is large, we firstly prove
in Proposition \ref{uniform cylindrical scale} that every cylindrical
point of $\left\{ \Sigma_{t}^{k}\right\} $ near $\mathcal{S}_{j}$
has a uniform (i.e., independent of position and $k$) cylindrical
scale and then in Remark \ref{tube-like hypersurface} and Remark
\ref{tubular region for splitting} show that $\Sigma_{\tau_{j}}^{k}$
locally looks like a ``tube'' around $\mathcal{S}_{j}$ for some
$\tau_{j}<T_{1}=u\left(\mathcal{S}_{j}\right)$. In Proposition \ref{characterization of splitting }
we demonstrate that $\Omega_{t}$ ``splits'' near $\mathcal{S}_{j}$
for $t>T_{1}$, which, under Assumption \ref{not getting into}, yields
that $\left\{ \Sigma_{t}^{k}\right\} $ must have singularities near
$\mathcal{S}_{j}$ (see Lemma \ref{existence of singular time}).
Lastly, in Proposition \ref{classification of singularities near splitting type }
and Proposition \ref{rule out bumpy type}, we conclude that $\left\{ \Sigma_{t}^{k}\right\} $
would have precisely one singular component near $\mathcal{S}_{j}$,
which is of the splitting type. 

Now let us start with the following proposition. Note that the ``cylindrical
scale'' therein refers to the $\left(\phi\left(n\right),\epsilon\left(n\right)\right)$-cylindrical
scale in Section \ref{cylindrical point} (see the exposition between
Remark \ref{one-sided isolated local max} and Remark \ref{one-sided local max}).
\begin{prop}
\label{uniform cylindrical scale}Let $\mathcal{S}_{j}$ be a singular
component of $\left\{ \Sigma_{t}\right\} $ consisting of cylindrical
points. Then there exists $0<r_{j}<\hat{\delta}$ with the following
properties:
\begin{itemize}
\item Every cylindrical point of $\left\{ \Sigma_{t}\right\} $ on $\mathcal{S}_{j}$
has a uniform cylindrical scale $r_{j}$.
\item When $k$ is large, every cylindrical point $q$ of $\left\{ \Sigma_{t}^{k}\right\} $
in $\mathcal{S}_{j}^{\hat{\delta}}$, if any, has a cylindrical scale
$r_{j}$; moreover, every other cylindrical points of $\left\{ \Sigma_{t}^{k}\right\} $
in $B_{r_{j}}\left(q\right)$ would be located in a small Lipschitz
graph over the axis of the tangent cylinder of $\left\{ \Sigma_{t}^{k}\right\} $
at $q$. 
\item In the case where $\mathcal{S}_{j}$ is a curve with endpoints, we
have
\begin{equation}
\bar{B}_{3r_{j}}^{+}\left(p\right)\cap\mathcal{S}_{j}=\left\{ p\right\} \label{uniform cylindrical scale: half ball}
\end{equation}
for each endpoint $p$ of $\mathcal{S}_{j}$, where $B_{3r_{j}}^{+}\left(p\right)$
is one of the half balls cut from $B_{3r_{j}}\left(p\right)$ by the
hyperplane $\mathcal{P}$ passing through $p$ and orthogonal to the
tangent cylinder of $\left\{ \Sigma_{t}\right\} $ at $p$, and $\bar{B}_{3r_{j}}^{+}\left(p\right)$
is the closure of $B_{3r_{j}}^{+}\left(p\right)$.
\end{itemize}
\end{prop}

\begin{proof}
Firstly, let $\psi,\delta>0$ be small constants (to be determined),
we claim that there is $0<r_{j}<\hat{\delta}$ so that every cylindrical
point of $\left\{ \Sigma_{t}\right\} $ on $\mathcal{S}_{j}$ has
a uniform $\left(\psi,\delta\right)$-cylindrical scale $3r_{j}$.
This is obviously true if $\mathcal{S}_{j}$ is just a single point.
When $\mathcal{S}_{j}$ is a curve, say 
\[
\mathcal{S}_{j}=\left\{ \Gamma\left(s\right)\,:\,s\in\left[0,1\right]\right\} ,
\]
then by Section \ref{cylindrical point}, for every $s\in\left[0,1\right]$
there is an open neighborhood $\mathcal{O}_{s}$ of $\Gamma\left(s\right)$
so that every cylindrical point of $\left\{ \Sigma_{t}\right\} $
in $\mathcal{O}_{s}$ has a uniform $\left(\psi,\delta\right)$-cylindrical
scale $\mathcal{R}_{s}$. By the compactness of $\mathcal{S}_{j}$,
there exist 
\[
0\leq s_{1}<\cdots<s_{m}\leq1
\]
so that 
\[
\mathcal{S}_{j}\,\subset\,\mathcal{O}_{s_{1}}\cup\cdots\cup\mathcal{O}_{s_{m}}.
\]
Then every cylindrical point of $\left\{ \Sigma_{t}\right\} $ on
$\mathcal{S}_{j}$ has a uniform $\left(\psi,\delta\right)$-cylindrical
scale $3r_{j}$, where 
\[
r_{j}\,\coloneqq\,\frac{1}{3}\min\left\{ \mathcal{R}_{s_{1}},\cdots,\mathcal{R}_{s_{m}},\hat{\delta}\right\} ,
\]
proving the claim. Note also that in the case where $\mathcal{S}_{j}$
is a curve with endpoints, (\ref{uniform cylindrical scale: half ball})
holds by choosing $r_{j}$ even smaller.

By Proposition \ref{C0 compactness} and Proposition \ref{pre-stability},
given a small constant $\sigma\in\left(0,1\right)$ (to be determined),
there exists $k_{0}\in\mathbb{N}$ so that for every $k\geq k_{0}$
the following hold:
\begin{itemize}
\item the singular set of $\left\{ \Sigma_{t}^{k}\right\} $ in $\mathcal{S}_{j}^{\hat{\delta}}$
is contained in $\mathcal{S}_{j}^{\sigma r_{j}}$;
\item $\left|T_{1,j}^{k}-T_{1}\right|^{\frac{1}{2}}\leq\sigma r_{j}$, where
$T_{1,j}^{k}$ is the unique singular time of $\left\{ \Sigma_{t}^{k}\right\} $
in $\mathcal{S}_{j}^{\hat{\delta}}$;
\item $\frac{1}{\sqrt{T_{1}-t}}\left(\Sigma_{t}^{k}-p\right)$ is $\sigma$-close
in $C^{\dot{m}}$-topology to $\frac{1}{\sqrt{T_{1}-t}}\left(\Sigma_{t}-p\right)$
in $B_{1/\sigma}$ for every $p\in\mathcal{S}_{j}$ and $t\in\left[T_{1}-\frac{r_{j}^{2}}{n-2},\,T_{1}-\frac{r_{j}^{2}}{8\left(n-2\right)}\right]$.
\end{itemize}
Then by Section \ref{cylindrical point} there are small positive
constants $\psi,\delta,\sigma$ (depending on $n$ and $\lambda$)\footnote{See Proposition \ref{entropy} for the uniform estimate of the entropy
of $\left\{ \Sigma_{t}^{k}\right\} $ when $k$ is large.} so that 
\begin{itemize}
\item every cylindrical point of $\left\{ \Sigma_{t}\right\} $ on $\mathcal{S}_{j}$
has a uniform $\left(\phi\left(n\right),\epsilon\left(n\right)\right)$-cylindrical
scale $r_{j}$;
\item when $k\geq k_{0}$, every cylindrical point of $\left\{ \Sigma_{t}^{k}\right\} $
in $\mathcal{S}_{j}^{\hat{\delta}}$, if any, has a uniform $\left(\phi\left(n\right),\epsilon\left(n\right)\right)$-cylindrical
scale $r_{j}$; moreover, for each cylindrical point $q$ of $\left\{ \Sigma_{t}^{k}\right\} $
in $\mathcal{S}_{j}^{\hat{\delta}}$, every other cylindrical point
of $\left\{ \Sigma_{t}^{k}\right\} $ in $B_{r_{j}}\left(q\right)$
would be located in a small Lipschitz graph over the axis of the tangent
cylinder of $\left\{ \Sigma_{t}^{k}\right\} $ at $q$.
\end{itemize}
\end{proof}
Based on the asymptotically cylindrical behavior of the flow within
the cylindrical scale (see Section \ref{cylindrical point}), we then
have the following remarks. 
\begin{rem}
\label{tube-like hypersurface}Let $\mathcal{S}_{j}$ be a singular
component of $\left\{ \Sigma_{t}\right\} $ consisting of cylindrical
points. By Proposition \ref{uniform cylindrical scale}, every cylindrical
point of $\left\{ \Sigma_{t}\right\} $ on $\mathcal{S}_{j}$ has
a cylindrical $r_{j}$. It follows that for every $p\in\mathcal{S}_{j}$,
$\Sigma_{\tau_{j}}$ is asymptotically cylindrical in $B_{r_{j}}\left(p\right)$,
where 
\[
\tau_{j}\,\coloneqq\,T_{1}-\frac{r_{j}^{2}}{4\left(n-2\right)};
\]
consequently, $\Sigma_{\tau_{j}}\cap\mathcal{S}_{j}^{r_{j}}$ is a
\textbf{tube-like hypersurface} around $\mathcal{S}_{j}$. By Proposition
\ref{pre-stability}, when $k$ is large, $\Sigma_{\tau_{j}}^{k}$
would be close to $\Sigma_{\tau_{j}}$ so that in $\mathcal{S}_{j}^{r_{j}}$
it is also a tube-like hypersurface around $\mathcal{S}_{j}$ (and
also every singular component of $\left\{ \Sigma_{t}^{k}\right\} $
in $\mathcal{S}_{j}^{r_{j}}$, if any). 
\end{rem}

\begin{rem}
\label{tubular region for splitting}Suppose in Remark \ref{tube-like hypersurface}
we assume that 
\[
\mathcal{S}_{j}=\left\{ \Gamma\left(s\right)\,:\,s\in\left[-1,1\right]\right\} 
\]
is a curve of the splitting type, namely, its endpoints are one-sided
saddle points. Since $\mathcal{S}_{j'}^{\hat{\delta}}\cap\mathcal{S}_{j}^{\hat{\delta}}=\emptyset$
for $j'\neq j$ and that $r_{j}<\hat{\delta}$ (see Proposition \ref{uniform cylindrical scale}),
the first singular time $T_{1}$ is indeed the unique singular time
of $\left\{ \Sigma_{t}\right\} $ in $\mathcal{S}_{j}^{r_{j}}$. Following
Remark \ref{flow near saddle}, near the endpoints\footnote{Here $\Gamma\left(1\right)$ and $\Gamma\left(-1\right)$ correspond
to the origin in Remark \ref{flow near saddle}. Note that the orientation
is chosen in such a way that the curve is on the downside $z\leq0$.} we can choose two hyperplanes $\mathcal{P}_{+}$ and $\mathcal{P}_{-}$\footnote{Here $\mathcal{P}_{+}$ and $\mathcal{P}_{-}$ correspond to $z=\mathring{\varepsilon}$
in (\ref{flow near saddle: tubular region}).} that are orthogonal to the tangent cylinders of $\left\{ \Sigma_{t}\right\} $
at $\Gamma\left(1\right)$ and $\Gamma\left(-1\right)$, respectively,
such that
\[
\mathcal{P}_{+}\cap\mathcal{S}_{j}^{r_{j}}\cap\mathcal{S}_{j}=\emptyset,\,\,\,\mathcal{P}_{-}\cap\mathcal{S}_{j}^{r_{j}}\cap\mathcal{S}_{j}=\emptyset
\]
and that the tube-like connected hypersurface $\Sigma_{\tau_{j}}$
together with $\mathcal{P}_{+}$ and $\mathcal{P}_{-}$ bound a closed
\textbf{tubular region} $\mathcal{T}_{j}$ in $\mathcal{S}_{j}^{r_{j}}$
with the following properties:
\begin{enumerate}
\item $\mathcal{S}_{j}\subset\textrm{int}\,\mathcal{T}_{j}$. 
\item Let $\mathcal{\tilde{P}}_{+}$ and $\mathcal{\tilde{P}}_{-}$\footnote{Here $\mathcal{\tilde{P}}_{+}$ and $\mathcal{\tilde{P}}_{-}$ correspond
to $z=0$ in (\ref{flow near saddle: tubular region}).} parallel to $\mathcal{P}_{+}$ and $\mathcal{P}_{-}$, respectively,
so that 
\[
\mathcal{\tilde{P}}_{+}\cap\mathcal{T}_{j}\cap\mathcal{S}_{j}=\left\{ \Gamma\left(1\right)\right\} ,\quad\mathcal{\tilde{P}}_{-}\cap\mathcal{T}_{j}\cap\mathcal{S}_{j}=\left\{ \Gamma\left(-1\right)\right\} .
\]
Note that the tube $\mathcal{T}_{j}$ is cut by $\mathcal{\tilde{P}}_{+}$
and $\mathcal{\tilde{P}}_{-}$ into three closed pieces,\footnote{Here $\mathcal{\hat{T}}_{j}^{+}$ and $\mathcal{\hat{T}}_{j}^{-}$
correspond to $\mathcal{\hat{T}}$ in (\ref{flow near saddle: tubular region}).} 
\[
\mathcal{T}_{j}=\mathcal{\hat{T}}_{j}^{+}\cup\mathcal{T}_{j}^{*}\cup\mathcal{\hat{T}}_{j}^{-},
\]
where $\mathcal{T}_{j}^{*}$ is the truncated tubular closed region
bounded by $\Sigma_{\tau_{j}}$ (from the lateral) and the two caps
$\mathcal{\tilde{P}}_{+}\cap\mathcal{T}_{j}$ and $\mathcal{\tilde{P}}_{-}\cap\mathcal{T}_{j}$
(from the two ends); $\mathcal{\hat{T}}_{j}^{+}$ is the tubular region
bounded by $\Sigma_{\tau_{j}}$ and the caps $\mathcal{P}_{+}\cap\mathcal{T}_{j}$
and $\mathcal{\tilde{P}}_{+}\cap\mathcal{T}_{j}$; and $\mathcal{\hat{T}}_{j}^{-}$
is the tubular region bounded by $\Sigma_{\tau_{j}}$ and the caps
$\mathcal{P}_{-}\cap\mathcal{T}_{j}$ and $\mathcal{\tilde{P}}_{-}\cap\mathcal{T}_{j}$.
Note also that $\mathcal{S}_{j}\subset\mathcal{T}_{j}^{*}$ and that
\[
\Gamma\left(1\right)\,\in\,\mathcal{\hat{T}}_{j}^{+}\,\subset\,\bar{B}_{r_{j}}^{+}\left(\Gamma\left(1\right)\right),
\]
\[
\Gamma\left(-1\right)\,\in\,\mathcal{\hat{T}}_{j}^{-}\,\subset\,\bar{B}_{r_{j}}^{-}\left(\Gamma\left(-1\right)\right),
\]
where $\bar{B}_{r_{j}}^{+}\left(\Gamma\left(1\right)\right)$ and
$\bar{B}_{r_{j}}^{-}\left(\Gamma\left(-1\right)\right)$ are the closures
of the open half-balls\footnote{Here $B_{r_{j}}^{+}\left(\Gamma\left(1\right)\right)$ and $B_{r_{j}}^{-}\left(\Gamma\left(-1\right)\right)$
correspond to $B_{r}\cap\left\{ z>0\right\} $ in Remark \ref{flow near saddle}.} $B_{r_{j}}^{+}\left(\Gamma\left(1\right)\right)$ and $B_{r_{j}}^{-}\left(\Gamma\left(-1\right)\right)$,
respectively, and are defined analogously as in (\ref{uniform cylindrical scale: half ball}). 
\item There exists a time $\mathring{t}_{j}>T_{1}$ so that for every $t\in\left[\tau_{j},\mathring{t}_{j}\right]$,\footnote{See (\ref{flow near saddle: intersection}).}
\[
\Sigma_{t}\cap\mathcal{P}_{+}\cap\mathcal{T}_{j}\neq\emptyset,\quad\Sigma_{t}\cap\mathcal{P}_{-}\cap\mathcal{T}_{j}\neq\emptyset.
\]
\item For every $t\in\left(T_{1},\mathring{t}_{j}\right]$, 
\[
\Sigma_{t}\cap\mathcal{\hat{T}}_{j}^{+}\subset\textrm{int}\,\mathscr{C}_{\phi}^{+},\quad\Sigma_{t}\cap\mathcal{\hat{T}}_{j}^{-}\subset\textrm{int}\,\mathscr{C}_{\phi}^{-},
\]
where $\mathscr{C}_{\phi}^{+}$ and $\mathscr{C}_{\phi}^{-}$ are
the cones\footnote{Here $\mathscr{C}_{\phi}^{+}$ and $\mathscr{C}_{\phi}^{-}$ correspond
to $\mathscr{C}_{\phi}$ in (\ref{flow near saddle: cone}).} defined analogously as in (\ref{cone}) for the endpoints.
\item There hold\footnote{See (\ref{flow near saddle: lower bound in cone}).}
\begin{equation}
\min_{\mathscr{C}_{\phi}^{+}\cap\mathcal{\hat{T}}_{j}^{+}}u\,>\,\tau_{j},\quad\min_{\mathscr{C}_{\phi}^{-}\cap\mathcal{\hat{T}}_{j}^{-}}u\,>\,\tau_{j}.\label{larger time in cone}
\end{equation}
\end{enumerate}
By the asymptotically cylindrical behavior of each point on $\mathcal{S}_{j}$
within the cylindrical scale $r_{j}$ and (\ref{larger time in cone}),
there exists $\tau_{j}'>\tau_{j}$ so that for $t\in\left[\tau_{j},\tau_{j}'\right]$,
$\bar{\Omega}_{t}\cap\mathcal{T}_{j}$ is also a tubular closed connected
region bounded by the tube-like hypersurface $\Sigma_{t}\cap\mathcal{T}_{j}$
and the two caps $\mathcal{P}_{+}\cap\mathcal{T}_{j},\mathcal{P}_{-}\cap\mathcal{T}_{j}$.

Note that even in the case where $\mathcal{S}_{j}$ is a single two-sided
saddle points (which can be regarded as a degenerate curve with one-sided
endpoints, see Definition \ref{singularity types}), the aforementioned
tubular region can be defined analogously.\footnote{In that case, $\mathcal{\tilde{P}}_{+}=\mathcal{\tilde{P}}_{-}$ and
$\mathcal{T}_{j}^{*}=\mathcal{\hat{T}}_{j}^{+}\cap\mathcal{\tilde{P}}_{+}$.} 

Furthermore, by Proposition \ref{C0 compactness}, Proposition \ref{pre-stability},
and Remark \ref{tube-like hypersurface}, when $k$ is large, $\Sigma_{\tau_{j}}^{k}$
and the two hyperplanes $\mathcal{P}_{+}$ and $\mathcal{P}_{-}$
would also bound a closed tubular region $\mathcal{T}_{j}^{k}$ in
$\mathcal{S}_{j}^{r_{j}}$; in addition, for $t\in\left[\tau_{j},\,\frac{1}{2}\left(\tau_{j}+\tau_{j}'\right)\right]$,
$\bar{\Omega}_{t}^{k}\cap\mathcal{T}_{j}^{k}$ is a tubular closed
connected region bounded by $\Sigma_{t}^{k}\cap\mathcal{T}_{j}^{k}$,
$\mathcal{P}_{+}\cap\mathcal{T}_{j}^{k}$, and $\mathcal{P}_{-}\cap\mathcal{T}_{j}^{k}$.
\end{rem}

The following proposition indicates that when $\mathcal{S}_{j}$ is
a singular component of the splitting type, the superlevel set $\Omega_{t}=\left\{ u>t\right\} $
becomes separated in a tubular neighborhood of $\mathcal{S}_{j}$
for $t>T_{1}=u\left(\mathcal{S}_{j}\right)$. 
\begin{prop}
\label{characterization of splitting } Let the notations be as defined
in Remark \ref{tubular region for splitting}. If 
\[
\mathcal{S}_{j}=\left\{ \Gamma\left(s\right)\,:\,s\in\left[-1,1\right]\right\} 
\]
is a singular component of $\left\{ \Sigma_{t}\right\} $ that belongs
to the splitting type, then there exists $\mathring{t}_{j}\in\left(T_{1},T_{2}\right)$
so that for every $t\in\left(T_{1},\mathring{t}_{j}\right]$, 
\[
\bar{\Omega}_{t}\cap\mathcal{T}_{j}\,\,\subset\,\,B_{\frac{3}{2}r_{j}}^{+}\left(\Gamma\left(1\right)\right)\,\cup\,B_{\frac{3}{2}r_{j}}^{-}\left(\Gamma\left(-1\right)\right),
\]
\[
\Sigma_{t}\cap\mathcal{P}_{+}\cap\mathcal{T}_{j}\neq\emptyset,\quad\Sigma_{t}\cap\mathcal{P}_{-}\cap\mathcal{T}_{j}\neq\emptyset.
\]
Note that 
\[
B_{\frac{3}{2}r_{j}}^{+}\left(\Gamma\left(1\right)\right)\,\cap\,B_{\frac{3}{2}r_{j}}^{-}\left(\Gamma\left(-1\right)\right)=\emptyset,
\]
\[
\mathcal{P}_{+}\cap\mathcal{T}_{j}\,\subset\,B_{\frac{3}{2}r_{j}}^{+}\left(\Gamma\left(1\right)\right),\quad\mathcal{P}_{-}\cap\mathcal{T}_{j}\,\subset\,B_{\frac{3}{2}r_{j}}^{-}\left(\Gamma\left(-1\right)\right).
\]
Moreover, the same results hold when $\mathcal{S}_{j}=\left\{ p\right\} $
is a single two-sided saddle point. In that case, set
\[
\Gamma\left(1\right)=\Gamma\left(-1\right)=p
\]
and let $B_{\frac{3}{2}r_{j}}^{+}\left(\Gamma\left(1\right)\right)$
and $B_{\frac{3}{2}r_{j}}^{-}\left(\Gamma\left(-1\right)\right)$
be the upper and lower open half-balls of $B_{\frac{3}{2}r_{j}}\left(p\right)$,
respectively (see Remark \ref{flow near saddle} and Remark \ref{tubular region for splitting}).
\end{prop}

\begin{proof}
The proofs for the case where $\mathcal{S}_{j}$ is a single two-sided
saddle point and the case where $\mathcal{S}_{j}$ is a curve with
one-sided saddle endpoints are in the same spirit, so without loss
of generality let us assume that 
\[
\mathcal{S}_{j}=\left\{ \Gamma\left(s\right)\,:\,s\in\left[-1,1\right]\right\} 
\]
is a curve with one-sided saddle endpoints .

By Proposition \ref{uniform cylindrical scale} we have 
\begin{equation}
\bar{B}_{3r_{j}}^{+}\left(\Gamma\left(1\right)\right)\cap\mathcal{S}_{j}\,=\,\left\{ \Gamma\left(1\right)\right\} ,\quad\bar{B}_{3r_{j}}^{-}\left(\Gamma\left(-1\right)\right)\cap\mathcal{S}_{j}\,=\,\left\{ \Gamma\left(-1\right)\right\} ,\label{characterization of splitting: half balls}
\end{equation}
where the half balls are defined as in Remark \ref{tubular region for splitting}.
Now let us apply the results in Remark \ref{tubular region for splitting}.
There exists a time $\mathring{t}_{j}\in\left(T_{1},T_{2}\right)$
so that for every $t\in\left(\tau_{j},\mathring{t}_{j}\right]$, 
\[
\Sigma_{t}\cap\mathcal{P}_{+}\cap\mathcal{T}_{j}\neq\emptyset,\quad\Sigma_{t}\cap\mathcal{P}_{-}\cap\mathcal{T}_{j}\neq\emptyset.
\]
Note that the caps satisfy
\[
\mathcal{P}_{+}\cap\mathcal{T}_{j}\,=\,\mathcal{P}_{+}\cap\mathcal{\hat{T}}_{j}^{+}\,\subset\,\mathcal{P}_{+}\cap\bar{B}_{r_{j}}^{+}\left(\Gamma\left(1\right)\right)\,\subset\,B_{\frac{3}{2}r_{j}}^{+}\left(\Gamma\left(1\right)\right),
\]
\[
\mathcal{P}_{-}\cap\mathcal{T}_{j}\,=\,\mathcal{P}_{-}\cap\mathcal{\hat{T}}_{j}^{-}\,\subset\,\mathcal{P}_{-}\cap\bar{B}_{r_{j}}^{-}\left(\Gamma\left(-1\right)\right)\,\subset\,B_{\frac{3}{2}r_{j}}^{-}\left(\Gamma\left(-1\right)\right).
\]
Moreover, for every $t\in\left(T_{1},\mathring{t}_{j}\right]$, 
\begin{equation}
\bar{\Omega}_{t}\cap\mathcal{\hat{T}}_{j}^{+}\,\subset\,\textrm{int}\,\mathscr{C}_{\phi}^{+}\cap\bar{B}_{r_{j}}^{+}\left(\Gamma\left(1\right)\right)\,\subset\,\textrm{int}\,\mathscr{C}_{\phi}^{+}\cap B_{\frac{3}{2}r_{j}}^{+}\left(\Gamma\left(1\right)\right),\label{characterization of splitting: disjoint}
\end{equation}
\[
\bar{\Omega}_{t}\cap\mathcal{\hat{T}}_{j}^{-}\,\subset\,\textrm{int}\,\mathscr{C}_{\phi}^{-}\cap\bar{B}_{r_{j}}^{-}\left(\Gamma\left(-1\right)\right)\,\subset\,\textrm{int}\,\mathscr{C}_{\phi}^{-}\cap B_{\frac{3}{2}r_{j}}^{-}\left(\Gamma\left(-1\right)\right).
\]
In view of the asymptotically cylindrical behavior of $\left\{ \Sigma_{t}\right\} $
about each point on $\mathcal{S}_{j}$ within the cylindrical scale
$r_{j}$, we deduce that for every $t\in\left(T_{1},\mathring{t}_{j}\right]$,
\[
\bar{\Omega}_{t}\cap\mathcal{T}_{j}^{*}=\emptyset
\]
(see Remark \ref{one-sided local max}) and so 
\[
\bar{\Omega}_{t}\cap\mathcal{T}_{j}\,=\,\left(\bar{\Omega}_{t}\cap\mathcal{\hat{T}}_{j}^{+}\right)\cup\left(\bar{\Omega}_{t}\cap\mathcal{\hat{T}}_{j}^{-}\right);
\]
it follows from (\ref{characterization of splitting: disjoint}) that
\[
\bar{\Omega}_{t}\cap\mathcal{T}_{j}\,=\,\left(\bar{\Omega}_{t}\cap\mathcal{\hat{T}}_{j}^{+}\right)\cup\left(\bar{\Omega}_{t}\cap\mathcal{\hat{T}}_{j}^{-}\right)\,\subset\,\,B_{\frac{3}{2}r_{j}}^{+}\left(\Gamma\left(1\right)\right)\,\cup\,B_{\frac{3}{2}r_{j}}^{-}\left(\Gamma\left(-1\right)\right).
\]
Note that by (\ref{characterization of splitting: half balls}) we
have
\[
B_{\frac{3}{2}r_{j}}^{+}\left(\Gamma\left(1\right)\right)\,\cap\,B_{\frac{3}{2}r_{j}}^{-}\left(\Gamma\left(-1\right)\right)=\emptyset.
\]
\end{proof}
Recall that in the end of Remark \ref{tubular region for splitting},
it is mentioned that when $k$ is large, $\bar{\Omega}_{\tau_{j}}^{k}\cap\mathcal{T}_{j}^{k}$
is a connected solid tube. The following lemma says that if $\bar{\Omega}_{t_{*}}^{k}\cap\mathcal{T}_{j}^{k}$
becomes disconnected for some $t_{*}>\tau_{j}$, then under some conditions
such as (\ref{existence of singular time: boundary condition}), the
flow $\left\{ \Sigma_{t}^{k}\right\} $ would become singular in $\mathcal{T}_{j}^{k}$
at some time $\hat{t}\in\left(\tau_{j},t_{*}\right)$.
\begin{lem}
\label{existence of singular time}Let $k$ be a large integer so
that for $t$ close to $\tau_{j}$, $\Sigma_{t}^{k}\cap\mathcal{S}_{j}^{r_{j}}$
is a tube-like hypersurface around $\mathcal{S}_{j}$ and $\bar{\Omega}_{t}^{k}\cap\mathcal{T}_{j}^{k}$
is a tubular closed connected region (see Remark \ref{tubular region for splitting}).
If there exists $t_{*}>\tau_{j}$ so that the following hold:
\begin{enumerate}
\item $\Sigma_{t_{*}}^{k}\cap\mathcal{T}_{j}^{k}$ has no singular points
of the flow $\left\{ \Sigma_{t}^{k}\right\} .$
\item $\bar{\Omega}_{t_{*}}^{k}\cap\mathcal{T}_{j}^{k}$ is disconnected
in such a way that every connected component of $\bar{\Omega}_{t_{*}}^{k}\cap\mathcal{T}_{j}^{k}$
intersects at most one of the caps (i.e., $\mathcal{P}_{+}\cap\mathcal{T}_{j}^{k}$
and $\mathcal{P}_{-}\cap\mathcal{T}_{j}^{k}$); specifically, there
exists one component that intersects $\mathcal{P}_{+}\cap\mathcal{T}_{j}^{k}$,
and there exists another component that does not intersect $\mathcal{P}_{+}\cap\mathcal{T}_{j}^{k}$.
\item The unit normal vector $\frac{\nabla u^{k}}{\left|\nabla u^{k}\right|}$
satisfies
\begin{equation}
\frac{\nabla u^{k}}{\left|\nabla u^{k}\right|}\cdot\nu_{+}\neq-1\,\,\,\textrm{on}\,\,\,\mathcal{P}_{+}\cap\mathcal{T}_{j}^{k}\cap\left\{ \tau_{j}<u^{k}<t_{*}\right\} ,\label{existence of singular time: boundary condition}
\end{equation}
\[
\frac{\nabla u^{k}}{\left|\nabla u^{k}\right|}\cdot\nu_{-}\neq-1\,\,\,\textrm{on}\,\,\,\mathcal{P}_{-}\cap\mathcal{T}_{j}^{k}\cap\left\{ \tau_{j}<u^{k}<t_{*}\right\} ,
\]
where $\nu_{+}$ and $\nu_{-}$ are the outward (i.e., pointing toward
the outside of $\mathcal{T}_{j}^{k}$) unit normal vector of the caps
$\mathcal{P}_{+}\cap\mathcal{T}_{j}^{k}$ and $\mathcal{P}_{-}\cap\mathcal{T}_{j}^{k}$,
respectively. 
\end{enumerate}
Then there is $\hat{t}\in\left(\tau_{j},t_{*}\right)$ such that $\Sigma_{\hat{t}}^{k}\cap\mathcal{T}_{j}^{k}$
has singular points of $\left\{ \Sigma_{t}^{k}\right\} $. 
\end{lem}

\begin{proof}
First of all, note that the set of regular values of $\left.u^{k}\right|_{\mathcal{T}_{j}^{k}}$,
namely,
\begin{equation}
\left\{ t\,:\,\nabla u^{k}\left(x\right)\neq0\,\,\,\textrm{whenever}\,\,\,x\in\mathcal{T}_{j}^{k}\,\,\,\textrm{with}\,\,\,u\left(x\right)=t\right\} \label{existence of singular time: regular values}
\end{equation}
is open because the set of critical values of $\left.u^{k}\right|_{\mathcal{T}_{j}^{k}}$
is compact (see Lemma \ref{closedness of singular set} or Assumption
\ref{local singular times hypothesis}). Note also that the minimum
of $u^{k}$ on any compact set can only be attained on the boundary
since $u^{k}$ has no interior local minimum points (see Corollary
\ref{isolated round point} and Corollary \ref{saddle criterion}).

Let $\bar{\Omega}_{t}^{k,+}$ be the union connected components of
$\bar{\Omega}_{t}^{k}\cap\mathcal{T}_{j}^{k}$ that intersect $\mathcal{P}_{+}\cap\mathcal{T}_{j}^{k}$
and let $\bar{\Omega}_{t}^{k,-}$ be the union the remaining connected
components of $\bar{\Omega}_{t}^{k}\cap\mathcal{T}_{j}^{k}$ . Define
$\mathfrak{T}$ as the set of time $t\in\left[\tau_{j},t_{*}\right)$
such that
\begin{itemize}
\item $\Sigma_{t}^{k}\cap\mathcal{T}_{j}^{k}$ has no singular points of
$\left\{ \Sigma_{t}^{k}\right\} $;
\item $\bar{\Omega}_{t}^{k,+}\cap\mathcal{P}_{-}\cap\mathcal{T}_{j}^{k}=\emptyset$
and $\bar{\Omega}_{t}^{k,-}\cap\mathcal{P}_{+}\cap\mathcal{T}_{j}^{k}=\emptyset.$
\end{itemize}
Let 
\begin{equation}
\hat{t}=\inf\left\{ \tilde{t}\in\left(\tau_{j},t_{*}\right]\,:\,\left[\tilde{t},t_{*}\right]\subset\mathfrak{T}\right\} .\label{existence of singular time: supremum}
\end{equation}
By the assumptions and (\ref{existence of singular time: regular values}),
it is not hard to see that $\hat{t}\in\left(\tau_{j},t_{*}\right)$.
We claim that $\Sigma_{\hat{t}}^{k}\cap\mathcal{T}_{j}^{k}$ must
have singular points of $\left\{ \Sigma_{t}^{k}\right\} $. 

Suppose the contrary that $\Sigma_{\hat{t}}^{k}\cap\mathcal{T}_{j}^{k}$
has no singular points of $\left\{ \Sigma_{t}^{k}\right\} $. Then
every time $t\in\left[\hat{t},t_{*}\right]$ would be a regular value
of $\left.u^{k}\right|_{\mathcal{T}_{j}^{k}}$; in fact, by (\ref{existence of singular time: regular values})
we can even find $\hat{t}'<\hat{t}$ so that
\begin{equation}
r\coloneqq\min_{\left\{ \hat{t}'\leq u^{k}\leq t_{*}\right\} \cap\mathcal{T}_{j}^{k}}\left|\nabla u^{k}\right|\,>\,0.\label{existence of singular time: scale}
\end{equation}
Note that $\bar{\Omega}_{\hat{t}}^{k}\cap\textrm{int}\,\mathcal{T}_{j}^{k}$
is contained in $\bar{\Omega}^{+}\cup\bar{\Omega}^{-}$, where 
\[
\bar{\Omega}^{+}\coloneqq\overline{\bigcup_{t\in\left(\hat{t},t_{*}\right]}\,\bar{\Omega}_{t}^{k,+}},\quad\bar{\Omega}^{-}\coloneqq\overline{\bigcup_{t\in\left(\hat{t},t_{*}\right]}\,\bar{\Omega}_{t}^{k,-}}.
\]
Moreover, it follows from (\ref{existence of singular time: scale}),
Definition \ref{smooth scale}, Proposition \ref{interior estimate},
Proposition \ref{smooth estimate}, and Corollary \ref{local graph estimate}
that there exists a sequence $t_{i}\searrow\hat{t}$ such that
\begin{equation}
\overline{\partial\bar{\Omega}_{t_{i}}^{k,+}\cap\mathcal{\textrm{int}\,}\mathcal{T}_{j}^{k}}\stackrel{C^{\infty}}{\longrightarrow}\,\Sigma^{+},\quad\overline{\partial\bar{\Omega}_{t_{i}}^{k,-}\cap\mathcal{\textrm{int}\,}\mathcal{T}_{j}^{k}}\stackrel{C^{\infty}}{\longrightarrow}\,\Sigma^{-}\,\,\,\textrm{as}\,\,i\rightarrow\infty,\label{existence of singular time: continuity}
\end{equation}
where $\Sigma^{+}$ and $\Sigma^{-}$ are hypersurfaces contained
in $\Sigma_{\hat{t}}^{k}\cap\mathcal{T}_{j}^{k}$ with boundary on
$\mathcal{P}_{+}\cap\mathcal{T}_{j}^{k}$ and $\mathcal{P}_{-}\cap\mathcal{T}_{j}^{k}$.
We have
\[
\overline{\partial\Omega^{+}\cap\mathcal{\textrm{int}\,}\mathcal{T}_{j}^{k}}\subset\Sigma^{+},\quad\overline{\partial\Omega^{-}\cap\mathcal{\textrm{int}\,}\mathcal{T}_{j}^{k}}\subset\Sigma^{-}.
\]
Furthermore, it is true that 
\begin{equation}
\bar{\Omega}^{+}\cap\mathcal{P}_{-}\cap\mathcal{T}_{j}^{k}=\emptyset,\quad\bar{\Omega}^{-}\cap\mathcal{P}_{+}\cap\mathcal{T}_{j}^{k}=\emptyset.\label{existence of singular time: away from caps}
\end{equation}
Because if one of them were false, say $\bar{\Omega}^{+}\cap\mathcal{P}_{-}\cap\mathcal{T}_{j}^{k}\neq\emptyset$,
then we would have $q\in\bar{\Omega}^{+}\cap\mathcal{P}_{-}\cap\mathcal{T}_{j}^{k}$.
By (\ref{existence of singular time: continuity}) and the condition
that 
\[
\bar{\Omega}_{t_{i}}^{k,+}\cap\mathcal{P}_{-}\cap\mathcal{T}_{j}^{k}=\emptyset\quad\forall\,i,
\]
$\bar{\Omega}^{+}$ would ``touch'' $\mathcal{P}_{-}\cap\mathcal{T}_{j}^{k}$
at $q$ since $\Sigma^{+}$ is on one side\footnote{The inside of the tube $\mathcal{T}_{j}^{k}$.}
of the cap $\mathcal{P}_{-}\cap\mathcal{T}_{j}^{k}$ with $q\in\Sigma^{+}\cap\mathcal{P}_{-}\cap\mathcal{T}_{j}^{k}$,
which yields that 
\[
\frac{\nabla u^{k}\left(q\right)}{\left|\nabla u^{k}\left(q\right)\right|}\cdot\nu_{-}=-1,
\]
which is in contradiction with (\ref{existence of singular time: boundary condition}).

Most importantly, we should have
\[
\liminf_{i\rightarrow\infty}\,\,\textrm{dist}\left(\bar{\Omega}_{t_{i}}^{k,+},\,\bar{\Omega}_{t_{i}}^{k,-}\right)=0;
\]
otherwise, upon passing to a subsequence, there would exist $\varrho>0$
such that
\[
\textrm{dist}\left(\bar{\Omega}_{t_{i}}^{k,+},\,\bar{\Omega}_{t_{i}}^{k,-}\right)\,\geq\,\varrho\quad\forall\,i,
\]
which yields that $\bar{\Omega}^{+}$ and $\bar{\Omega}^{-}$ would
be disjoint, which would lead to a contradiction with the choice of
$\hat{t}$ in view of (\ref{existence of singular time: scale}).
Thus, for each $i$ we can choose $p_{i}^{+}\in\partial\bar{\Omega}_{t_{i}}^{k,+}\cap\mathcal{\textrm{int}\,}\mathcal{T}_{j}^{k}$
and $p_{i}^{-}\in\partial\bar{\Omega}_{t_{i}}^{k,-}\cap\mathcal{\textrm{int}\,}\mathcal{T}_{j}^{k}$
so that 
\[
\liminf_{i\rightarrow\infty}\,\,\left|p_{i}^{+}-p_{i}^{-}\right|=0.
\]
By passing to a subsequence, we may assume that 
\begin{equation}
\lim_{i\rightarrow\infty}p_{i}^{+}=\lim_{i\rightarrow\infty}p_{i}^{-}=p.\label{existence of singular time: approximate sequence}
\end{equation}
By hypothesis $p$ is a regular point of $\left\{ \Sigma_{t}^{k}\right\} $
(since $u^{k}\left(p\right)=\hat{t}$ is a regular value of $\left.u^{k}\right|_{\mathcal{T}_{j}^{k}}$);
additionally, $p$ is neither on the lateral boundary of $\mathcal{T}_{j}^{k}$
(because $u^{k}\left(p\right)=\hat{t}>\tau_{j}$) nor on the caps
(in view of (\ref{existence of singular time: away from caps})).
Thus, we can choose $r>0$ sufficiently small so that $B_{r}\left(p\right)\subset\mathcal{\textrm{int}\,}\mathcal{T}_{j}^{k}$
and that every level set of $u^{k}$ in $B_{r}\left(p\right)$ is
a small Lipschitz graph over $T_{p}\Sigma_{\hat{t}}^{k}$ (see Definition
\ref{smooth scale}). 

By (\ref{existence of singular time: approximate sequence}), when
$i$ is large the points $p_{i}^{+}$ and $p_{i}^{-}$ would be in
$B_{r}\left(p\right)$. Since $B_{r}\left(p\right)\cap\left\{ u^{k}\geq t_{i}\right\} $
is path-connected, there exists a path in $B_{r}\left(p\right)\cap\left\{ u^{k}\geq t_{i}\right\} $
that joining $p_{i}^{+}$ and $p_{i}^{-}$ together. This is in contradiction
to the definitions of $\bar{\Omega}_{t_{i}}^{k,+}$ and $\bar{\Omega}_{t_{i}}^{k,-}$.
Therefore, we conclude that $\left\{ \Sigma_{t}^{k}\right\} $ must
have singular points on $\Sigma_{\hat{t}}^{k}\cap\mathcal{T}_{j}^{k}$. 
\end{proof}
Lemma \ref{existence of singular time} will be used in the proofs
of Proposition \ref{classification of singularities near splitting type }
and Proposition \ref{rule out bumpy type}. In order that the condition
(\ref{existence of singular time: boundary condition}) holds so as
to apply Lemma \ref{existence of singular time}, we make the following
assumption for every splitting component of $\left\{ \Sigma_{t}\right\} $.
\begin{assumption}
\label{not getting into}Whenever $\mathcal{S}_{j}$ is a singular
component of $\left\{ \Sigma_{t}\right\} $ that belongs to the splitting
type, we assume that the flow $\left\{ \Sigma_{t}\right\} $ ``does
not get into $\mathcal{S}_{j}$ near the endpoints'' in the sense
that
\[
\min_{\mathcal{P}_{+}\cap\mathcal{T}_{j}}\,\frac{\nabla u}{\left|\nabla u\right|}\cdot\nu_{+}>-1,\quad\min_{\mathcal{P}_{-}\cap\mathcal{T}_{j}}\,\frac{\nabla u}{\left|\nabla u\right|}\cdot\nu_{-}>-1,
\]
where $\nu_{+}$ and $\nu_{-}$ are the outward (i.e., pointing toward
the outside of $\mathcal{T}_{j}$) unit normal vector of the caps
$\mathcal{P}_{+}\cap\mathcal{T}_{j}$ and $\mathcal{P}_{-}\cap\mathcal{T}_{j}$,
respectively. Note that $\nabla u$ vanishes at no points on the caps
$\mathcal{P}_{+}\cap\mathcal{T}_{j}$ and $\mathcal{P}_{-}\cap\mathcal{T}_{j}$
(see Remark \ref{tubular region for splitting}) and that 
\[
\frac{\nabla u}{\left|\nabla u\right|}\cdot\nu_{\pm}\approx0
\]
on the part of $\mathcal{P}_{\pm}\cap\mathcal{T}_{j}$ that is close
to the lateral boundary $\Sigma_{\tau_{j}}\cap\mathcal{T}_{j}$ due
to the asymptotically cylindrical behavior of the flow $\left\{ \Sigma_{t}\right\} $. 

It follows from Proposition \ref{C1 compactness} that when $k$ is
large, 
\[
\min_{\mathcal{P}_{+}\cap\mathcal{T}_{j}^{k}}\,\frac{\nabla u^{k}}{\left|\nabla u^{k}\right|}\cdot\nu_{+}>-1,\quad\min_{\mathcal{P}_{-}\cap\mathcal{T}_{j}^{k}}\,\frac{\nabla u^{k}}{\left|\nabla u^{k}\right|}\cdot\nu_{-}>-1.
\]
\end{assumption}

We are now in a position to prove the main results of this subsection.
\begin{prop}
\label{classification of singularities near splitting type }Let $\mathcal{S}_{j}$
be a singular component of the flow $\left\{ \Sigma_{t}\right\} $
that belongs to the splitting type. Then for every sufficiently large
$k$, there exist singular components of $\left\{ \Sigma_{t}^{k}\right\} $
in $\mathcal{S}_{j}^{\hat{\delta}}$ and they all belong to the splitting/bumpy
type; moreover, if one of the singular components of $\left\{ \Sigma_{t}^{k}\right\} $
in $\mathcal{S}_{j}^{\hat{\delta}}$ is of the splitting type, then
it would be the unique singular component of $\left\{ \Sigma_{t}^{k}\right\} $
in $\mathcal{S}_{j}^{\hat{\delta}}$.
\end{prop}

\begin{proof}
Let $k$ be a sufficiently large integer. By Proposition \ref{C0 compactness},
Proposition \ref{pre-stability}, Proposition \ref{characterization of splitting },
Assumption \ref{not getting into}, and Lemma \ref{existence of singular time},
there exists 
\[
\hat{t}\in\left(\tau_{j},\,\,\frac{1}{2}\left(T_{1}+\mathring{t}_{j}\right)\right)
\]
so that $\Sigma_{\hat{t}}^{k}\cap\mathcal{T}_{j}^{k}$ has singular
points of $\left\{ \Sigma_{t}^{k}\right\} $, where $\mathring{t}_{j}\in\left(T_{1},T_{2}\right)$
is the time given in Proposition \ref{characterization of splitting }.
It follows from Assumption \ref{local singular times hypothesis}
that 
\[
\hat{t}=T_{1,j}^{k}
\]
is the unique singular time of $\left\{ \Sigma_{t}^{k}\right\} $
in $\mathcal{T}_{j}^{k}\subset\mathcal{S}_{j}^{\hat{\delta}}$. Note
that by Corollary \ref{regular space} and Remark \ref{tubular region for splitting},
all the singular points of $\left\{ \Sigma_{t}^{k}\right\} $ in $\mathcal{S}_{j}^{\hat{\delta}}$
would be contained in $\textrm{int}\,\mathcal{T}_{j}^{k}$.

By Assumption \ref{local singular times hypothesis} and Proposition
\ref{global maximum points } (see also Proposition \ref{pre-stability}),
there are finitely many connected components of the singular set of
$\left\{ \Sigma_{t}^{k}\right\} $ in $\mathcal{T}_{j}^{k}$, each
of which belongs to the splitting or bumpy types (in particular, there
are no round points of $\left\{ \Sigma_{t}^{k}\right\} $ in $\mathcal{T}_{j}^{k}$).

Now assume that there is a singular component $\mathcal{\dot{S}}_{j}^{k}$
of $\left\{ \Sigma_{t}^{k}\right\} $ in $\mathcal{T}_{j}^{k}$ which
belongs to the splitting type. By Proposition \ref{uniform cylindrical scale},
any point $q\in\mathcal{\dot{S}}_{j}^{k}$ has a cylindrical scale
$r_{j}$ and every other cylindrical point of $\left\{ \Sigma_{t}^{k}\right\} $
in $B_{r_{j}}\left(q\right)$ is located in a small Lipschitz graph
over the axis of the tangent cylinder of $\left\{ \Sigma_{t}^{k}\right\} $
at $q$. Applying Remark \ref{tube-like hypersurface} and Remark
\ref{tubular region for splitting} (see also Assumption \ref{local singular times hypothesis})
to the singular component $\mathcal{\dot{S}}_{j}^{k}$ of $\left\{ \Sigma_{t}^{k}\right\} $,
we can find two hyperplanes $\mathcal{\dot{P}}_{+}^{k}$ and $\mathcal{\dot{P}}_{-}^{k}$
near the endpoints\footnote{In the case where $\mathcal{\dot{S}}_{j}^{k}$ is a single point,
$\mathcal{\dot{P}}_{+}^{k}$ and $\mathcal{\dot{P}}_{-}^{k}$ would
two parallel hyperplanes on each ``side'' of the point, see Remark
\ref{flow near saddle} and Remark \ref{tubular region for splitting}.} of $\mathcal{\dot{S}}_{j}^{k}$ that are orthogonal to the tangent
cylinders of $\left\{ \Sigma_{t}^{k}\right\} $ at the endpoints,
respectively, such that
\begin{itemize}
\item there are no singular points of $\left\{ \Sigma_{t}^{k}\right\} $
on $\mathcal{\dot{P}}_{+}^{k}\cap\mathcal{T}_{j}^{k}$ and $\mathcal{\dot{P}}_{-}^{k}\cap\mathcal{T}_{j}^{k}$
(see Proposition \ref{uniform cylindrical scale});
\item $\Sigma_{\tau_{j}}^{k}$ is a tube-like hypersurface around $\mathcal{\dot{S}}_{j}^{k}$
and it together with the two caps, $\mathcal{\dot{P}}_{+}^{k}\cap\mathcal{T}_{j}^{k}$
and $\mathcal{\dot{P}}_{-}^{k}\cap\mathcal{T}_{j}^{k}$, bound a ``truncated''
tubular closed region $\mathcal{\dot{T}}_{j}^{k}\subset\mathcal{T}_{j}^{k}$
satisfying $\mathcal{\dot{S}}_{j}^{k}\subset\textrm{int}\,\mathcal{\dot{T}}_{j}^{k}$;
\item $\mathcal{\dot{S}}_{j}^{k}$ is the unique singular component of $\left\{ \Sigma_{t}^{k}\right\} $
in $\mathcal{\dot{T}}_{j}^{k}$ (see Proposition \ref{uniform cylindrical scale}). 
\end{itemize}
Next, we would like to prove (by contradiction) that $\mathcal{\dot{S}}_{j}^{k}$
is indeed the only singular component of $\left\{ \Sigma_{t}^{k}\right\} $
in $\mathcal{T}_{j}^{k}$ (instead of only in $\mathcal{\dot{T}}_{j}^{k}$).
For convenience, let us assume that the ``orientation'' is chosen
in such a way that $\mathcal{\dot{P}}_{+}^{k}$ is adjacent to $\mathcal{P}_{+}$
and $\mathcal{\dot{P}}_{-}^{k}$ is adjacent to $\mathcal{P}_{-}$.
Thus, the tube $\mathcal{T}_{j}^{k}$ would be cut by $\mathcal{\dot{P}}_{+}^{k}$
and $\mathcal{\dot{P}}_{-}^{k}$ into three closed pieces: 
\[
\mathcal{T}_{j}^{k}\,=\,\mathcal{\hat{T}}_{j}^{k+}\cup\mathcal{\dot{T}}_{j}^{k}\cup\mathcal{\hat{T}}_{j}^{k-},
\]
where $\mathcal{\hat{T}}_{j}^{k+}$ is the tubular closed region bounded
by $\Sigma_{\tau_{j}}^{k}$ (from the lateral) and the two caps $\mathcal{P}_{+}\cap\mathcal{T}_{j}^{k}$
and $\mathcal{\dot{P}}_{+}^{k}\cap\mathcal{T}_{j}^{k}$ (from the
two ends); likewise, $\mathcal{\hat{T}}_{j}^{k-}$ is the closed region
bounded by $\Sigma_{\tau_{j}}^{k}$, $\mathcal{P}_{-}\cap\mathcal{T}_{j}^{k}$,
and $\mathcal{\dot{P}}_{-}^{k}\cap\mathcal{T}_{j}^{k}$.

Suppose the contrary that there exists another singular component
$\mathcal{\ddot{S}}_{j}^{k}$ in $\mathcal{T}_{j}^{k}$. Since $\mathcal{\ddot{S}}_{j}^{k}\cap\mathcal{\dot{T}}_{j}^{k}=\emptyset$,
we may assume without loss of generality that $\mathcal{\ddot{S}}_{j}^{k}\subset\mathcal{\hat{T}}_{j}^{k+}$.
Following Remark \ref{tubular region for splitting}, let $\mathcal{\tilde{P}}_{+}^{k}$
be the hyperplane that passes through the endpoint of $\mathcal{\dot{S}}_{j}^{k}$
(the one that is near the side of $\mathcal{\dot{P}}_{+}^{k}$) and
is parallel to $\mathcal{\dot{P}}_{+}^{k}$ (i.e., orthogonal to the
tangent cylinder of $\left\{ \Sigma_{t}^{k}\right\} $ at the endpoint).
Fix a point $p\in\mathcal{\ddot{S}}_{j}^{k}$ and let $\mathcal{\ddot{P}}$
be the hyperplane passing through $p$ and orthogonal to the tangent
cylinder of $\left\{ \Sigma_{t}^{k}\right\} $ at $p$. Define $\mathcal{\ddot{T}}_{j}^{k}$
as the tubular closed region bounded by $\Sigma_{\tau_{j}}^{k}$ (from
the lateral) and the two caps $\mathcal{\tilde{P}}_{+}^{k}\cap\mathcal{T}_{j}^{k}$
and $\mathcal{\ddot{P}}\cap\mathcal{T}_{j}^{k}$ (from the two ends).
Notice that $u^{k}=\tau_{j}<T_{1,j}^{k}$ on the lateral boundary
of the tube $\mathcal{\ddot{T}}_{j}^{k}$ and that 
\[
\max_{\mathcal{\tilde{P}}_{+}^{k}\cap\mathcal{T}_{j}^{k}}u^{k}=\max_{\mathcal{\ddot{P}}\cap\mathcal{T}_{j}^{k}}u^{k}=T_{1,j}^{k},
\]
in light of the asymptotically cylindrical behavior of $\left\{ \Sigma_{t}^{k}\right\} $
(see (\ref{saddle criterion: max})). 

On the other hand, applying Remark \ref{flow near saddle} to the
saddle endpoint of $\mathcal{\dot{S}}_{j}^{k}$ on $\mathcal{\tilde{P}}_{+}^{k}\cap\mathcal{T}_{j}^{k}$,
we would find $\mathring{t}_{j}^{k}>T_{1,j}^{k}$ so that 
\[
\Sigma_{t}^{k}\cap\mathcal{\ddot{T}}_{j}^{k}\neq\emptyset\quad\forall\,t\in\left(T_{1,j}^{k},\,\mathring{t}_{j}^{k}\right].
\]
Let $\hat{\Sigma}_{\mathring{t}_{j}^{k}}^{k}$ be any connected component
of $\Sigma_{\mathring{t}_{j}^{k}}^{k}\cap\mathcal{\ddot{T}}_{j}^{k}$.
Note that $\hat{\Sigma}_{\mathring{t}_{j}^{k}}^{k}$ is away from
the boundary $\partial\mathcal{\ddot{T}}_{j}^{k}$ since 
\[
\max_{\partial\mathcal{\ddot{T}}_{j}^{k}}u^{k}=T_{1,j}^{k}.
\]
Thus, $\hat{\Sigma}_{\mathring{t}_{j}^{k}}^{k}$ would be a closed
hypersurface in $\mathcal{\ddot{T}}_{j}^{k}$ and hence enclose an
open region $\hat{\Omega}_{\mathring{t}_{j}^{k}}^{k}$ in $\mathcal{\ddot{T}}_{j}^{k}$.
As $u^{k}=\mathring{t}_{j}^{k}$ on $\partial\hat{\Omega}_{\mathring{t}_{j}^{k}}^{k}=\hat{\Sigma}_{\mathring{t}_{j}^{k}}^{k}$,
the maximum value
\[
\tilde{T}_{1,j}^{k}\coloneqq\max_{\overline{\hat{\Omega}_{\mathring{t}_{j}^{k}}^{k}}}\,u^{k}\,\geq\,\mathring{t}_{j}^{k}\,>\,T_{1,j}^{k}
\]
would be attained at some (interior) point in $\hat{\Omega}_{\mathring{t}_{j}^{k}}^{k}$,
which would be a critical point of $u^{k}$. It follows that $\tilde{T}_{1,j}^{k}$
is another singular time of $\left\{ \Sigma_{t}^{k}\right\} $ in
$\mathcal{T}_{j}^{k}\subset\mathcal{S}_{j}^{\hat{\delta}}$, contradicting
Assumption \ref{local singular times hypothesis}. Therefore, there
are no other singular components in $\mathcal{T}_{j}^{k}$.
\end{proof}
One last step to establish the ``stability'' of singular components
of the splitting type is to rule out the existence of the bumpy components
in Proposition \ref{classification of singularities near splitting type }.
\begin{prop}
\label{rule out bumpy type}In fact, in Proposition \ref{classification of singularities near splitting type }
there are no singular components of $\left\{ \Sigma_{t}^{k}\right\} $
that belong to the bumpy type. Therefore, $\left\{ \Sigma_{t}^{k}\right\} $
has precisely one singular component in $\mathcal{S}_{j}^{\hat{\delta}}$,
which is of the splitting type.
\end{prop}

\begin{proof}
Suppose that there is a singular component $\mathcal{\mathring{S}}_{j}^{k}$
of $\left\{ \Sigma_{t}^{k}\right\} $ in $\mathcal{T}_{j}^{k}$ belonging
to the bumpy type. When $\mathcal{\mathring{S}}_{j}^{k}$ is a curve,
let $p$ be the local maximum endpoint of $\mathcal{\mathring{S}}_{j}^{k}$
(see Definition \ref{singularity types}) and let us assume (without
loss of generality) that $p$ is ``adjacent'' to $\mathcal{P}_{+}$
(so the other saddle endpoint is adjacent to $\mathcal{P}_{-}$).
When $\mathcal{\mathring{S}}_{j}^{k}$ is a single one-sided saddle
point, let $p$ be $\mathcal{\mathring{S}}_{j}^{k}$ itself and let
us assume that $u^{k}$ is a local maximum point ``on the side toward
$\mathcal{P}_{+}$'' (see Definition \ref{one-sided saddle}). 

Recall that $\mathcal{S}_{j}$ is a singular component of $\left\{ \Sigma_{t}\right\} $
that belongs to the splitting type. By Proposition \ref{characterization of splitting }
there exists $\mathring{t}_{j}\in\left(T_{1},T_{2}\right)$ so that
\[
\bar{\Omega}_{t}\cap\mathcal{P}_{+}\cap\mathcal{T}_{j}\,\supset\,\Sigma_{t}\cap\mathcal{P}_{+}\cap\mathcal{T}_{j}\,\neq\,\emptyset\quad\forall\,t\in\left(T_{1},\,\mathring{t}_{j}\right].
\]
By Proposition \ref{C0 compactness} and Proposition \ref{pre-stability},
if $k$ is sufficiently, we may assume that $\mathring{t}_{j}>T_{1,j}^{k}$
and that
\[
\bar{\Omega}_{\frac{1}{2}\left(T_{1,j}^{k}+\mathring{t}_{j}\right)}^{k}\cap\mathcal{P}_{+}\cap\mathcal{T}_{j}^{k}\,\neq\,\emptyset,
\]
where $T_{1,j}^{k}\approx T_{1}$ is the unique singular time of $\left\{ \Sigma_{t}^{k}\right\} $
in $\mathcal{T}_{j}^{k}\subset\mathcal{S}_{j}^{\hat{\delta}}$ (see
Proposition \ref{classification of singularities near splitting type }).
Then applying the intermediate value theorem to $\left.u^{k}\right|_{\mathcal{P}_{+}\cap\mathcal{T}_{j}^{k}}$
implies that
\begin{equation}
\Sigma_{t}^{k}\cap\mathcal{P}_{+}\cap\mathcal{T}_{j}^{k}\neq\emptyset\quad\forall\,t\in\left[\tau_{j},\,\,\frac{1}{2}\left(T_{1,j}^{k}+\mathring{t}_{j}\right)\right].\label{rule out bumpy type: upper intersection}
\end{equation}
In addition, by Assumption \ref{not getting into} there holds 
\begin{equation}
\min_{\mathcal{P}_{+}\cap\mathcal{T}_{j}^{k}}\,\frac{\nabla u^{k}}{\left|\nabla u^{k}\right|}\cdot\nu_{+}>-1\label{rule out bumpy type: upper normal}
\end{equation}
where $\nu_{+}$ is the unit normal vector of $\mathcal{P}_{+}\cap\mathcal{T}_{j}^{k}$
that points toward the outside of the tube $\mathcal{T}_{j}^{k}$.

On the other hand, using Remark \ref{one-sided isolated local max}
and Proposition \ref{uniform cylindrical scale}, we could find two
parallel hyperplanes, $\mathcal{\tilde{P}}_{+}^{k}$ and $\mathcal{\mathring{P}}_{+}^{k}$,
with the following properties:
\begin{itemize}
\item $\mathcal{\tilde{P}}_{+}^{k}$ passes through $p$ and is orthogonal
to the tangent cylinder of $\left\{ \Sigma_{t}^{k}\right\} $ at $p$;
$\mathcal{\mathring{P}}_{+}^{k}\cap\mathcal{T}_{j}^{k}$ is close
to $\mathcal{\tilde{P}}_{+}^{k}\cap\mathcal{T}_{j}^{k}$ and is ``between''
$\mathcal{\tilde{P}}_{+}^{k}\cap\mathcal{T}_{j}^{k}$ and $\mathcal{P}_{+}\cap\mathcal{T}_{j}^{k}$.
\item There are no singular points of $\left\{ \Sigma_{t}^{k}\right\} $
on $\mathcal{\mathring{P}}_{+}^{k}\cap\mathcal{T}_{j}^{k}$ and 
\[
\bar{t}\coloneqq\,\max_{\mathcal{\mathring{P}}_{+}^{k}\cap\mathcal{T}_{j}^{k}}u^{k}\,\in\,\left(\tau_{j},\,u^{k}\left(p\right)\right)
\]
(see (\ref{one-sided isolated local max: level})). Note that $u^{k}\left(p\right)=T_{1,j}^{k}$.
\end{itemize}
In view of the the asymptotically cylindrical behavior of $\left\{ \Sigma_{t}^{k}\right\} $
on $\mathcal{\tilde{P}}_{+}^{k}\cap\mathcal{T}_{j}^{k}\setminus\left\{ p\right\} $,
we infer that 
\begin{equation}
\Sigma_{t}^{k}\cap\mathcal{\tilde{P}}_{+}^{k}\cap\mathcal{T}_{j}^{k}\neq\emptyset\quad\forall\,t\in\left[\tau_{j},\,T_{1,j}^{k}\right)\label{rule out bumpy type: lower intersection}
\end{equation}
and that
\begin{equation}
\frac{\nabla u^{k}}{\left|\nabla u^{k}\right|}\cdot\tilde{\nu}\approx0\quad\textrm{on}\,\,\,\mathcal{\tilde{P}}_{+}^{k}\cap\mathcal{T}_{j}^{k}\cap\left\{ \tau_{j}\leq u^{k}<T_{1,j}^{k}\right\} ,\label{rule out bumpy type: lower normal}
\end{equation}
where $\tilde{\nu}$ is a unit normal vector of $\mathcal{\tilde{P}}_{+}^{k}\cap\mathcal{T}_{j}^{k}$. 

Now define $\mathcal{\mathring{T}}_{j}^{k}$ as the tubular closed
region bounded by $\Sigma_{\tau_{j}}^{k}$ (from the lateral) and
the two caps $\mathcal{P}_{+}\cap\mathcal{T}_{j}^{k}$ and $\mathcal{\tilde{P}}_{+}^{k}\cap\mathcal{T}_{j}^{k}$
(from the two ends). Note that $\mathcal{\mathring{P}}_{+}^{k}\cap\mathcal{T}_{j}^{k}$
would cut $\mathcal{\mathring{T}}_{j}^{k}$ into two pieces. Let 
\[
t_{*}=\frac{1}{2}\left(\bar{t}+T_{1,j}^{k}\right)\,<\,T_{1,j}^{k},
\]
which is a regular value of $\left.u^{k}\right|_{\mathcal{\mathring{T}}_{j}^{k}}$
by Assumption \ref{local singular times hypothesis}. Note that by
(\ref{rule out bumpy type: upper intersection}) and (\ref{rule out bumpy type: lower intersection})
we have
\[
\bar{\Omega}_{t_{*}}^{k}\cap\mathcal{P}_{+}\cap\mathcal{T}_{j}^{k}\neq\emptyset,\quad\bar{\Omega}_{t_{*}}^{k}\cap\mathcal{\tilde{P}}_{+}^{k}\cap\mathcal{\mathring{T}}_{j}^{k}\neq\emptyset.
\]
Moreover, since $u^{k}\leq\bar{t}$ on $\mathcal{\mathring{P}}_{+}^{k}\cap\mathcal{T}_{j}^{k}$,
we have
\[
\bar{\Omega}_{t_{*}}^{k}\cap\mathcal{\mathring{P}}_{+}^{k}\cap\mathcal{\mathring{T}}_{j}^{k}=\emptyset.
\]
Thus, every connected component of $\bar{\Omega}_{t_{*}}^{k}\cap\mathcal{\mathring{T}}_{j}^{k}$
intersects at most one of the caps, i.e., $\mathcal{P}_{+}\cap\mathcal{T}_{j}^{k}$
and $\mathcal{\tilde{P}}_{+}^{k}\cap\mathcal{T}_{j}^{k}$. Applying
Lemma \ref{existence of singular time} to the flow $\left\{ \Sigma_{t}^{k}\right\} $
on the tube $\mathcal{\mathring{T}}_{j}^{k}$ (noting that we have
(\ref{rule out bumpy type: upper normal}) and (\ref{rule out bumpy type: lower normal})),
we deduce that $\Sigma_{\hat{t}}^{k}\cap\mathcal{\mathring{T}}_{j}^{k}$
contains a singular point of $\left\{ \Sigma_{t}^{k}\right\} $ for
some $\hat{t}\in\left(\tau_{j},t_{*}\right)$. In particular, $\hat{t}<T_{1,j}^{k}$
is another singular time of $\left\{ \Sigma_{t}^{k}\right\} $ in
$\mathcal{\mathring{T}}_{j}^{k}\subset\mathcal{S}_{j}^{\hat{\delta}}$.
This is in contradiction with Assumption \ref{local singular times hypothesis}.
Therefore, there are no singular components of $\left\{ \Sigma_{t}^{k}\right\} $
that belong to the bumpy type.
\end{proof}

\subsection{Bumpy type\label{bumpy type}}

In this subsection we shall prove the ``stability'' of singular
components of the bumpy type under Assumption \ref{existence of singularities near bumpy}.
Specifically, let $\mathcal{S}_{j}$ be a singular component of $\left\{ \Sigma_{t}\right\} $
that belongs to the bumpy type, when $k$ is large we assume that
$\left\{ \Sigma_{t}^{k}\right\} $ has singular components near $\mathcal{S}_{j}$
(see Assumption \ref{existence of singularities near bumpy}) and
show that all these singular components must be of the bumpy type
(see Proposition \ref{no saddle toward right}) and that there is
indeed only one singular component of $\left\{ \Sigma_{t}^{k}\right\} $
near $\mathcal{S}_{j}$ (see Proposition \ref{uniqueness of bumpy}).
Note that $u\left(\mathcal{S}_{j}\right)=T_{1}<T_{ext}$ by Proposition
\ref{global maximum points }. 

Let us begin with the following remark, which is the counterpart of
Remark \ref{tubular region for splitting} in Section \ref{splitting type}
for the bumpy case. 
\begin{rem}
\label{tubular region for bumpy}When $\mathcal{S}_{j}$ is a curve
of the bumpy type, say
\[
\mathcal{S}_{j}=\left\{ \Gamma\left(s\right)\,:\,s\in\left[-1,1\right]\right\} .
\]
Let us assume without loss of generality that $\Gamma\left(1\right)$
is the local maximum point of $u$ and $\Gamma\left(-1\right)$ is
the one-sided saddle point of $u$. By Remark \ref{tube-like hypersurface},
$\Sigma_{\tau_{j}}\cap\mathcal{S}_{j}^{r_{j}}$ is a tube-like hypersurface
around $\mathcal{S}_{j}$. Near the endpoints $\Gamma\left(1\right)$
and $\Gamma\left(-1\right)$, by Remark \ref{one-sided isolated local max}
and Remark \ref{flow near saddle}, we can find two hyperplanes, $\mathcal{P}_{+}$
and $\mathcal{P}_{-}$, that are orthogonal to the tangent cylinders
of $\left\{ \Sigma_{t}\right\} $ at $\Gamma\left(1\right)$ and $\Gamma\left(-1\right)$,
respectively, such that\footnote{$\mathcal{P}_{+}$ corresponds to $\left\{ z=z_{0}\right\} $ in Remark
\ref{one-sided isolated local max} with $z_{0}>0$ sufficiently close
to $0$; $\mathcal{P}_{-}$ corresponds to $\left\{ z=\mathring{\varepsilon}\right\} $
in Remark \ref{flow near saddle}.}
\[
\mathcal{P}_{+}\cap\mathcal{S}_{j}^{r_{j}}\cap\mathcal{S}_{j}=\emptyset,\quad\mathcal{P}_{-}\cap\mathcal{S}_{j}^{r_{j}}\cap\mathcal{S}_{j}=\emptyset
\]
and that the tube-like connected hypersurface $\Sigma_{\tau_{j}}$
together with $\mathcal{P}_{+}$ and $\mathcal{P}_{-}$ bound a tubular
closed region $\mathcal{T}_{j}$ in $\mathcal{S}_{j}^{r_{j}}$ with
the following properties:
\begin{enumerate}
\item $\mathcal{S}_{j}\subset\textrm{int}\,\mathcal{T}_{j}$. 
\item Let $\mathcal{\tilde{P}}_{+}$ and $\mathcal{\tilde{P}}_{-}$ be two
hypersurfaces parallel to $\mathcal{P}_{+}$ and $\mathcal{P}_{-}$,
\footnote{$\mathcal{\tilde{P}}_{+}$ corresponds to $\left\{ z=0\right\} $
in Remark \ref{one-sided isolated local max}; $\mathcal{\tilde{P}}_{-}$
corresponds to $\left\{ z=0\right\} $ in Remark \ref{flow near saddle}.}respectively, such that 
\[
\mathcal{\tilde{P}}_{+}\cap\mathcal{T}_{j}\cap\mathcal{S}_{j}=\left\{ \Gamma\left(1\right)\right\} ,\quad\mathcal{\tilde{P}}_{-}\cap\mathcal{T}_{j}\cap\mathcal{S}_{j}=\left\{ \Gamma\left(-1\right)\right\} .
\]
\item On the cap $\mathcal{P}_{+}\cap\mathcal{T}_{j}$ we have\footnote{See (\ref{one-sided isolated local max: level}). }
\[
\bar{t}_{j}\coloneqq\max_{\mathcal{P}_{+}\cap\mathcal{T}_{j}}u\,\in\,\left(\tau_{j},\,T_{1}\right).
\]
\item There exists a time $\mathring{t}_{j}>T_{1}$ so that\footnote{See (\ref{flow near saddle: intersection}).}
\[
\Sigma_{t}\cap\mathcal{P}_{-}\cap\mathcal{T}_{j}\neq\emptyset\quad\forall\,t\in\left[\tau_{j},\,\mathring{t}_{j}\right].
\]
\end{enumerate}
In the case where $\mathcal{S}_{j}$ is a single one-sided saddle
point (which can be regarded as a ``degenerate'' curve with a local
maximum endpoint and a one-sided saddle endpoint, see the comment
following Definition \ref{singularity types}), the aforementioned
results still hold. In this case, the two hyperplanes, $\mathcal{P}_{+}$
and $\mathcal{P}_{-}$, are parallel (since both are orthogonal to
the tangent cylinder of $\left\{ \Sigma_{t}\right\} $ at $\mathcal{S}_{j}$)
and close to each other; $\mathcal{P}_{+}$ is assumed to be on the
side in which $u$ is a local maximum point (see Definition \ref{one-sided saddle}).
Note that $\mathcal{\tilde{P}}_{+}=\mathcal{\tilde{P}}_{-}$ is the
hyperplane that passes through $\mathcal{S}_{j}$ and orthogonal to
the tangent cylinder of $\left\{ \Sigma_{t}\right\} $ at $\mathcal{S}_{j}$. 

Furthermore, by Proposition \ref{C0 compactness}, Proposition \ref{pre-stability},
and Remark \ref{tube-like hypersurface}, when $k$ is large, $\Sigma_{\tau_{j}}^{k}$
and the two hyperplanes $\mathcal{P}_{+}$ and $\mathcal{P}_{-}$
would also bound a closed tubular region $\mathcal{T}_{j}^{k}$ in
$\mathcal{S}_{j}^{r_{j}}$; in addition, let 
\[
\bar{t}_{j}^{k}=\max_{\mathcal{P}_{+}\cap\mathcal{T}_{j}^{k}}u^{k}
\]
and $T_{1,j}^{k}$ be the unique singular time of $\left\{ \Sigma_{t}^{k}\right\} $
in $\mathcal{S}_{j}^{\hat{\delta}}$, then we have
\[
\left|\bar{t}_{j}^{k}-\bar{t}_{j}\right|\,+\,\left|T_{1,j}^{k}-T_{1}\right|\,\leq\,\frac{1}{3}\min\left\{ T_{1}-\bar{t}_{j},\,\bar{t}_{j}-\tau_{j}\right\} .
\]
It follows that 
\begin{equation}
\bar{t}_{j}^{k}\,\in\,\left(\tau_{j},\,T_{1,j}^{k}\right).\label{tubular region for bumpy: barrier}
\end{equation}
\end{rem}

Unlike the vanishing and splitting cases, we assume (instead of proving)
the existence of singularities of $\left\{ \Sigma_{t}^{k}\right\} $
near a bumpy component of $\left\{ \Sigma_{t}\right\} $ as follows.
\begin{assumption}
\label{existence of singularities near bumpy}In case $\mathcal{S}_{j}$
is a singular component of $\left\{ \Sigma_{t}\right\} $ that belongs
to the bumpy type, we assume that for every sufficiently large $k$,
the flow $\left\{ \Sigma_{t}^{k}\right\} $ has singularities in $\mathcal{S}_{j}^{\hat{\delta}}$. 
\end{assumption}

Next, for the singular components of $\left\{ \Sigma_{t}^{k}\right\} $
in $\mathcal{S}_{j}^{\hat{\delta}}$, we are going to characterize
their types in Proposition \ref{no saddle toward right} and prove
the uniqueness of the singular components in Proposition \ref{uniqueness of bumpy}.
To streamline the proofs, in the following remark we set up the requisite
notations and provide a preliminary description of the singular components
of $\left\{ \Sigma_{t}^{k}\right\} $ in $\mathcal{S}_{j}^{\hat{\delta}}$. 
\begin{rem}
\label{singularities near bumpy type}Let $\mathcal{S}_{j}$ be a
singular component of $\left\{ \Sigma_{t}\right\} $ that belongs
to the bumpy type. When $k$ is large, by Corollary \ref{regular space}
and Remark \ref{tubular region for bumpy}, all the singular points
of $\left\{ \Sigma_{t}^{k}\right\} $ in $\mathcal{S}_{j}^{\hat{\delta}}$
would be contained in $\textrm{int}\,\mathcal{T}_{j}^{k}$. Moreover,
by Assumption \ref{local singular times hypothesis} and Proposition
\ref{global maximum points } (see also Proposition \ref{pre-stability}),
there are finitely many connected components of the singular set of
$\left\{ \Sigma_{t}^{k}\right\} $ in $\mathcal{T}_{j}^{k}$, each
of which belongs to the splitting/bumpy type (in particular, there
are no round points of $\left\{ \Sigma_{t}^{k}\right\} $ in $\mathcal{T}_{j}^{k}$). 

Let $\mathcal{\dot{S}}_{j}^{k}$ be a singular component of $\left\{ \Sigma_{t}^{k}\right\} $
in $\mathcal{T}_{j}^{k}$. By Proposition \ref{uniform cylindrical scale},
any point $q\in\mathcal{\dot{S}}_{j}^{k}$ has a cylindrical scale
$r_{j}$ and every other cylindrical point of $\left\{ \Sigma_{t}^{k}\right\} $
in $B_{r_{j}}\left(q\right)$ is located in a small Lipschitz graph
over the axis of the tangent cylinder of $\left\{ \Sigma_{t}^{k}\right\} $
at $q$. Applying Remark \ref{tubular region for splitting} (if $\mathcal{\dot{S}}_{j}^{k}$
is of the splitting type) / Remark \ref{tubular region for bumpy}
(if $\mathcal{\dot{S}}_{j}^{k}$ is of the bumpy type) to the singular
component $\mathcal{\dot{S}}_{j}^{k}$ of $\left\{ \Sigma_{t}^{k}\right\} $,
we can find two hyperplanes $\mathcal{\dot{P}}_{+}^{k}$ and $\mathcal{\dot{P}}_{-}^{k}$
near the endpoints of $\mathcal{\dot{S}}_{j}^{k}$ that are orthogonal
to the tangent cylinders of $\left\{ \Sigma_{t}^{k}\right\} $ at
the endpoints, respectively, such that
\begin{enumerate}
\item there are no singular points of $\left\{ \Sigma_{t}^{k}\right\} $
on $\mathcal{\dot{P}}_{+}^{k}\cap\mathcal{T}_{j}^{k}$ and $\mathcal{\dot{P}}_{-}^{k}\cap\mathcal{T}_{j}^{k}$
(see Proposition \ref{uniform cylindrical scale});
\item $\Sigma_{\tau_{j}}^{k}$ is a tube-like hypersurface around $\mathcal{\dot{S}}_{j}^{k}$
and it together with the two caps, $\mathcal{\dot{P}}_{+}^{k}\cap\mathcal{T}_{j}^{k}$
and $\mathcal{\dot{P}}_{-}^{k}\cap\mathcal{T}_{j}^{k}$, bound a ``truncated''
tubular closed region $\mathcal{\dot{T}}_{j}^{k}\subset\mathcal{T}_{j}^{k}$
satisfying $\mathcal{\dot{S}}_{j}^{k}\subset\textrm{int}\,\mathcal{\dot{T}}_{j}^{k}$;
\item $\mathcal{\dot{S}}_{j}^{k}$ is the only singular component of $\left\{ \Sigma_{t}^{k}\right\} $
in $\mathcal{\dot{T}}_{j}^{k}$ (see Proposition \ref{uniform cylindrical scale}). 
\end{enumerate}
For convenience, let us assume that the ``orientation'' is chosen
in such a way that $\mathcal{\dot{P}}_{+}^{k}$ is adjacent to $\mathcal{P}_{+}$
and $\mathcal{\dot{P}}_{-}^{k}$ is adjacent to $\mathcal{P}_{-}$.
Thus, the tube $\mathcal{T}_{j}^{k}$ would be cut by $\mathcal{\dot{P}}_{+}^{k}$
and $\mathcal{\dot{P}}_{-}^{k}$ into three closed pieces: 
\[
\mathcal{T}_{j}^{k}\,=\,\mathcal{\hat{T}}_{j}^{k+}\cup\mathcal{\dot{T}}_{j}^{k}\cup\mathcal{\hat{T}}_{j}^{k-},
\]
where $\mathcal{\hat{T}}_{j}^{k+}$ is the tubular closed region bounded
by $\Sigma_{\tau_{j}}^{k}$ (from the lateral) and the two caps $\mathcal{P}_{+}\cap\mathcal{T}_{j}^{k}$
and $\mathcal{\dot{P}}_{+}^{k}\cap\mathcal{T}_{j}^{k}$ (from the
two ends); likewise, $\mathcal{\hat{T}}_{j}^{k-}$ is the closed region
bounded by $\Sigma_{\tau_{j}}^{k}$, $\mathcal{P}_{-}\cap\mathcal{T}_{j}^{k}$,
and $\mathcal{\dot{P}}_{-}^{k}\cap\mathcal{T}_{j}^{k}$.
\end{rem}

\begin{prop}
\label{no saddle toward right}Let $\mathcal{S}_{j}$ be a singular
component of $\left\{ \Sigma_{t}\right\} $ that belongs to the bumpy
type and let $\mathcal{\dot{S}}_{j}^{k}$ be any singular component
of $\left\{ \Sigma_{t}^{k}\right\} $ in $\mathcal{S}_{j}^{\hat{\delta}}$
as stated in Remark \ref{singularities near bumpy type}. Then 
\begin{itemize}
\item if the singular component $\mathcal{\dot{S}}_{j}^{k}$ is a curve,
the endpoint that is adjacent to $\mathcal{P}_{+}$ must be a local
maximum point of $u^{k}$; 
\item if $\mathcal{\dot{S}}_{j}^{k}$ is a single point, $\mathcal{\dot{S}}_{j}^{k}$
must be a local maximum point of $u^{k}$ ``on the side toward $\mathcal{P}_{+}$''
(see Definition \ref{one-sided saddle}).
\end{itemize}
In either case, $\mathcal{\dot{S}}_{j}^{k}$ is of the bumpy type. 
\end{prop}

\begin{proof}
Suppose the contrary that 
\begin{itemize}
\item when $\mathcal{S}_{j}$ is a curve: the endpoint $p$ that is adjacent
to $\mathcal{P}_{+}$ is a one-sided saddle point of $u^{k}$; 
\item when $\mathcal{\dot{S}}_{j}^{k}$ is a single point: $p=\mathcal{\dot{S}}_{j}^{k}$
is not a local maximum point of $u^{k}$ on the side toward $\mathcal{P}_{+}$.
\end{itemize}
Recall that the choice of the two hyperplanes $\mathcal{\dot{P}}_{+}^{k}$
and $\mathcal{\dot{P}}_{-}^{k}$ in Remark \ref{singularities near bumpy type}
is based on applying Remark \ref{tubular region for splitting}/Remark
\ref{tubular region for bumpy} to the singular component $\mathcal{\dot{S}}_{j}^{k}$
of $\left\{ \Sigma_{t}^{k}\right\} $. By either of the remarks, there
exists a time $\mathring{t}_{j}^{k}>T_{1,j}^{k}$ so that 
\begin{equation}
\Sigma_{t}^{k}\cap\mathcal{\dot{P}}_{+}^{k}\cap\mathcal{T}_{j}^{k}\,\neq\,\emptyset\quad\forall\,t\in\left[\tau_{j},\,\mathring{t}_{j}^{k}\right].\label{no saddle toward right: nonempty}
\end{equation}
On the other hand, recall that by Remark \ref{tubular region for bumpy}
we have 
\begin{equation}
\bar{t}_{j}^{k}=\max_{\mathcal{P}_{+}\cap\mathcal{T}_{j}^{k}}u^{k}\,\in\,\left(\tau_{j},\,T_{1,j}^{k}\right).\label{no saddle toward right: right barrier}
\end{equation}
Let $\mathcal{\tilde{P}}_{+}^{k}$ be the hyperplane passing through
$p$ and orthogonal to the tangent cylinder of $\left\{ \Sigma_{t}^{k}\right\} $
at $p$. Note that $\mathcal{\dot{P}}_{+}^{k}\cap\mathcal{T}_{j}^{k}$
is parallel (and close) to $\mathcal{\tilde{P}}_{+}^{k}\cap\mathcal{T}_{j}^{k}$
and is ``between'' $\mathcal{P}_{+}\cap\mathcal{T}_{j}^{k}$ and
$\mathcal{\tilde{P}}_{+}^{k}\cap\mathcal{T}_{j}^{k}$. In view of
the the asymptotically cylindrical behavior of $\left\{ \Sigma_{t}^{k}\right\} $
on $\mathcal{\tilde{P}}_{+}^{k}\cap\mathcal{T}_{j}^{k}\setminus\left\{ p\right\} $,
we have (see (\ref{saddle criterion: max}))
\begin{equation}
\max_{\mathcal{\tilde{P}}_{+}^{k}\cap\mathcal{T}_{j}^{k}}u^{k}=T_{1,j}^{k}.\label{no saddle toward right: left barrier}
\end{equation}
Now define $\mathcal{\check{T}}_{j}^{k}$ as the tubular closed region
bounded by $\Sigma_{\tau_{j}}^{k}$ (from the lateral) and the two
caps $\mathcal{P}_{+}\cap\mathcal{T}_{j}^{k}$ and $\mathcal{\tilde{P}}_{+}^{k}\cap\mathcal{T}_{j}^{k}$
(from the two ends). Notice that by (\ref{no saddle toward right: nonempty})
we have
\[
\Sigma_{\mathring{t}_{j}^{k}}^{k}\cap\mathcal{\check{T}}_{j}^{k}\,\neq\,\emptyset
\]
and that by (\ref{no saddle toward right: right barrier}), (\ref{no saddle toward right: left barrier}),
and that $u^{k}=\tau_{j}$ on the lateral, we have 
\[
\Sigma_{\mathring{t}_{j}^{k}}^{k}\cap\partial\mathcal{\check{T}}_{j}^{k}\,=\,\emptyset.
\]
As $\mathring{t}_{j}^{k}$ is a regular value\footnote{$T_{1,j}^{k}$ is the only critical value of $\left.u^{k}\right|_{\mathcal{\check{T}}_{j}^{k}}$
by Assumption \ref{local singular times hypothesis}.} of $\left.u^{k}\right|_{\mathcal{\check{T}}_{j}^{k}}$, $\Sigma_{\mathring{t}_{j}^{k}}\cap\mathcal{\check{T}}_{j}^{k}$
is a closed hypersurface, which encloses an open region $\check{\Omega}_{\mathring{t}_{j}^{k}}^{k}\subset\mathcal{\check{T}}_{j}^{k}$.
It follows that 
\[
\tilde{T}_{1,j}^{k}\coloneqq\max_{\overline{\check{\Omega}_{\mathring{t}_{j}^{k}}^{k}}}\,u^{k}\,\geq\,\mathring{t}_{j}^{k}\,>\,T_{1,j}^{k}
\]
would be attained at some interior point of $\mathcal{\check{T}}_{j}^{k}$,
which would be a critical point of $u^{k}$. Consequently, $\tilde{T}_{1,j}^{k}$
is another singular time of $\left\{ \Sigma_{t}^{k}\right\} $ in
$\mathcal{\check{T}}_{j}^{k}\subset\mathcal{S}_{j}^{\hat{\delta}}$,
contradicting Assumption \ref{local singular times hypothesis}. Therefore,
when $\mathcal{\dot{S}}_{j}^{k}$ is a curve, $p$ must be a local
maximum point of $u^{k}$; when $\mathcal{\dot{S}}_{j}^{k}$ is a
single point, $p$ is a local maximum point of $u^{k}$ on the side
toward $\mathcal{P}_{+}$.
\end{proof}
\begin{prop}
\label{uniqueness of bumpy}In Proposition \ref{no saddle toward right},
there is actually only one singular component of $\left\{ \Sigma_{t}^{k}\right\} $
in $\mathcal{T}_{j}^{k}$. Therefore, $\left\{ \Sigma_{t}^{k}\right\} $
has precisely one singular component in $\mathcal{S}_{j}^{\hat{\delta}}$,
which is of the bumpy type. 
\end{prop}

\begin{proof}
Let $\mathcal{\dot{S}}_{j}^{k}$ be a singular component of $\left\{ \Sigma_{t}^{k}\right\} $
in $\mathcal{T}_{j}^{k}$ as stated in Proposition \ref{no saddle toward right}.
Suppose the contrary that there exists another singular component
$\mathcal{\ddot{S}}_{j}^{k}$ of $\left\{ \Sigma_{t}^{k}\right\} $
in $\mathcal{T}_{j}^{k}$. Since $\mathcal{\ddot{S}}_{j}^{k}\cap\mathcal{\dot{T}}_{j}^{k}=\emptyset$,
we may assume without loss of generality that $\mathcal{\ddot{S}}_{j}^{k}\subset\mathcal{\hat{T}}_{j}^{k+}$
(see Remark \ref{singularities near bumpy type}). 

Let $q$ be the endpoint of $\mathcal{\ddot{S}}_{j}^{k}$ that is
adjacent to $\mathcal{\dot{P}}_{+}^{k}\cap\mathcal{T}_{j}^{k}$; in
case $\mathcal{\ddot{S}}_{j}^{k}$ is a single one-sided saddle point,
then $q$ is $\mathcal{\ddot{S}}_{j}^{k}$ itself. By Proposition
\ref{no saddle toward right}, $q$ is a one-sided saddle point of
$u^{k}$ such that it is not a local maximum point of $u^{k}$ on
the side toward $\mathcal{\dot{P}}_{+}^{k}\cap\mathcal{T}_{j}^{k}$. 

By Remark \ref{tubular region for bumpy} and Remark \ref{singularities near bumpy type},
we can find two hyperplanes $\mathcal{\ddot{P}}_{+}^{k}$ and $\mathcal{\ddot{P}}_{-}^{k}$
near the endpoints of $\mathcal{\ddot{S}}_{j}^{k}$ that are orthogonal
to the tangent cylinders of $\left\{ \Sigma_{t}^{k}\right\} $ at
the endpoints, respectively, such that the following hold:
\begin{enumerate}
\item $\mathcal{\ddot{P}}_{+}^{k}\cap\mathcal{T}_{j}^{k}$ and $\mathcal{\ddot{P}}_{-}^{k}\cap\mathcal{T}_{j}^{k}$
are between $\mathcal{P}_{+}\cap\mathcal{T}_{j}^{k}$ and $\mathcal{\dot{P}}_{+}^{k}\cap\mathcal{T}_{j}^{k}$;
the orientation is chosen in such a way that $\mathcal{\ddot{P}}_{+}^{k}\cap\mathcal{T}_{j}^{k}$
is adjacent to $\mathcal{P}_{+}\cap\mathcal{T}_{j}^{k}$ and $\mathcal{\ddot{P}}_{-}^{k}\cap\mathcal{T}_{j}^{k}$
is adjacent to $\mathcal{\dot{P}}_{+}^{k}\cap\mathcal{T}_{j}^{k}$.
\item There are no singular points of $\left\{ \Sigma_{t}^{k}\right\} $
on $\mathcal{\ddot{P}}_{+}^{k}\cap\mathcal{T}_{j}^{k}$ and $\mathcal{\ddot{P}}_{-}^{k}\cap\mathcal{T}_{j}^{k}$.
\item There exists a time $\ddot{t}_{j}^{k}>T_{1,j}^{k}$ so that 
\begin{equation}
\Sigma_{t}^{k}\cap\mathcal{\ddot{P}}_{-}^{k}\cap\mathcal{T}_{j}^{k}\,\neq\,\emptyset\quad\forall\,t\in\left[\tau_{j},\,\ddot{t}_{j}^{k}\right].\label{uniqueness of bumpy: nonempty}
\end{equation}
\end{enumerate}
Let $\mathcal{\widetilde{P}}_{-}^{k}$ be the hyperplane passing through
$q$ and orthogonal to the tangent cylinder of $\left\{ \Sigma_{t}^{k}\right\} $
at $q$. Note that $\mathcal{\widetilde{P}}_{-}^{k}$ is parallel
and close to $\mathcal{\ddot{P}}_{-}^{k}$ and that $\mathcal{\widetilde{P}}_{-}^{k}\cap\mathcal{T}_{j}^{k}$
is between $\mathcal{\ddot{P}}_{+}^{k}\cap\mathcal{T}_{j}^{k}$ and
$\mathcal{\ddot{P}}_{-}^{k}\cap\mathcal{T}_{j}^{k}$. In view of the
the asymptotically cylindrical behavior of $\left\{ \Sigma_{t}^{k}\right\} $
on $\mathcal{\widetilde{P}}_{-}^{k}\cap\mathcal{T}_{j}^{k}\setminus\left\{ q\right\} $,
we have (see (\ref{saddle criterion: max})) 
\begin{equation}
\max_{\mathcal{\widetilde{P}}_{-}^{k}\cap\mathcal{T}_{j}^{k}}u^{k}=T_{1,j}^{k}.\label{uniqueness of bumpy: right barrier}
\end{equation}
Let $p$ be the endpoint $\mathcal{\dot{S}}_{j}^{k}$ that is adjacent
to $\mathcal{\dot{P}}_{+}^{k}\cap\mathcal{T}_{j}^{k}$. Let $\mathcal{\tilde{P}}_{+}^{k}$
be the hyperplane passing through $p$ and orthogonal to the tangent
cylinder of $\left\{ \Sigma_{t}^{k}\right\} $ at $p$. Note that
$\mathcal{\tilde{P}}_{+}^{k}$ is parallel and close to $\mathcal{\dot{P}}_{+}^{k}$.
In view of the the asymptotically cylindrical behavior of $\left\{ \Sigma_{t}^{k}\right\} $
on $\mathcal{\tilde{P}}_{+}^{k}\cap\mathcal{T}_{j}^{k}\setminus\left\{ p\right\} $,
we have (see (\ref{saddle criterion: max})) 
\begin{equation}
\max_{\mathcal{\tilde{P}}_{+}^{k}\cap\mathcal{T}_{j}^{k}}u^{k}=T_{1,j}^{k}.\label{uniqueness of bumpy: left barrier}
\end{equation}
Now define $\mathcal{\widehat{T}}_{j}^{k}$ as the tubular closed
region bounded by $\Sigma_{\tau_{j}}^{k}$ (from the lateral) and
the two caps $\mathcal{\tilde{P}}_{+}^{k}\cap\mathcal{T}_{j}^{k}$
and $\mathcal{\widetilde{P}}_{-}^{k}\cap\mathcal{T}_{j}^{k}$ (from
the two ends). Note that $\mathcal{\ddot{P}}_{-}^{k}\cap\mathcal{T}_{j}^{k}$
is between $\mathcal{\tilde{P}}_{+}^{k}\cap\mathcal{T}_{j}^{k}$ and
$\mathcal{\widetilde{P}}_{-}^{k}\cap\mathcal{T}_{j}^{k}$. It follows
from (\ref{uniqueness of bumpy: nonempty}) that
\[
\Sigma_{\ddot{t}_{j}^{k}}^{k}\cap\mathcal{\widehat{T}}_{j}^{k}\,\neq\,\emptyset.
\]
Note that 
\[
\Sigma_{\ddot{t}_{j}^{k}}^{k}\cap\partial\mathcal{\widehat{T}}_{j}^{k}\,=\,\emptyset
\]
by (\ref{uniqueness of bumpy: right barrier}), (\ref{uniqueness of bumpy: left barrier}),
and that $u^{k}=\tau_{j}$ on the lateral. Since $\ddot{t}_{j}^{k}$
is a regular value\footnote{$T_{1,j}^{k}$ is the only critical value of $\left.u^{k}\right|_{\mathcal{\widehat{T}}_{j}^{k}}$.}
of $\left.u^{k}\right|_{\mathcal{\widehat{T}}_{j}^{k}}$, $\Sigma_{\ddot{t}_{j}^{k}}\cap\mathcal{\widehat{T}}_{j}^{k}$
is a closed hypersurface, which encloses an open region $\widehat{\Omega}_{\ddot{t}_{j}^{k}}^{k}$
in $\mathcal{\widehat{T}}_{j}^{k}$. Thus, 
\[
\widetilde{T}_{1,j}^{k}\coloneqq\max_{\overline{\widehat{\Omega}_{\ddot{t}_{j}^{k}}^{k}}}\,u^{k}\,\geq\,\ddot{t}_{j}^{k}\,>\,T_{1,j}^{k}
\]
would be attained at some interior point of $\mathcal{\widehat{T}}_{j}^{k}$,
which would be a critical point of $u^{k}$. Consequently, $\widetilde{T}_{1,j}^{k}$
is another singular time of $\left\{ \Sigma_{t}^{k}\right\} $ in
$\mathcal{\widehat{T}}_{j}^{k}\subset\mathcal{S}_{j}^{\hat{\delta}}$,
contradicting Assumption \ref{local singular times hypothesis}. Thus,
$\mathcal{\dot{S}}_{j}^{k}$ is the only one singular component of
$\left\{ \Sigma_{t}^{k}\right\} $ in $\mathcal{T}_{j}^{k}$.
\end{proof}

\subsection{Appendix: Two-convexity\label{appendix:two-convexity}}

The purpose of this appendix is to show that in Theorem \ref{stability of singular types},
as $\Sigma_{0}$ is two-convex, $\Sigma_{0}^{k}$ would also be two-convex
when $k$ is large (see Proposition \ref{two-convexity of approximations}).
To this end, let us begin with the following lemma, which is a direct
result of the min-max theorem in linear algebra. 
\begin{lem}
\label{min-max}Let $V$ be a finite dimensional vector space with
an inner product $g$. Let $S$ be a self-adjoint operator on the
inner product space $\left(V,g\right)$ with eigenvalues $\lambda_{1}\leq\lambda_{2}\leq\cdots$.
Let $A$ be the associated symmetric bilinear form of $S$, namely,
\[
A\left(v,w\right)=g\left(Sv,w\right)\quad\forall\,v,w\in V.
\]
Given a pair of linearly independent vectors $v_{1}$ and $v_{2}$
in $V$, let $\boldsymbol{G}\left(v_{1},v_{2}\right)$ and $\boldsymbol{A}\left(v_{1},v_{2}\right)$
be the $2\times2$ matrices whose $\left(k,l\right)$-components are
given by respectively
\[
\boldsymbol{G}_{kl}\left(v_{1},v_{2}\right)=g\left(v_{k},v_{l}\right),\quad\boldsymbol{A}_{kl}\left(v_{1},v_{2}\right)=A\left(v_{k},v_{l}\right),\quad\textrm{for}\,\,\,k,l\in\left\{ 1,2\right\} .
\]
Then 
\begin{enumerate}
\item the symmetric matrix $\boldsymbol{G}\left(v_{1},v_{2}\right)$ is
positive definite;
\item whenever $\textrm{span}\,\left\{ \tilde{v}_{1},\tilde{v}_{2}\right\} =\textrm{span}\,\left\{ v_{1},v_{2}\right\} $,
we have 
\[
\textrm{trace}\,\left(\boldsymbol{G}^{-1}\boldsymbol{A}\right)\left(\tilde{v}_{1},\tilde{v}_{2}\right)=\textrm{trace}\,\left(\boldsymbol{G}^{-1}\boldsymbol{A}\right)\left(v_{1},v_{2}\right).
\]
\end{enumerate}
Thus, $\boldsymbol{G}^{-1}\boldsymbol{A}$ can be regarded as a function
defined on the Grassmannian $Gr_{2}\left(V\right)$, i.e., the set
of all $2$-dimensional vector subspaces of $V$. Most importantly,
\[
\lambda_{1}+\lambda_{2}=\min_{Gr_{2}\left(V\right)}\,\,\textrm{trace}\,\left(\boldsymbol{G}^{-1}\boldsymbol{A}\right).
\]
\end{lem}

The following corollary follows from applying Lemma \ref{min-max}
to the Weingarten map of a hypersurface $\Sigma$ on the tangent space
at each point so as to obtain an expression of $\kappa_{1}+\kappa_{2}$,
which will be used in Proposition \ref{two-convexity of approximations}.
Note that as the function $\boldsymbol{G}^{-1}\boldsymbol{A}\left(v_{1},v_{2}\right)$
in Lemma \ref{min-max} is defined on the Grassmannian, we may assume
that the two linearly independent tangent vectors, 
\[
v_{1}\,=\,v_{1}^{1}\,\partial_{1}+\cdots+v_{1}^{n-1}\,\partial_{n-1}\,\simeq\,\left(v_{1}^{1},\cdots,v_{1}^{n-1}\right),
\]
\[
v_{2}\,=\,v_{2}^{1}\,\partial_{1}+\cdots+v_{2}^{n-1}\,\partial_{n-1}\,\simeq\,\left(v_{2}^{1},\cdots,v_{2}^{n-1}\right),
\]
are orthonormal with respect to the dot product in order to have the
compactness.
\begin{cor}
\label{sum of first two principal curvatures}Let $\Sigma$ be a hypersurface
and let $x=x\left(\xi\right)$ be a local parametrization of $\Sigma$.
Let $g_{ij}\left(\xi\right)$ and $A_{ij}\left(\xi\right)$ be the
components of the metric and the second fundamental form of $\Sigma$
with respect to the local parametrization $x\left(\xi\right)$, respectively.

Let $\mathbb{V}\subset\mathbb{R}^{n-1}\times\mathbb{R}^{n-1}$ be
the set of all pairs of orthonormal (with respect to the dot product)
vectors in $\mathbb{R}^{n-1}$. For each $\xi$, define two matrix-valued
functions $\boldsymbol{G}\left(\xi;\cdotp\right)$ and $\boldsymbol{A}\left(\xi;\cdot\right)$
on $\mathbb{V}$ as follows: Given 
\[
\left\{ v_{1}=\left(v_{1}^{1},\cdots,v_{1}^{n-1}\right),\quad v_{2}=\left(v_{2}^{1},\cdots,v_{2}^{n-1}\right)\right\} \,\in\,\mathbb{V},
\]
let $\boldsymbol{G}\left(\xi;v_{1},v_{2}\right)$ and $\boldsymbol{A}\left(\xi;v_{1},v_{2}\right)$
be the $2\times2$ matrices whose $\left(k,l\right)$-components are
given by respectively
\[
\boldsymbol{G}_{kl}\left(\xi;v_{1},v_{2}\right)=g_{ij}\left(\xi\right)v_{k}^{i}v_{l}^{j},\quad\boldsymbol{A}{}_{kl}\left(\xi;v_{1},v_{2}\right)=A_{ij}\left(\xi\right)v_{k}^{i}v_{l}^{j},\quad\textrm{for}\,\,\,k,l\in\left\{ 1,2\right\} .
\]
Then we have
\begin{equation}
\kappa_{1}\left(\xi\right)+\kappa_{2}\left(\xi\right)=\min_{\mathbb{V}}\,\,\textrm{trace}\,\left(\boldsymbol{G}^{-1}\boldsymbol{A}\right)\left(\xi;\cdot\right),\label{sum of first two principal curvatures: expression}
\end{equation}
where $\kappa_{1}\left(\xi\right)$ and $\kappa_{2}\left(\xi\right)$
are the smallest two principal curvatures of $\Sigma$ at $x\left(\xi\right)$.
Also, note that by the compactness of $\mathbb{V}$ and the continuity
of $\boldsymbol{G}^{-1}\boldsymbol{A}$, $\kappa_{1}+\kappa_{2}$
is a continuous function.
\end{cor}

Since $\Sigma_{0}$ is a two-convex closed hypersurface, the function
$\kappa_{1}+\kappa_{2}$ is bounded below by a positive constant on
$\Sigma_{0}$. By virtue of (\ref{sum of first two principal curvatures: expression}),
we get the following result.
\begin{prop}
\label{two-convexity of approximations}Let $\beta$ be a positive
constant such that 
\[
\min_{\Sigma_{0}}\,\left(\kappa_{1}+\kappa_{2}\right)\,>\,\beta\,\max_{\Sigma_{0}}\,H.
\]
Since 
\[
\Sigma_{0}^{k}\,\stackrel{C^{4}}{\rightarrow}\,\Sigma_{0}\quad\textrm{as}\,\,\,k\rightarrow\infty,
\]
when $k$ is large, $\Sigma_{0}^{k}$ would be $\beta$-uniformly
convex in the sense that 
\begin{equation}
\kappa_{1}+\kappa_{2}\,\geq\,\beta H\,>\,0.\label{uniformly two-convex}
\end{equation}
\end{prop}

It then follows from \cite{CHN} that the LSF $\left\{ \Sigma_{t}^{k}\right\} $
is $\beta$-uniformly two-convex in the sense that (\ref{uniformly two-convex})
holds at every regular point (with respect to the unit normal vector
field $\frac{\nabla u^{k}}{\left|\nabla u^{k}\right|}$). 

\bigskip{}
\bigskip{}
Department of Mathematics, National Taiwan University, Taipei 106,
Taiwan. 

\smallskip{}
\smallskip{}
\smallskip{}
\textit{E-mail:} \texttt{shguo@ntu.edu.tw}
\end{document}